\documentclass[oneside,reqno,11pt,a4paper]{amsart}
\usepackage[T1]{fontenc}
\usepackage[margin=2.cm]{geometry}
\usepackage{amsmath,amsfonts,amsbsy,amssymb,amscd,amsthm,bm}
\usepackage[usenames,dvipsnames,svgnames]{xcolor}
\usepackage[utf8]{inputenc}
\usepackage[australian]{babel}
\usepackage[ruled]{algorithm2e}
\usepackage{graphicx}
\usepackage{svg}
\usepackage{caption}
\usepackage{subcaption}
\usepackage{doi}
\usepackage[foot]{amsaddr}
\usepackage{makecell}
\usepackage{hyperref}
\usepackage{acronym}
\usepackage{listings}
\usepackage{textgreek}
\usepackage{url}
\usepackage{color}
\usepackage{framed}
\usepackage{verbatim}
\usepackage{fancyvrb}
\usepackage{stmaryrd}
\usepackage{newtxtext}
\usepackage{newtxmath}
\usepackage{microtype} 
\usepackage[final]{pdfpages}
\usepackage{tikz}
\usetikzlibrary{arrows.meta}
\usetikzlibrary{patterns.meta}
\usetikzlibrary{patterns}
\usetikzlibrary{calc}
\usepackage{booktabs}
\usepackage{multirow}
\newcolumntype{C}[1]{>{\centering\arraybackslash}p{#1}}
\newcolumntype{M}[1]{>{\centering\arraybackslash}m{#1}}

\usepackage[backend=biber,
sorting=nty,
defernumbers=false,
maxbibnames=5,
style=numeric-comp,
isbn=false,
bibencoding=utf8,
safeinputenc,
url=false,
doi=false,
eprint=false,
giveninits=true]{biblatex}
\hypersetup{
breaklinks=true,
bookmarksopen=true,
pdftitle={Four-field mixed finite elements for incompressible nonlinear elasticity},    
pdfauthor={Santiago Badia, Wei Li and Ricardo Ruiz-Baier},     
colorlinks=true,       
linkcolor=blue,          
citecolor=blue,        
filecolor=black,      
urlcolor=blue           
}

\acrodef{pde}[PDE]{partial differential equation}
\acrodef{fe}[FE]{finite element}
\acrodef{fem}[FEM]{finite element method}
\acrodef{csfem}[CSFEM]{compatible-strain mixed finite element method}
\acrodef{ddfem}[DDFEM]{discontinuous displacement mixed finite element method}
\acrodef{bdm}[BDM]{Brezzi--Douglas--Marini}
\acrodef{rt}[RT]{Raviart--Thomas}
\acrodef{lbb}[LBB]{Ladyzhenskaya--Babu\v{s}ka--Brezzi}
\acrodef{stp}[STP]{Simo--Taylor--Pister}
\acrodef{ndtns}[NDTNS]{normal displacement tangential-normal stress continuous}
\acrodef{dof}[DoF]{degree of freedom}
\acrodefplural{dof}[DoFs]{degrees of freedom}

\graphicspath{{./}{./images}}
\newcommand{\fig}[1]{Fig.~\ref{#1}}
\newcommand{\tab}[1]{Tab.~\ref{#1}}
\newcommand{\sect}[1]{Sect.~\ref{#1}}

\newcommand{\norm}[1]{\left\lVert #1 \right\rVert}

\newcommand{\pairone}{$\overline{\text{P}}$0d1d1P1}
\newcommand{\pairtwo}{$\overline{\text{P}}$1d2d2P2}
\newcommand{\pairthree}{P1c1$\overline{\text{d}}$0$\overline{\text{P}}$0}
\newcommand{\pairfour}{P2c2$\overline{\text{d}}$1$\overline{\text{P}}$1}
\newcommand{\pairfive}{P2$\hat{\text{c}}$1$\overline{\text{d}}$0$\overline{\text{P}}$0}
\newcommand{\nedelec}{N{\'e}d{\'e}lec}

\newcommand{\twoprod}[2]{\langle\!\langle #1 \rangle\!\rangle_{#2}}

\newcommand{\bcurl}{\operatorname*{\mathbf{curl}}}
\newcommand{\bdiv}{\operatorname*{\mathbf{div}}}
\newcommand{\vdiv}{\operatorname*{div}}

\newcommand{\bb}{\boldsymbol{b}}
\newcommand{\bC}{\mathbf{C}}
\newcommand{\bI}{\mathbf{I}}
\newcommand{\bF}{\mathbf{F}}

\newcommand{\bu}{\boldsymbol{u}}
\newcommand{\bV}{\mathbf{V}}
\newcommand{\bv}{\boldsymbol{v}}
\newcommand{\bK}{\mathbf{K}}
\newcommand{\bk}{\boldsymbol{k}}
\newcommand{\bkappa}{\boldsymbol{\kappa}}
\newcommand{\bP}{\mathbf{P}}
\newcommand{\bpsi}{\boldsymbol{\psi}}
\newcommand{\bpi}{\boldsymbol{\pi}}
\newcommand{\bH}{\mathbf{H}}
\newcommand{\bL}{\mathbf{L}}

\newcommand{\bn}{\boldsymbol{n}}
\newcommand{\bbH}{\mathbb{H}}
\newcommand{\bbI}{\mathbb{I}}

\newcommand{\bbL}{\mathbb{L}}
\newcommand{\bbR}{\mathbb{R}}
\newcommand{\bnabla}{\boldsymbol{\nabla}}

\newcommand{\bDelta}{\boldsymbol{\Delta}}

\newcommand{\bS}{\mathbf{S}}
\newcommand{\bT}{\mathbf{T}}
\newcommand{\cT}{\mathcal{T}}
\newcommand{\bt}{\boldsymbol{t}}
\newcommand{\bQ}{\mathbf{Q}}
\newcommand{\bX}{\mathbf{X}}
\newcommand{\bx}{\mathbf{x}}
\newcommand{\bzero}{\mathbf{0}}
\newcommand{\bz}{\boldsymbol{z}}

\newcommand{\rL}{\mathrm{L}}
\newcommand{\rH}{\mathrm{H}}

\newcommand\btau{\boldsymbol{\tau}}
\newcommand\bsigma{\boldsymbol{\sigma}}
\newcommand\bgamma{\boldsymbol{\gamma}}
\newcommand\tr{\mathop{\mathbf{tr}}\nolimits}
\newcommand\dev{\mathrm{dev}}

\newcommand{\drawarrow}[5]{
  \draw[-Triangle, #4, thick, opacity=#5]
    ($(#1)!#3!(#2)$) -- ($($(#1)!#3!(#2)$)!{0.11/(1-#3)}!-90:(#2)$);
}

\newtheorem{lemma}{Lemma}[section]
\newtheorem{theorem}{Theorem}[section]

\newtheorem{corollary}{Corollary}

\begin{document}

\title[Four-field mixed finite elements]{Four-field mixed finite elements for incompressible nonlinear elasticity}
\author{Santiago Badia$^\dagger$}
\email{santiago.badia@monash.edu}
\author{Wei Li$^{\dagger,*}$}
\email{wei.li@monash.edu}
\author{Ricardo Ruiz-Baier$^\dagger$}
\email{ricardo.ruizbaier@monash.edu}
\address{$^\dagger$School of Mathematics, Monash University, Clayton, Victoria 3800, Australia.}
\address{$^*$Corresponding author.}

\date{\today}

\begin{abstract}
We present a stable finite element method for incompressible nonlinear elasticity based on a four-field mixed formulation involving the displacement, displacement gradient, first Piola--Kirchhoff stress and pressure. Unlike existing four-field mixed formulations, such as the compatible strain mixed finite element method (CSFEM), the proposed approach employs a discontinuous displacement field and requires no stabilisation in either 2D or 3D. A Newton--Raphson linearisation is derived and finite element pairs satisfying the relevant inf-sup conditions are identified. To recover accurate continuous displacement fields, an efficient postprocessing technique is further introduced. We establish the well-posedness of the linearised continuous problem together with a priori error estimates for the discrete formulation. Extensive numerical experiments in both 2D and 3D demonstrate optimal or even super convergence rates and enhanced robustness, particularly in 3D where CSFEM typically requires stabilisation.
\end{abstract}


\maketitle

\section{Introduction} \label{sec:intro}
Nonlinear modelling of incompressible solids arises in a wide range of engineering and scientific applications, including biomedical tissue mechanics~\cite{Rossi2014,henke2026electromechanical}, rubber elasticity~\cite{Treloar2005} and soft-material design~\cite{Liu2014}. Since analytical solutions to the underlying initial boundary-value problems are often infeasible, the \ac{fem} is a ubiquitous tool for approximating their solutions. Over the past decades, \ac{fem} has developed a rigorous mathematical foundation for nonlinear continuum mechanics~\cite{Zienkiewicz2005,Bonet2008}. In parallel, advances in solvers and high-performance computing have expanded the range of problems that \ac{fem} can handle, enabling simulations at extreme scales~\cite{Badia2016}.

In solid mechanics, variational principles are commonly categorised into single-field, two-field, three-field and four-field formulations. Many \acp{fem} are built on these principles, from displacement-based approaches to mixed and hybrid methods~\cite{Hughes2000,Zienkiewicz2005,Bonet2008}. 
The stationary potential energy principle yields a weak formulation in which displacement is the only independent variable. For (nearly) incompressible materials, it is often combined with a penalty approach that treats the material as slightly compressible through a large bulk modulus. After \ac{fe} discretisation, the stiffness matrix becomes increasingly ill-conditioned as the penalty parameter grows, yet a sufficiently large penalty is needed to enforce incompressibility~\cite{Holzapfel2001}. Consequently, displacement-based formulations can suffer severe volumetric locking and perform poorly for nearly or fully incompressible materials~\cite{Brezzi1991}.

To mitigate locking, multi-field variational principles introduce additional independent fields and lead to mixed \acp{fem}~\cite{Arnold1990}. A common two-field principle uses a Lagrange multiplier (hydrostatic pressure) to impose incompressibility, yielding a mixed displacement--pressure formulation. After discretisation, the resulting system is of saddle-point type, and well-posedness requires compatible \ac{fe} spaces for displacement and pressure that satisfy an inf-sup condition at the discrete level.

Several two-field \ac{fe} pairs satisfy the discrete inf-sup condition, including the Crouzeix--Raviart element~\cite{Crouzeix1973}, the Taylor--Hood element~\cite{Taylor1973} and the MINI element~\cite{Arnold1984}. A biorthogonal construction is proposed in~\cite{Lamichhane2009}, with pressure test functions biorthogonal to the trial functions, and a stress--displacement two-field formulation follows from the Hellinger--Reissner principle~\cite{Hellinger1907,Reissner1950} (see, e.g.,~\cite{Viebahn2018} for nearly incompressible elasticity). While many of these pairs are stable for linear mixed formulations, stability can deteriorate in large-deformation incompressible hyperelasticity; for instance, the MINI element has a reduced stability range in 2D nonlinear incompressible elasticity~\cite{Auricchio2013}.

Adding one more field can lead to more robust three-field principles. The \ac{stp} principle~\cite{Simo1985}, based on the Hu--Washizu framework~\cite{Hu1954,Washizu1955}, augments displacement and pressure with a third kinematic variable whose constraint is enforced through pressure. The enhanced strain mixed method~\cite{Simo1992} derived from this principle decomposes the deformation gradient into conforming and enhanced parts, extending incompatible mode ideas to the nonlinear regime. Another three-field method in~\cite{Farrell2021} uses the symmetric Kirchhoff stress together with displacement and pressure; it preserves stress symmetry but requires a heuristic choice of stabilisation constant in the lowest-order case.

Four-field mixed formulations for (nearly) incompressible solids have recently attracted increasing interest. Typically based on the Hu--Washizu principle~\cite{Hu1954,Washizu1955}, they treat displacement, displacement gradient, stress and pressure as independent fields. The \ac{csfem}~\cite{Angoshtari2016} yields robust methods for 2D~\cite{Shojaei2018} and 3D~\cite{Shojaei2019}, but requires different element pairs across dimensions and, in 3D, non-standard \acp{fe} and additional stabilisation. The \ac{ndtns} method~\cite{Fu2025}, based on the mass-conserving mixed stress formulation~\cite{Gopalakrishnan2019} and the Hu--Washizu functional, avoids inverting the constitutive law and allows flexibility in field regularity, but still relies on stabilisation and non-standard matrix-valued tangential-normal continuous \ac{fe} spaces.

In this work, we develop a \emph{stable} four-field mixed \ac{fe} formulation for incompressible solids. We analyse the linearised nonlinear elasticity problem to identify stable element pairs and study their error convergence. Unlike \ac{csfem}, the same stable pairs apply in both 2D and 3D. In contrast to \ac{csfem} and \ac{ndtns}, our approach uses standard \acp{fe} (e.g., Lagrange and \ac{bdm}), making implementation straightforward and avoiding additional stabilisation and parameter tuning. We present 2D and 3D numerical experiments, including comparisons with \ac{csfem}, showing that the method avoids common artefacts (locking and checkerboarding) while remaining robust and accurate.

The remainder of the article is organised as follows. \sect{sec:formu} introduces the notation and mathematical setting and presents a four-field formulation with its Newton--Raphson linearisation. \sect{sec:fe} describes the \ac{fe} discretisation, identifies stable pairs and introduces a discontinuous displacement correction. \sect{sec:na} establishes well-posedness and derives a priori error estimates. \sect{sec:exps} presents 2D and 3D numerical experiments, including comparisons with \acp{csfem}. \sect{sec:conclu} concludes and outlines future work.

\section{A mixed formulation for incompressible nonlinear elasticity} \label{sec:formu}
\subsection{Notation} \label{sec:formu:note}
Let $\Omega\subset\bbR^d$, $d\in\{2,3\}$, be an open, connected Lipschitz domain with piecewise smooth boundary $\Gamma \doteq \partial\Omega$, representing the referential (or material) configuration of a deformable body. Let $\bn$ denote the outward unit normal vector on $\Gamma$. The boundary is partitioned into the displacement and traction boundaries, $\Gamma_d$ and $\Gamma_t$, respectively, with   $\Gamma_d \cap \Gamma_t = \emptyset$ and $\Gamma_d \cup \Gamma_t = \Gamma$.

We consider a scalar field $p : \Omega \rightarrow \mathbb{R}$, a vector field $\bv : \Omega \rightarrow \mathbb{R}^d$ and a second-order tensor field $\bT : \Omega \rightarrow \mathbb{R}^{d \times d}$.
Let $\mathrm{L}^2(\Omega)$, $\mathbf{L}^2(\Omega)$ and $\mathbb{L}^2(\Omega)$ denote square-integrable scalar, vector and second-order tensor fields, respectively, with norms $\norm{\cdot}_{\mathrm{L}^2(\Omega)}$, $\norm{\cdot}_{\mathbf{L}^2(\Omega)}$ and $\norm{\cdot}_{\mathbb{L}^2(\Omega)}$. We then introduce the standard $\mathbf{grad}$- and $\bdiv$-conforming spaces
\begin{align*}
  \mathrm{H}^1(\Omega) &\doteq \{p \in \mathrm{L}^2(\Omega) : \nabla p \in \mathbf{L}^2(\Omega)\}, &
  \mathbf{H}^1(\Omega) &\doteq \{\bv \in \mathbf{L}^2(\Omega) : \bnabla \bv \in \mathbb{L}^2(\Omega)\},\\
  \mathbf{H}(\vdiv, \Omega) &\doteq \{\bv \in \mathbf{L}^2(\Omega) : \nabla \cdot \bv \in \mathrm{L}^2(\Omega)\}, &
  \mathbb{H}(\bdiv, \Omega) &\doteq \{\bT \in \mathbb{L}^2(\Omega) : \bnabla \cdot \bT \in \mathbf{L}^2(\Omega)\}.
\end{align*}
The $\bcurl$-conforming spaces depend on the spatial dimension:
\begin{align*}
  \mathbf{H}(\bcurl, \Omega) &\doteq 
    \begin{cases} 
      \{\bv \in \mathbf{L}^2(\Omega) : \bnabla \times \bv \in \mathrm{L}^2(\Omega)\}, & d=2, \\
      \{\bv \in \mathbf{L}^2(\Omega) : \bnabla \times \bv \in \mathbf{L}^2(\Omega)\}, & d=3,
    \end{cases} \\
  \mathbb{H}(\bcurl, \Omega) &\doteq 
    \begin{cases} 
      \{\bT \in \mathbb{L}^2(\Omega) : \bnabla \times \bT \in \mathbf{L}^2(\Omega)\}, & d=2, \\
      \{\bT \in \mathbb{L}^2(\Omega) : \bnabla \times \bT \in \mathbb{L}^2(\Omega)\}, & d=3.
    \end{cases}
\end{align*}
Each space is equipped with its natural norm. For instance, the norm in $\mathbb{H}(\bdiv, \Omega)$ is
$$ \norm{\bT}_{\mathbb{H}(\bdiv, \Omega)} = \left( \norm{\bT}^2_{\mathbb{L}^2(\Omega)} + \norm{\bnabla \cdot \bT}_{\mathbf{L}^2(\Omega)} \right)^{1/2}. $$

The $L^2$-inner product for scalar, vector and tensor fields is denoted uniformly by $\twoprod{\cdot,\cdot}{D}$. For example,
\begin{align*}
  \twoprod{p,q}{\Omega} = \int_{\Omega} p q \,\mathrm{d}V, \quad
  \twoprod{\bu,\bv}{\Gamma} = \int_{\Gamma} \bu \cdot \bv \,\mathrm{d}S, \quad 
  \twoprod{\bT,\bS}{\Omega} = \int_{\Omega} \bT : \bS \,\mathrm{d}V,
\end{align*}
where $p,q$ are scalar fields, $\bu,\bv$ are vector fields and $\bT,\bS$ are second-order tensor fields. The symbol ``$:$'' denotes the Frobenius inner product, i.e.,
$$ \bT : \bS = \sum_{i,j} \bT_{i,j} \bS_{i,j}. $$

\subsection{Incompressible nonlinear elasticity} \label{sec:formu:elas}
In this subsection, we review the essential ingredients of finite deformation. The displacement vector field
$\bu \doteq \bx - \bX$
maps each material point $\bX$ in the referential configuration to its spatial (or current) position $\bx$ in the deformed configuration. 
The displacement gradient in the material description is the second-order tensor
\begin{equation} \label{eq:graddisp}
  \bK \doteq \bnabla \bu.
\end{equation} 
Let $\bbI$ denote the second-order identity tensor; then the deformation gradient is
$$\bF \doteq \bnabla\bu + \bbI,$$
with the Jacobian determinant
$$J \doteq \det\bF = \det(\bnabla\bu + \bbI) > 0,$$
which measures the solid volume change during deformation.
An important strain measure in material coordinates is the right Cauchy--Green deformation tensor 
$$\bC \doteq \bF^{\tt t}\bF,$$
where the superscript $(\cdot)^{\tt t}$ denotes the transpose operator.

The constitutive behaviour of the material is encoded in a strain energy density function $\tilde{\Psi}$, representing the work per unit reference volume done by the stress in deforming the material system. It is expressed solely in terms of the deformation gradient.
For simplicity, we adopt a generic incompressible neo-Hookean material law, although the proposed method applies straightforwardly to other models. The strain energy density reads
\begin{equation} \label{eq:Psi}
  \tilde{\Psi}(\bF) = \frac{\mu}{2} (I_1 - d),
\end{equation}
where $\mu$ is the Lam{\'e} parameter, $I_1 = \tr\bC$ and $\tr(\cdot)$ denotes the trace operator.
Incompressible solids that maintain constant volume during deformation are characterised by the incompressibility constraint
$J = 1.$
To impose this constraint, we introduce a smooth function $C: \mathbb{R}^+ \rightarrow \mathbb{R}$ satisfying $C(J)=0$ if and only if $J=1$. As suggested by~\cite{Shojaei2018,Shojaei2019,Fu2025}, we employ two common choices:
\begin{subequations} \label{eq:C}
  \begin{align}
    C_1(J) &= J - 1, \label{eq:C1} \\
    C_2(J) &= \ln(J).\label{eq:C2}
  \end{align}
\end{subequations}
To derive a constitutive law for incompressible materials, we introduce the modified strain energy 
$$\Psi(\bF) = \tilde{\Psi}(\bF) - p C(J),$$
where $p \in \mathrm{L}^2(\Omega)$ is the hydrostatic pressure, i.e., the Lagrange multiplier enforcing incompressibility in the finite strain regime.
The first Piola--Kirchhoff stress tensor is then
\begin{equation}\label{eq:P}
  \bP = \frac{\partial\Psi}{\partial\bF} = \frac{\partial\tilde{\Psi}}{\partial \bF} - p \frac{\partial C}{\partial \bF} = \tilde{\bP}(\bK) - p \bQ(\bK).
\end{equation}
For neo-Hookean materials, it follows directly that
\begin{equation} \label{eq:tildeP}
  \tilde{\bP}(\bK) = \mu (\bK + \bbI),
\end{equation}
and
\begin{subequations} \label{eq:Qs}
  \begin{align} 
    \bQ_1(\bK) &= \frac{\partial C_1}{\partial \bK} = \det(\bK + \bbI)(\bK + \bbI)^{\tt -t}, \label{eq:Q1} \\
    \bQ_2(\bK) &= \frac{\partial C_2}{\partial \bK} = (\bK + \bbI)^{\tt -t}. \label{eq:Q2}
  \end{align}
\end{subequations}
In the inertial reference frame under static mechanical equilibrium, combining the balance of linear momentum, the constitutive relation, the incompressibility constraint and the boundary conditions yields the following strong form:
\begin{subequations}\label{eq:strong}
  \begin{align}
    \bnabla \cdot \bP & = -\rho_0 \bb &\text{ in } \Omega, \label{eq:momentum}\\ 
    \bP &= \tilde{\bP}(\bK) - p \bQ(\bK) &\text{ in } \Omega, \label{eq:constitutive}\\ 
    \bK &= \bnabla \bu   &\text{ in } \Omega,\label{eq:constitutive2} \\
    C(J) & = 0 & \text{in } \Omega, \label{eq:mass}
  \end{align}
\end{subequations}
where $\rho_0 > 0$ is the reference medium density and $\bb \in \mathbf{L}^2(\Omega)$ is the body force per unit undeformed volume. The above governing equations are supplemented with displacement and traction boundary conditions
\begin{equation} \label{eq:bc}
    \bu = \bar{\bu} \quad \text{on} \quad \Gamma_d, \qquad \text{and} \qquad \bP \bn = \bar{\bt} \quad \text{on} \quad \Gamma_t,
\end{equation}
where $\bar{\bu} \in \mathbf{H}^{1/2}(\Gamma_d)$ and $\bar{\bt} \in \mathbf{H}^{-1/2}(\Gamma_t)$ are the prescribed displacement and traction, respectively.

\subsection{A four-field mixed formulation}\label{sec:formu:fourfield}
We now test equations~\eqref{eq:momentum}-\eqref{eq:mass} against suitable functions. Unlike the approach in~\cite{Shojaei2018,Shojaei2019}, where integration by parts is applied to~\eqref{eq:momentum}, we integrate by parts in~\eqref{eq:constitutive2} to form a new four-field mixed weak formulation: find $(\bu, \bK, \bP, p) \in \mathbf{L}^2(\Omega)\times \mathbb{L}^2(\Omega) \times \mathbb{H}^{\bar{\bt}}(\bdiv, \Omega) \times \mathrm{L}^2(\Omega)$ such that
\begin{subequations}\label{eq:weak}
  \begin{align}
    \twoprod{\bv, \bnabla \cdot \bP}{\Omega} &= -\twoprod{\bv, \rho_0 \bb}{\Omega} &\forall \bv \in \mathbf{L}^2(\Omega),\label{eq:weak-1}\\
    \twoprod{\bgamma, \bP - \tilde{\bP}(\bK) + p \bQ(\bK)}{\Omega} &= 0 &\forall \bgamma \in \mathbb{L}^2(\Omega),\label{eq:weak-2}\\
    \twoprod{\btau, \bK}{\Omega} + \twoprod{\bnabla \cdot \btau, \bu}{\Omega} & = \twoprod{\btau \bn, \bar{\bu}}{\Gamma_d} &\forall \btau \in \mathbb{H}^{\bzero}(\bdiv,\Omega),\label{eq:weak-3}\\
    \twoprod{q, C(J)}{\Omega} &= 0 &\forall q\in \mathrm{L}^2(\Omega).\label{eq:weak-4}
  \end{align}
\end{subequations}
Here, $\mathbb{H}^{\bzero}(\bdiv, \Omega) \doteq \{ \bT \in \mathbb{H}(\bdiv, \Omega) : \bT \bn = \bzero \}$ and $\mathbb{H}^{\bar{\bt}}(\bdiv, \Omega) \doteq \{ \bT \in \mathbb{H}(\bdiv, \Omega) : \bT \bn = \bar{\bt} \}$.
This four-field formulation ensures that the displacement gradient, stress and pressure are not postprocessed from the displacement. In contrast to the formulations of~\cite{Shojaei2018,Shojaei2019}, the displacement boundary condition is imposed weakly via~\eqref{eq:weak-3}, whereas the traction boundary condition is enforced strongly through the space $ \mathbb{H}^{\bar{\bt}}(\bdiv, \Omega)$.
 
\subsection{Linearisation} \label{sec:formu:lin}
We now apply a Newton--Raphson linearisation to~\eqref{eq:weak}, starting from an initial guess. Using the identities for the derivative of the inverse transpose and the determinant of a tensor, we obtain 
\begin{equation*}
  \frac{d}{d\bF}(\bF^{-\tt t})|_{\bu} = - \bF^{-\tt t}(\nabla\bu)^{\tt t}\bF^{-\tt t}, \quad 
  \frac{d}{d\bF}(\det\bF)|_{\bu} = J(\bF^{-\tt t}:\nabla \bu)\bbI.
\end{equation*}
We initialise the iteration from the stress-free, motionless state 
\begin{equation*}
  (\bu^{k=0},\bK^{k=0},\bP^{k=0},p^{k=0}) = (\bzero,\bzero,\bzero,0).
\end{equation*}
For simplicity, we set the traction boundary data $\bar{\bt}=\bzero$; the linearised problem reduces to finding $(\bu,\bk,\bsigma,p)\in \mathbf{L}^2(\Omega)\times \mathbb{L}^2(\Omega)\times \mathbb{H}^{\bzero}(\bdiv,\Omega)\times \mathrm{L}^2(\Omega)$ such that 
\begin{equation}\label{eq:weak-lin}
\begin{alignedat}{4}
& &\twoprod{\bv, \bnabla \cdot \bsigma}{\Omega}& & &= -\twoprod{\bv, \rho_0 \bb}{\Omega} &\quad \forall \bv \in \mathbf{L}^2(\Omega),\\
&-\mu \twoprod{\bgamma, \bk}{\Omega} & +\,\twoprod{\bgamma, \bsigma}{\Omega}& +\,\twoprod{\tr(\bgamma), p}{\Omega}& &= \twoprod{\tr(\bgamma), \mu}{\Omega} &\quad \forall \bgamma \in \mathbb{L}^2(\Omega),\\
\twoprod{\bnabla\cdot\btau, \bu}{\Omega}& +\,\twoprod{\btau, \bk}{\Omega}& & & &= \twoprod{\btau\bn, \bar{\bu}}{\Gamma_d} &\quad \forall \btau \in \mathbb{H}^{\bzero}(\bdiv,\Omega),\\
&\twoprod{q, \tr(\bk)}{\Omega}& & & &= 0 &\quad \forall q \in \mathrm{L}^2(\Omega).
\end{alignedat}
\end{equation}

\section{Finite element discretisation} \label{sec:fe}
Let $\cT_h$ be a shape-regular partition of $\overline{\Omega}$ into triangles in 2D or tetrahedra in 3D. The mesh size is defined as $h \doteq \max\left\{h_K:\ K\in \cT_h \right\}$, where $K$ is an element from $\cT_h$ with diameter $h_K$.
Given an integer $k > 0$, for each $K\in \cT_h$, we denote by $\mathrm{P}_k(K)$ the space of polynomials on $K$ of degree at most $k$ and by $\mathrm{P}_k(\cT_h)$ its global counterpart (similarly for the other discrete spaces below).
We define $\bP_k(K) \doteq [\mathrm{P}_k(K)]^d$ as the vector-valued polynomial space on $K$. The \ac{bdm} \ac{fe} space on $K$ is given by $\mathbf{BDM}_k(K) \doteq \bP_k(K)$. We then introduce the discrete spaces:
\begin{align*}
  \mathrm{H}_{k,h} &\doteq \mathrm{P}_k(\cT_h) \cap C^0(\Omega) 
  \subset \mathrm{H}^1(\Omega),&
  \mathbb{H}_{k,h} &\doteq [\mathbf{BDM}_k(\cT_h)]^d 
  \subset \mathbb{H}(\bdiv, \Omega),\\
  \mathbf{H}_{k,h} &\doteq \mathbf{P}_k(\cT_h) \cap C^0(\Omega) 
  \subset \mathbf{H}^1(\Omega),&
  \overline{\mathbf{H}}_{k-1,h} &\doteq \mathbf{P}_{k-1}(\cT_h) 
  \subset \mathbf{L}^2(\Omega).
\end{align*}
Finally, we define the $\bdiv$-conforming subspaces with prescribed normal traces on $\Gamma_t$:
\begin{equation*}
  \mathbb{H}_{k,h}^{\bzero} \doteq  \{ \bT_h \in \mathbb{H}_{k,h} : \bT_h \bn = \bzero \text{ on } \Gamma_t \}, \quad \mathbb{H}_{k,h}^{\bar{\bt}} \doteq  \{ \bT_h \in \mathbb{H}_{k,h} : \bT_h \bn = \bar{\bt} \text{ on } \Gamma_t \},
\end{equation*}
and the $\mathbf{grad}$-conforming subspaces with prescribed displacements on $\Gamma_d$:
\begin{equation*}
  \mathbf{H}^{\bzero}_{k,h} \doteq   \{ \bv_h \in \mathbf{H}_{k,h} : \bv_h = \bzero \text{ on } \Gamma_d \}, \quad
  \mathbf{H}^{\bar{\bu}}_{k,h} \doteq   \{ \bv_h \in \mathbf{H}_{k,h} : \bv_h = \bar{\bu} \text{ on } \Gamma_d \}.
\end{equation*}

\subsection{A mixed finite element method for incompressible nonlinear elasticity} \label{sec:fe:mixed_fe}
Based on the continuous variational formulation~\eqref{eq:weak}, we can derive its discrete counterpart. After discretisation, appropriate \ac{fe} spaces must be chosen to satisfy the inf-sup conditions; such combinations are referred to as \emph{stable \ac{fe} pairs}.
With the numerical analysis done in \sect{sec:na:discrete}, we arrive at the following discrete formulation of~\eqref{eq:weak}: find $(\bu_h, \bK_h, \bP_h, p_h) \in \overline{\mathbf{H}}_{k-1,h} \times \mathbb{H}_{k,h} \times \mathbb{H}_{k,h}^{\bar{\bt}} \times \mathrm{H}_{k,h}$ such that
\begin{equation}\label{eq:fe}
\begin{aligned}
  \twoprod{\bv_h, \bnabla \cdot \bP_h}{\Omega} &= -\twoprod{\bv_h, \rho_0 \bb}{\Omega} &\forall \bv_h \in \overline{\mathbf{H}}_{k-1,h},\\
  \twoprod{\bgamma_h, \bP_h - \tilde{\bP}(\bK_h) + p_h \bQ(\bK_h)}{\Omega} &= 0 &\forall \bgamma_h \in \mathbb{H}_{k,h},\\
  \twoprod{\btau_h, \bK_h}{\Omega} + \twoprod{\bnabla \cdot \btau_h, \bu_h}{\Omega} & = \twoprod{\btau_h \bn, \bar{\bu}}{\Gamma_d} &\forall \btau_h \in \mathbb{H}_{k,h}^{\bzero},\\
  \twoprod{q_h, C(J)}{\Omega} &= 0 &\forall q_h \in \mathrm{H}_{k,h}.
\end{aligned}
\end{equation}

\subsection{Finite elements} \label{sec:fe:fe}
For notational convenience, we use a two-character symbol to denote both the type and the polynomial order of each \ac{fe} space. The symbols used throughout this paper are summarised in \tab{tab:fesymbols}. 
Note that we also include the special \nedelec{} elements $\hat{\text{c}}1$, which combine six second-kind first-order and nine first-kind third-order \nedelec{} shape functions. These elements are essential for the 3D \ac{csfem}~\cite{Shojaei2019}.

\begin{table}[hbt!]
  \centering
  \begin{tabular}{clll}
    \toprule
    Symbol & Description & Conformity & Order range\\
    \midrule
    Pi & Continuous Lagrange elements of order i. & $\mathrm{H}^1(\Omega)$ & $\text{i}\in(1, 2)$\\
    $\overline{\text{P}}$i & Discontinuous Lagrange elements of order i. & L$^2(\Omega)$ & $\text{i}\in(0, 1)$\\
    ci & Second-kind \nedelec{} elements of order i. & $\mathbf{H}(\bcurl,\Omega)$ & $\text{i}\in(0, 1)$\\
    $\hat{\text{c}}1$ & Special \nedelec{} elements proposed in~\cite{Shojaei2019}. & $\mathbf{H}(\bcurl,\Omega)$ & N/A\\
    di & \ac{bdm} elements of order i. & $\mathbf{H}(\vdiv,\Omega)$ & $\text{i}\in(1, 2)$\\
    $\overline{\text{d}}$i & Raviart-Thomas elements of order i. & $\mathbf{H}(\vdiv,\Omega)$ & $\text{i}\in(0, 1)$\\
    \bottomrule
  \end{tabular}
  \caption{Notation for the \ac{fe} spaces considered in this work.}
  \label{tab:fesymbols}
\end{table}

Since our method is based on a four-field formulation, we encode each mixed \ac{fe} combination by eight characters, with each two-character segment specifying the \ac{fe} space for one field. For example, \pairtwo{} uses a discontinuous vector-valued Lagrange space of order 1 ($\overline{\mathbf{H}}_{1,h}$) for $\bu_h$, a second-order tensor-valued \ac{bdm} space ($\mathbb{H}_{2,h}$) for $\bK_h$ and $\bP_h$, and a continuous scalar-valued Lagrange space of order 2 ($\mathrm{H}_{2,h}$) for $p_h$.
Note that vector-valued Lagrange spaces are constructed by stacking $d$ copies of the corresponding scalar Lagrange space, while tensor-valued \ac{fe} spaces of a given type are built from the corresponding vector-valued spaces by treating each row or column of the tensor as a vector field.

As shown in \sect{sec:na:discrete}, the stable \ac{fe} pairs in 2D and 3D are \pairone{} and \pairtwo{}. Higher-order pairs may also be stable but are computationally expensive, so we focus on these two.
Figure~\ref{fig:cell2ddof} presents the element-wise \acp{dof} of the two stable pairs in 2D. The low-order pair \pairone{} has 29 \acp{dof} per cell, while the high-order pair \pairtwo{} has 60 \acp{dof}. In 3D, these counts increase to 79 and 202 \acp{dof}, respectively.
In the numerical experiments, we compare these pairs with the \ac{csfem} pairs \pairthree{} (25 \acp{dof}) and \pairfour{} (55 \acp{dof}) in 2D and \pairfive{} (88 \acp{dof}) in 3D. The low-order pairs (\pairone{} and \pairthree{}) and high-order pairs (\pairtwo{} and \pairfour{}) have comparable \ac{dof} counts per element in 2D, allowing for direct comparison. In 3D, however, \pairtwo{} has significantly more \acp{dof} than \pairfive{}, so a more balanced comparison is between \pairone{} and \pairfive{}.

\begin{figure}[hbt!]
  \centering
  \begin{tabular}{M{1.6cm}M{2.6cm}M{2.6cm}M{2.6cm}M{2.6cm}M{0.6cm}}
    Pair & $\bu_h$ & $\bK_h$ & $\bP_h$ & $p_h$ & \acp{dof} \\[1ex]
    \pairone{} & 
    \begin{tikzpicture}[scale=1.2]
      \draw[thick] (-1, 0) coordinate (A) -- (1, 0) coordinate (B) -- (0, {sqrt(3)}) coordinate (C) -- cycle;
      \fill[magenta] ($($(A)!0.41!(B)$)!0.4!(C)$) circle [radius=0.08];
      \fill[magenta] ($($(A)!0.59!(B)$)!0.4!(C)$) circle [radius=0.08];

      \fill[white,opacity=0] (C) circle [radius=0.08];
      \drawarrow{A}{B}{1/2}{white}{0}
    \end{tikzpicture} & 
  \begin{tikzpicture}[scale=1.2]
    \draw[thick] (-1, 0) coordinate (A) -- (1, 0) coordinate (B) -- (0, {sqrt(3)}) coordinate (C) -- cycle;
    \drawarrow{A}{B}{03/16}{blue}{1}
    \drawarrow{A}{B}{05/16}{blue}{1}
    \drawarrow{A}{B}{11/16}{blue}{1}
    \drawarrow{A}{B}{13/16}{blue}{1}
    \drawarrow{C}{A}{03/16}{blue}{1}
    \drawarrow{C}{A}{05/16}{blue}{1}
    \drawarrow{C}{A}{11/16}{blue}{1}
    \drawarrow{C}{A}{13/16}{blue}{1}
    \drawarrow{B}{C}{03/16}{blue}{1}
    \drawarrow{B}{C}{05/16}{blue}{1}
    \drawarrow{B}{C}{11/16}{blue}{1}
    \drawarrow{B}{C}{13/16}{blue}{1}

    \fill[white,opacity=0] (C) circle [radius=0.08];
  \end{tikzpicture} & 
  \begin{tikzpicture}[scale=1.2]
    \draw[thick] (-1, 0) coordinate (A) -- (1, 0) coordinate (B) -- (0, {sqrt(3)}) coordinate (C) -- cycle;
    \drawarrow{A}{B}{003/16}{cyan}{1}
    \drawarrow{A}{B}{05/16}{cyan}{1}
    \drawarrow{A}{B}{11/16}{cyan}{1}
    \drawarrow{A}{B}{13/16}{cyan}{1}
    \drawarrow{C}{A}{03/16}{cyan}{1}
    \drawarrow{C}{A}{05/16}{cyan}{1}
    \drawarrow{C}{A}{11/16}{cyan}{1}
    \drawarrow{C}{A}{13/16}{cyan}{1}
    \drawarrow{B}{C}{03/16}{cyan}{1}
    \drawarrow{B}{C}{05/16}{cyan}{1}
    \drawarrow{B}{C}{11/16}{cyan}{1}
    \drawarrow{B}{C}{13/16}{cyan}{1}

    \fill[white,opacity=0] (C) circle [radius=0.08];
  \end{tikzpicture} & 
  \begin{tikzpicture}[scale=1.2]
    \draw[thick] (-1, 0) coordinate (A) -- (1, 0) coordinate (B) -- (0, {sqrt(3)}) coordinate (C) -- cycle;
    \fill[orange] (A) circle [radius=0.08];
    \fill[orange] (B) circle [radius=0.08];
    \fill[orange] (C) circle [radius=0.08];

    \drawarrow{A}{B}{1/2}{white}{0}
  \end{tikzpicture} & 
  29 \\
  \pairtwo{} &
  \begin{tikzpicture}[scale=1.2]
    \draw[thick] (-1, 0) coordinate (A) -- (1, 0) coordinate (B) -- (0, {sqrt(3)}) coordinate (C) -- cycle;
    \fill[magenta] ($($(A)!0.4!(B)$)!0.5!(C)$) circle [radius=0.08];
    \fill[magenta] ($($(A)!0.4!(C)$)!0.5!(B)$) circle [radius=0.08];
    \fill[magenta] ($($(C)!0.4!(B)$)!0.5!(A)$) circle [radius=0.08];
    \fill[magenta] ($($(A)!0.6!(B)$)!0.5!(C)$) circle [radius=0.08];
    \fill[magenta] ($($(A)!0.6!(C)$)!0.5!(B)$) circle [radius=0.08];
    \fill[magenta] ($($(C)!0.6!(B)$)!0.5!(A)$) circle [radius=0.08];

    \drawarrow{A}{B}{1/2}{white}{0}
    \fill[white,opacity=0] (C) circle [radius=0.08];
  \end{tikzpicture} &
  \begin{tikzpicture}[scale=1.2]
    \draw[thick] (-1, 0) coordinate (A) -- (1, 0) coordinate (B) -- (0, {sqrt(3)}) coordinate (C) -- cycle;
    \drawarrow{A}{B}{2/16}{blue}{1}
    \drawarrow{A}{B}{4/16}{blue}{1}
    \drawarrow{A}{B}{7/16}{blue}{1}
    \drawarrow{A}{B}{9/16}{blue}{1}
    \drawarrow{A}{B}{12/16}{blue}{1}
    \drawarrow{A}{B}{14/16}{blue}{1}
    \drawarrow{C}{A}{2/16}{blue}{1}
    \drawarrow{C}{A}{4/16}{blue}{1}
    \drawarrow{C}{A}{7/16}{blue}{1}
    \drawarrow{C}{A}{9/16}{blue}{1}
    \drawarrow{C}{A}{12/16}{blue}{1}
    \drawarrow{C}{A}{14/16}{blue}{1}
    \drawarrow{B}{C}{2/16}{blue}{1}
    \drawarrow{B}{C}{4/16}{blue}{1}
    \drawarrow{B}{C}{7/16}{blue}{1}
    \drawarrow{B}{C}{9/16}{blue}{1}
    \drawarrow{B}{C}{12/16}{blue}{1}
    \drawarrow{B}{C}{14/16}{blue}{1}

    \fill[blue] ($($(A)!0.4!(B)$)!0.5!(C)$) circle [radius=0.08];
    \fill[blue] ($($(A)!0.4!(C)$)!0.5!(B)$) circle [radius=0.08];
    \fill[blue] ($($(C)!0.4!(B)$)!0.5!(A)$) circle [radius=0.08];
    \fill[blue] ($($(A)!0.6!(B)$)!0.5!(C)$) circle [radius=0.08];
    \fill[blue] ($($(A)!0.6!(C)$)!0.5!(B)$) circle [radius=0.08];
    \fill[blue] ($($(C)!0.6!(B)$)!0.5!(A)$) circle [radius=0.08];

    \fill[white,opacity=0] (C) circle [radius=0.08];
  \end{tikzpicture} & 
  \begin{tikzpicture}[scale=1.2]
    \draw[thick] (-1, 0) coordinate (A) -- (1, 0) coordinate (B) -- (0, {sqrt(3)}) coordinate (C) -- cycle;
    \drawarrow{A}{B}{2/16}{cyan}{1}
    \drawarrow{A}{B}{4/16}{cyan}{1}
    \drawarrow{A}{B}{7/16}{cyan}{1}
    \drawarrow{A}{B}{9/16}{cyan}{1}
    \drawarrow{A}{B}{12/16}{cyan}{1}
    \drawarrow{A}{B}{14/16}{cyan}{1}
    \drawarrow{C}{A}{2/16}{cyan}{1}
    \drawarrow{C}{A}{4/16}{cyan}{1}
    \drawarrow{C}{A}{7/16}{cyan}{1}
    \drawarrow{C}{A}{9/16}{cyan}{1}
    \drawarrow{C}{A}{12/16}{cyan}{1}
    \drawarrow{C}{A}{14/16}{cyan}{1}
    \drawarrow{B}{C}{2/16}{cyan}{1}
    \drawarrow{B}{C}{4/16}{cyan}{1}
    \drawarrow{B}{C}{7/16}{cyan}{1}
    \drawarrow{B}{C}{9/16}{cyan}{1}
    \drawarrow{B}{C}{12/16}{cyan}{1}
    \drawarrow{B}{C}{14/16}{cyan}{1}

    \fill[cyan] ($($(A)!0.4!(B)$)!0.5!(C)$) circle [radius=0.08];
    \fill[cyan] ($($(A)!0.4!(C)$)!0.5!(B)$) circle [radius=0.08];
    \fill[cyan] ($($(C)!0.4!(B)$)!0.5!(A)$) circle [radius=0.08];
    \fill[cyan] ($($(A)!0.6!(B)$)!0.5!(C)$) circle [radius=0.08];
    \fill[cyan] ($($(A)!0.6!(C)$)!0.5!(B)$) circle [radius=0.08];
    \fill[cyan] ($($(C)!0.6!(B)$)!0.5!(A)$) circle [radius=0.08];

    \fill[white,opacity=0] (C) circle [radius=0.08];
  \end{tikzpicture} & 
  \begin{tikzpicture}[scale=1.2]
    \draw[thick] (-1, 0) coordinate (A) -- (1, 0) coordinate (B) -- (0, {sqrt(3)}) coordinate (C) -- cycle;
    \fill[orange] (A) circle [radius=0.08];
    \fill[orange] (B) circle [radius=0.08];
    \fill[orange] (C) circle [radius=0.08];
    \fill[orange] ($(A)!0.5!(B)$) circle [radius=0.08];
    \fill[orange] ($(A)!0.5!(C)$) circle [radius=0.08];
    \fill[orange] ($(B)!0.5!(C)$) circle [radius=0.08];

    \drawarrow{A}{B}{1/2}{white}{0}
  \end{tikzpicture} & 
  60
  \end{tabular}
  \caption{Illustration of element-wise \acp{dof} for the proposed stable pairs in 2D.}
  \label{fig:cell2ddof}
\end{figure}
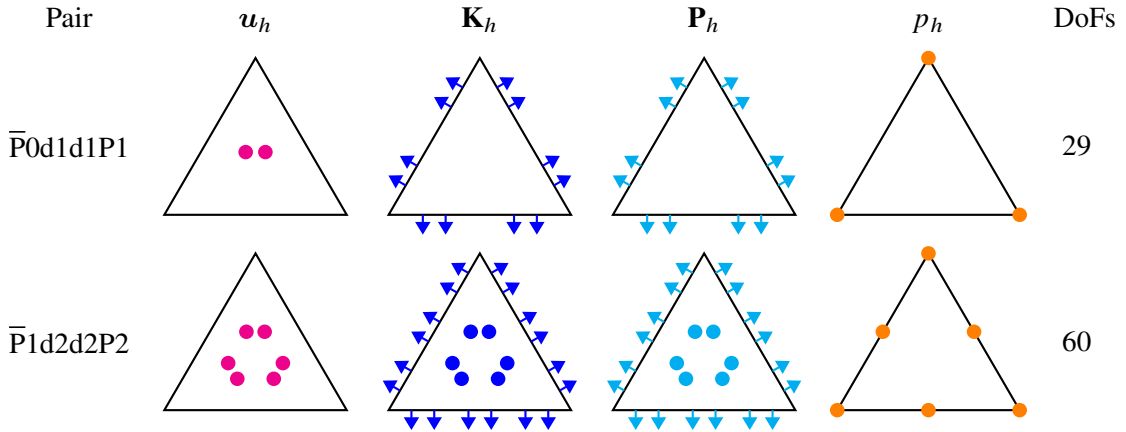

\subsection{Discontinuous displacement correction} \label{sec:fe:ddcorr}
In the proposed method, the displacement field is approximated in discontinuous Lagrange spaces. Therefore, the deformed configuration may exhibit inter-element discontinuities, especially on coarse meshes and for the pair \pairone{}. These discontinuities manifest as gaps or overlaps between neighbouring elements in the visualisation of the deformed geometry. 
Moreover, since displacement boundary conditions are imposed weakly, the discrete displacement field may not satisfy the prescribed boundary data pointwise, and noticeable deviations can occur in the vicinity of displacement boundaries.
Although the inflation experiments in \sect{sec:exps:inf} confirm the expected convergence of the discrete displacement, a postprocessing procedure is introduced to recover a globally continuous displacement field that satisfies the displacement boundary conditions strongly and improves the approximation in an energy-consistent sense.
The proposed strategy is reliable and computationally inexpensive, as it involves the solution of a symmetric positive-definite problem posed on a higher-order conforming space.

The correction procedure is conceptually related to Stenberg-type postprocessing techniques for mixed \acp{fem}~\cite{Stenberg1991}, where improved displacement fields are reconstructed from discrete fluxes. The main difference is that our approach is global rather than element-wise, operating over the entire mesh. Specifically, starting from the \ac{fe} solution $(\bu_h, \bK_h, \bP_h, p_h) \in \overline{\mathbf{H}}_{k-1,h} \times \mathbb{H}_{k,h} \times \mathbb{H}_{k,h}^{\bar{\bt}} \times \mathrm{H}_{k,h}$, the continuous relation~\eqref{eq:graddisp} is enforced in a weak sense to reconstruct a conforming displacement. The reconstruction problem reads: find $\tilde{\bu}_h \in \bH^{\bar{\bu}}_{k+1,h}$ such that for all $\bv_h \in \bH^{\bzero}_{k+1,h}$,
\begin{equation*}
  \twoprod{\bnabla \tilde{\bu}_h, \bnabla \bv_h}{\Omega} = \twoprod{\bK_h, \bnabla \bv_h}{\Omega}.
\end{equation*}
The resulting \ac{fe} tuple $(\tilde{\bu}_h, \bK_h, \bP_h, p_h) \in \bH^{\bar{\bu}}_{k+1,h} \times \mathbb{H}_{k,h} \times \mathbb{H}_{k,h}^{\bar{\bt}} \times \mathrm{H}_{k,h}$ features a globally continuous displacement field and satisfies both the displacement and traction boundary conditions in a strong sense. 

It is natural to use continuous Lagrange spaces for $\tilde{\bu}_h$. Since the relation~\eqref{eq:graddisp} involves a gradient acting on the displacement, the displacement correction space can be chosen one polynomial order higher than that of the displacement gradient \ac{fe} space. Accordingly, we use, e.g., P2 elements to correct \pairone{} displacements.
This postprocessing step is computationally much cheaper than the nonlinear solvers: it requires solving only a single, much smaller, symmetric positive-definite system.

In the numerical experiments, we append the tag ``(corr)'' after any pair whose displacement field has been postprocessed. For example, ``\pairone{}(corr)'' denotes the corrected \ac{fe} solution $(\tilde{\bu}_h, \bK_h, \bP_h, p_h)$, where $\tilde{\bu}_h$ is reconstructed from $\bK_h$ while the remaining fields are taken from the original \pairone{} \ac{fe} solution.

\section{Continuous and discrete analysis of the linearised problem} \label{sec:na}
\subsection{Well-posedness analysis of the linearised problem} \label{sec:na:lin}
Reordering the unknowns and equations, we write the linearised problem~\eqref{eq:weak-lin} as follows: find $(\bk,(\bsigma, p), \bu)\in \mathbb{L}^2(\Omega)\times [\mathbb{H}^{\bzero}(\bdiv,\Omega) \times \mathrm{L}^2(\Omega)] \times \bL^2(\Omega)$ such that 
\begin{equation}\label{eq:weak-lin3}
\begin{alignedat}{4}
  \mu \twoprod{\bk,\bgamma}{\Omega}& \,-\twoprod{\bsigma,\bgamma}{\Omega}& -\,\twoprod{p, \tr(\bgamma)}{\Omega}& & &= \twoprod{\mu, \tr(\bgamma)}{\Omega} &\quad \forall \bgamma \in \mathbb{L}^2(\Omega),\\
  -\twoprod{\bk, \btau}{\Omega}& & & -\twoprod{\bu, \bnabla\cdot\btau}{\Omega}& &= -\twoprod{\bar{\bu}, \btau\bn}{\Gamma_d} &\quad \forall \btau \in \mathbb{H}^{\bzero}(\bdiv,\Omega),\\
  -\twoprod{\tr(\bk), q}{\Omega}& & & & &= 0 &\quad \forall q \in \mathrm{L}^2(\Omega),\\
  &-\twoprod{\bnabla \cdot \bsigma, \bv}{\Omega}& & & &= \twoprod{\rho_0 \bb, \bv}{\Omega} &\quad \forall \bv \in \mathbf{L}^2(\Omega).
\end{alignedat}  
\end{equation}
Defining the operators
\begin{align*}
  A&: \mathbb{L}^2(\Omega)\to \mathbb{L}^2(\Omega)', \quad \langle A(\bk), \bgamma \rangle \doteq \mu \twoprod{\bk,\bgamma}{\Omega},\\
  B_1&: \mathbb{L}^2(\Omega) \to [\mathbb{H}^{\bzero}(\bdiv,\Omega) \times \mathrm{L}^2(\Omega)]', \quad \langle B_1 (\bgamma),(\btau,q) \rangle \doteq -\twoprod{\bgamma, \btau}{\Omega} - \twoprod{\tr(\bgamma), q}{\Omega},\\
  B_2&: [\mathbb{H}^{\bzero}(\bdiv,\Omega) \times \mathrm{L}^2(\Omega)] \to \bL^2(\Omega)', \quad \langle B_2(\btau,q), \bv\rangle \doteq - \twoprod{\nabla \cdot \btau, \bv}{\Omega},\\
  F_1&: \mathbb{L}^2(\Omega) \to \mathbb{R}, \quad \langle F_1, \bgamma \rangle \doteq \twoprod{\mu,\tr(\bgamma)}{\Omega},\\
  F_2&: [\mathbb{H}^{\bzero}(\bdiv,\Omega) \times \mathrm{L}^2(\Omega)] \to \mathbb{R}, \quad \langle F_2, (\btau,q)\rangle \doteq -\twoprod{\bar{\bu}, \btau\bn}{\Gamma_d},\\
  F_3&: \bL^2(\Omega) \to \mathbb{R}, \quad \langle F_3, \bv\rangle \doteq  \twoprod{\rho_0 \bb, \bv}{\Omega}, 
\end{align*}
problem~\eqref{eq:weak-lin3} can be compactly written as
\begin{equation*}
  \begin{pmatrix}
    A & B_1' & 0 \\ B_1 & 0 & B_2' \\ 0 & B_2 & 0 
  \end{pmatrix}\begin{pmatrix}
    \bk \\ (\bsigma,p) \\ \bu
  \end{pmatrix} = \begin{pmatrix}
    F_1 \\ F_2 \\ F_3
  \end{pmatrix} \qquad \text{in } (\mathbb{L}^2(\Omega)\times [\mathbb{H}^{\bzero}(\bdiv,\Omega) \times \mathrm{L}^2(\Omega)]\times\bL^2(\Omega))'.
\end{equation*} 
The abstract theory for twofold saddle-point problems is stated in~\cite[Theorem 2.2]{Gatica2003} (see also~\cite{Howell2011}).

\begin{theorem}\label{th:abstract}
Denote $\bV_2 \doteq \ker(B_2)$. Assume that the following conditions hold:
\begin{itemize}
  \item The operator $A$ is linear, bounded and coercive in $\mathbb{L}^2(\Omega)$;
  \item The linear operator $B_1$ satisfies an inf-sup condition on the kernel $\bV_2$;
  \item The linear operator $B_2$ satisfies an inf-sup condition.
\end{itemize}
Then, for any given linear functionals $F_1$, $F_2$ and $F_3$, there exists a unique solution $(\bk,(\bsigma,p),\bu)\in \mathbb{L}^2(\Omega)\times [\mathbb{H}^{\bzero}(\bdiv,\Omega) \times \mathrm{L}^2(\Omega)]\times\bL^2(\Omega)$ to \eqref{eq:weak-lin3}. Furthermore, there exists a constant $C>0$ depending only on the continuity, coercivity and inf-sup constants, such that 
$$ \norm{(\bk,(\bsigma,p),\bu)} \leq C (\norm{F_1} + \norm{F_2} + \norm{F_3}). $$
\end{theorem}

We now examine the properties of the bilinear forms and linear functionals. 
The boundedness of $A(\bullet,\bullet)$, $B_1(\bullet,\bullet)$ and $B_2(\bullet,\bullet)$ follows directly from H{\"o}lder's inequality and the definitions of the norms. 
In addition, $A(\bullet,\bullet)$ is coercive on $\mathbb{L}^2(\Omega)$. It therefore remains to prove the corresponding inf-sup conditions.

\begin{lemma}\label{lem:inf-sup-B2}
There exists a constant $\beta_2>0$ such that
\begin{equation*}
  \sup_{(\btau,q)\in \mathbb{H}^{\bzero}(\bdiv,\Omega)\times \mathrm{L}^2(\Omega)\setminus\{\bzero\}}
  \frac{\langle B_2(\btau,q),\bv\rangle}{\norm{(\btau,q)}}
  \ge
  \beta_2 \norm{\bv}_{\bL^2(\Omega)}
  \qquad \forall\, \bv\in \bL^2(\Omega).
\end{equation*}
\end{lemma}

\begin{proof}
It follows by taking $q=0$ and noting that
\[
\sup_{(\btau,q)} \frac{\langle B_2(\btau,q),\bv\rangle}{\norm{(\btau,q)}}
\ge
\sup_{\btau\in \mathbb{H}^{\bzero}(\bdiv,\Omega)\setminus\{\bzero\}}
\frac{-\twoprod{\bnabla \cdot \btau, \bv}{\Omega}}{\norm{\btau}_{\bbH(\bdiv,\Omega)}}.
\]
The right-hand side is bounded from below by a standard argument based on the surjectivity of the divergence operator; we include the details for completeness.
Given $\bv\in \bL^2(\Omega)$, consider $\bz\in \bH^1(\Omega)$ by solving
\begin{equation}\label{eq:aux12}
  -\bDelta \bz = \bv \quad \text{in } \Omega,\qquad 
  \bz = \bzero \quad \text{on } \Gamma_d, \qquad 
  \nabla\bz \, \bn = \bzero \quad \text{on } \Gamma_t.
\end{equation}
Next, define $\tilde{\btau} \doteq \nabla\bz$. Then
\[
\bnabla \cdot \tilde{\btau} = -\bv \quad \text{in } \Omega,
\qquad
\tilde{\btau}\bn=\bzero \quad \text{on } \Gamma_t,
\]
so that $\tilde{\btau}\in \mathbb{H}^{\bzero}(\bdiv,\Omega)$. By Poincar{\'e}'s inequality and the standard data regularity for problem~\eqref{eq:aux12}, there exists a constant $\bar C>0$ such that
$$ \norm{\tilde{\btau}}_{\bbH(\bdiv,\Omega)} \le \bar{C} \norm{\bv}_{\bL^2(\Omega)}. $$

Therefore,
\begin{align*}
\sup_{\btau\in \mathbb{H}^{\bzero}(\bdiv,\Omega)\setminus\{\bzero\}}
\frac{-\twoprod{\bnabla \cdot \btau, \bv}{\Omega}}{\norm{\btau}_{\bbH(\bdiv,\Omega)}}
&\ge
\frac{-\twoprod{\bnabla \cdot \tilde{\btau},\bv}{\Omega}}{\norm{\tilde{\btau}}_{\bbH(\bdiv,\Omega)}}  =
\frac{\norm{\bv}_{\bL^2(\Omega)}^2}{\norm{\tilde{\btau}}_{\bbH(\bdiv,\Omega)}}
\ge
\frac{1}{\bar C}\norm{\bv}_{\bL^2(\Omega)}.
\end{align*}
Setting $\beta_2 \doteq \bar C^{-1}$ concludes the proof.
\end{proof}

Next, we have the following characterisation for the kernel of $B_2$:
\begin{align*}
\bV_2 \doteq \ker(B_2)
&= \{ (\btau,q)\in \mathbb{H}^{\bzero}(\bdiv,\Omega)\times \mathrm{L}^2(\Omega)
: \langle B_2(\btau,q),\bv\rangle = 0 \quad \forall\, \bv\in \bL^2(\Omega) \} \\
&= \{ (\btau,q)\in \mathbb{H}^{\bzero}(\bdiv,\Omega)\times \mathrm{L}^2(\Omega)
: \bnabla \cdot \btau = \bzero \text{ in } \Omega \}.
\end{align*}
Indeed, the second equality follows by choosing $\bv=\bnabla \cdot \btau\in \bL^2(\Omega)$.

We also recall the following auxiliary estimate, which is a direct consequence of~\cite[Lemmas~3.1--3.2]{Alvarez2015}: there exists a constant $\gamma_0>0$ such that
\begin{equation}\label{eq:aux1}
\norm{\btau}_{\bbH(\bdiv,\Omega)} \le
\gamma_0\bigl(\norm{\btau^{\dev}}_{\bbL^2(\Omega)}
+
\norm{\bnabla \cdot \btau}_{\bL^2(\Omega)}\bigr)
\qquad \forall\, \btau\in \mathbb{H}^{\bzero}(\bdiv,\Omega),
\end{equation}
where $\btau^{\dev} \doteq \btau - \frac{\tr(\btau)}{d}\bI$ is the deviatoric part of $\btau$ and $\tr(\btau^{\dev})=0$.

\begin{lemma}\label{lem:inf-sup-B1}
There exists a constant $\beta_1>0$ such that
\[
\sup_{\bgamma\in \mathbb{L}^2(\Omega)\setminus\{\bzero\}}
\frac{\langle B_1(\bgamma),(\btau,q)\rangle}{\norm{\bgamma}_{\bbL^2(\Omega)}}
\ge
\beta_1\bigl(\norm{\btau}_{\bbH(\bdiv,\Omega)} + \norm{q}_{\rL^2(\Omega)}\bigr)
\qquad \forall\,(\btau,q)\in \bV_2.
\]
\end{lemma}

\begin{proof}
A related proof for the 2D case with pure displacement boundary conditions can be found in~\cite{Gatica2004}. We adapt the argument to the present setting.
Consider a generic pair $(\btau,q)\in \bV_2$. By the characterisation of $\bV_2$, we know that $\btau\in \mathbb{H}^{\bzero}(\bdiv,\Omega)$, $q\in \mathrm{L}^2(\Omega)$ and $\bnabla \cdot \btau=\bzero$ in $\Omega$.

\smallskip
\noindent\emph{Case 1: $\norm{q}_{\rL^2(\Omega)}\le \norm{\btau}_{\bbH(\bdiv,\Omega)}$.}
Choose $\tilde{\bgamma} \doteq -\btau^{\dev}$. Then  
\begin{align*}
\frac{\langle B_1(\tilde{\bgamma}),(\btau,q)\rangle}{\norm{\tilde{\bgamma}}_{\bbL^2(\Omega)}} &=
\frac{\twoprod{\btau, \btau^{\dev}}{\Omega} + \twoprod{\tr(\btau^{\dev}), q}{\Omega}}{\norm{\btau^{\dev}}_{\bbL^2(\Omega)}} =
\frac{\norm{\btau^{\dev}}_{\bbL^2(\Omega)}^2 + \twoprod{\frac{\tr(\btau)}{d}\bbI, \btau^{\dev}}{\Omega}}{\norm{\btau^{\dev}}_{\bbL^2(\Omega)}} = \norm{\btau^{\dev}}_{\bbL^2(\Omega)}.
\end{align*}
Using~\eqref{eq:aux1}, $\bnabla \cdot \btau=\bzero$ and $\norm{q}_{\rL^2(\Omega)}\le \norm{\btau}_{\bbH(\bdiv,\Omega)}$ we obtain
\[
\norm{\btau^{\dev}}_{\bbL^2(\Omega)}
\ge \gamma_0^{-1}\norm{\btau}_{\bbH(\bdiv,\Omega)}
\ge \tfrac12\gamma_0^{-1}
\bigl(\norm{\btau}_{\bbH(\bdiv,\Omega)}+\norm{q}_{\rL^2(\Omega)}\bigr).
\]

\smallskip
\noindent\emph{Case 2: $\norm{q}_{\rL^2(\Omega)}\ge \norm{\btau}_{\bbH(\bdiv,\Omega)}$.}
Choose $\hat{\bgamma} \doteq \btau-q\bbI$. Then
\[
\langle B_1(\hat{\bgamma}),(\btau,q)\rangle
= -\twoprod{\btau, \btau - q\bbI}{\Omega} - \twoprod{\tr(\btau - q\bbI), q}{\Omega}
= -\norm{\btau}_{\bbL^2(\Omega)}^2 + d\norm{q}_{\rL^2(\Omega)}^2.
\]
Moreover, by the triangle inequality,
\[
\norm{\hat{\bgamma}}_{\bbL^2(\Omega)} \le \norm{\btau}_{\bbL^2(\Omega)} + d \norm{q}_{\rL^2(\Omega)}.
\]
Therefore, with $\norm{q}_{\rL^2(\Omega)}\ge \norm{\btau}_{\bbH(\bdiv,\Omega)}$,
\[
\frac{\langle B_1(\hat{\bgamma}),(\btau,q)\rangle}{\norm{\hat{\bgamma}}_{\bbL^2(\Omega)}} \ge
\frac{d-1}{d+1}\norm{q}_{\rL^2(\Omega)} \ge
\frac{d-1}{2(d+1)} \bigl(\norm{\btau}_{\bbH(\bdiv,\Omega)}+\norm{q}_{\rL^2(\Omega)}\bigr).
\]

By choosing
\[
\beta_1 \doteq \min\!\left\{ \tfrac12\gamma_0^{-1},\,\tfrac{d-1}{2(d+1)}\right\}
\]
we finish the proof.
\end{proof}

\begin{theorem}
There exists a unique 
\[
(\bk,(\bsigma,p),\bu) \in 
\mathbb{L}^2(\Omega) \times \bigl[\mathbb{H}^{\bzero}(\bdiv,\Omega) \times \mathrm{L}^2(\Omega)\bigr] \times \bL^2(\Omega)
\]
satisfying
\begin{equation}\label{eq:well-posed-elast}
\begin{alignedat}{3}
  \langle A(\bk), \bgamma \rangle &+ \langle B_1(\bgamma), (\bsigma,p)\rangle & &= F_1(\bgamma), \quad
  &\forall \bgamma \in \mathbb{L}^2(\Omega),\\
  \langle B_1(\bk), (\btau,q)\rangle & &+\,\langle B_2(\btau,q), \bu \rangle &= F_2(\btau,q), \quad
  &\forall (\btau,q) \in \mathbb{H}^{\bzero}(\bdiv,\Omega) \times \mathrm{L}^2(\Omega),\\
  &\langle B_2(\bsigma,p), \bv \rangle & &= F_3(\bv), \quad
  &\forall \bv \in \bL^2(\Omega),
\end{alignedat} 
\end{equation}
and moreover there exists a constant $C>0$ such that
\[
\norm{\bk}_{\bbL^2(\Omega)} + \norm{(\bsigma,p)}_{\mathbb{H}(\bdiv,\Omega) \times \mathrm{L}^2(\Omega)} + \norm{\bu}_{\bL^2(\Omega)} 
\le C \bigl(\norm{\rho_0 \bb}_{\bL^2(\Omega)} + \norm{\bar{\bu}}_{\bH^{1/2}(\Gamma_d)} \bigr).
\]
\end{theorem}

\begin{proof}
The result follows directly from Lemmas~\ref{lem:inf-sup-B2} and~\ref{lem:inf-sup-B1}, together with Theorem~\ref{th:abstract}.
\end{proof}

\subsection{Discrete solvability analysis for the linearised problem} \label{sec:na:discrete}
After discretisation, we select the \ac{fe} spaces described in \sect{sec:fe:mixed_fe}. The associated Galerkin scheme for~\eqref{eq:well-posed-elast} then reads:
find $(\bk_h,(\bsigma_h,p_h),\bu_h) \in \mathbb{H}_{k,h} \times [\mathbb{H}_{k,h}^{\bzero} \times \mathrm{H}_{k,h}]  \times \overline{\mathbf{H}}_{k-1,h}$ such that 
\begin{equation}\label{eq:well-posed-elast-h}
\begin{alignedat}{3}
  \langle A(\bk_h), \bgamma_h \rangle &+ \langle B_1(\bgamma_h), (\bsigma_h, p_h)\rangle & &= F_1(\bgamma_h), &\forall \bgamma_h \in \mathbb{H}_{k,h},\\
  \langle B_1(\bk_h), (\btau_h,q_h)\rangle & &+ \langle B_2(\btau_h,q_h), \bu_h \rangle &= F_2(\btau_h,q_h), &\forall (\btau_h,q_h) \in \mathbb{H}_{k,h}^{\bzero} \times \mathrm{H}_{k,h},\\
  &\langle B_2(\bsigma_h,p_h), \bv_h \rangle & &= F_3(\bv_h), &\forall \bv_h \in \overline{\mathbf{H}}_{k-1,h}.
\end{alignedat} 
\end{equation}
It is worth noting that the continuous \ac{fe} space $\mathrm{H}_{k,h}$ is a design choice.

We then state the discrete version of Theorem~\ref{th:abstract}; see also~\cite[Theorem 3.2]{Gatica2003}. 
\begin{theorem}\label{th:abstract-h}
Denote $\bV_{2h} = \ker(B_2)|_{\mathbb{H}_{k,h}^{\bzero} \times \mathrm{H}_{k,h}}$. Assume that 
\begin{itemize}
  \item The operator $A$ is linear, bounded and coercive in  $\mathbb{H}_{k,h}$;
  \item The linear operator $B_1$ satisfies a discrete inf-sup condition on the discrete kernel $\bV_{2h}$;
  \item The linear operator $B_2$ satisfies a discrete inf-sup condition.
\end{itemize}
Then there exists a unique solution $(\bk_h,(\bsigma_h,p_h),\bu_h)\in \mathbb{H}_{k,h} \times [\mathbb{H}_{k,h}^{\bzero} \times \mathrm{H}_{k,h}] \times \overline{\mathbf{H}}_{k-1,h}$ to the Galerkin method associated with~\eqref{eq:weak-lin3}. Furthermore, there exists a constant $C>0$, independent of $h$, but depending on the inf-sup constants, continuity and coercivity constants, such that 
\[ \norm{(\bk_h,(\bsigma_h, p_h),\bu_h)} \leq C \bigl(\norm{F_1} + \norm{F_2} + \norm{F_3} \bigr).\]
\end{theorem}

\begin{lemma}\label{lem:inf-sup-B1-h}
There exists a constant $\tilde{\beta_1}>0$, independent of $h$,  such that 
\[ \sup_{ \bgamma \in  \mathbb{H}_{k+1,h} \setminus \{\bzero\}} \frac{\langle B_1(\bgamma_h),(\btau_h,q_h)\rangle}{\norm{\bgamma_h}_{\bbL^2(\Omega)}}\geq \tilde{\beta_1} (\norm{\btau_h}_{\bbH(\bdiv,\Omega)} + \norm{q_h}_{\rL^2(\Omega)}) \qquad \forall (\btau_h,q_h) \in \bV_{2h}. \]
\end{lemma}
\begin{proof}
For the selected \ac{fe} spaces $\mathbb{H}_{k,h} \times [\mathbb{H}_{k,h}^{\bzero} \times \mathrm{H}_{k,h}]  \times \overline{\mathbf{H}}_{k-1,h}$, the proof of Lemma~\ref{lem:inf-sup-B1} carries over verbatim. 
Indeed, the \ac{fe} spaces for $\bk_h$ and $\bsigma_h$ coincide, and for any $q_h \in \mathrm{H}_{k,h}$, the piecewise linear and overall continuous tensor $q_h \bbI$ is contained in the \ac{bdm} space $\mathbb{H}_{k,h}$.
\end{proof}

\subsection{Error estimates for the discrete linearised problem} \label{sec:na:estimate}
We begin by introducing the following interpolation operators: 
\[
  \pi_h : \mathrm{L}^2(\Omega) \rightarrow \mathrm{H}_{k,h}, \quad
  \bar{\pi}_h : \bL^2(\Omega) \rightarrow \overline{\bH}_{k-1,h}, \quad
  \pi_h^{\bdiv} : \bbH(\bdiv, \Omega) \rightarrow \bbH_{k,h},
\]
which satisfy the following commutativity and $\bL^2(\Omega)$-orthogonality properties (see, e.g.~\cite{Ern04}):  
\begin{subequations}\label{eq:pi_properties}
\begin{align}
  \bnabla \cdot \left(\pi_h^{\bdiv}(\bsigma)\right) = \bar{\pi}_h(\bnabla \cdot \bsigma) \in \overline{\mathbf{H}}_{k-1,h} \quad &\forall \bsigma \in \bbH(\bdiv,\Omega),\label{eq:div_commu} \\
  \twoprod{\bu - \bar{\pi}_h(\bu), \bv_h}{\Omega} = 0 \quad &\forall \bv_h \in \overline{\bH}_{k-1,h}. \label{eq:l2_ortho}
\end{align}
\end{subequations}

\begin{theorem}[A priori error estimates]\label{th:na:estimate}
Let $(\bk, (\bsigma, p), \bu)$ denote the exact solution of the continuous problem~\eqref{eq:well-posed-elast}, and assume that these fields are sufficiently smooth.
Denote by $(\bk_h, (\bsigma_h, p_h), \bu_h)$ the solution of the discrete problem~\eqref{eq:well-posed-elast-h}.
Then there exists a constant $C>0$, independent of $h$, such that the errors $\varepsilon_{\bk}=\bk-\bk_h$, $\varepsilon_{\bsigma}=\bsigma-\bsigma_h$, $\varepsilon_p=p-p_h$ and $\varepsilon_{\bu}=\bu-\bu_h$ satisfy
\begin{equation} \label{eq:err_estimates}
\begin{aligned}
  &\norm{\varepsilon_{\bk}}_{\bbL^2(\Omega)} + \norm{\varepsilon_{\bsigma}}_{\bbH(\bdiv,\Omega)} + \norm{\varepsilon_p}_{\rL^2(\Omega)} + \norm{\varepsilon_{\bu}}_{\bL^2(\Omega)}\\
  \leq\, & C h^{k} \left(h\norm{\bk}_{\bbH^{k+1}(\Omega)} + \norm{\bsigma}_{\bbH^{k+1}(\Omega)} + h\norm{p}_{\rH^{k+1}(\Omega)} + \norm{\bu}_{\bH^k(\Omega)}\right).
\end{aligned}
\end{equation}
Moreover,
\begin{equation}\label{eq:kp_estimates}
\begin{aligned}
  \norm{\varepsilon_{\bk}}_{\bbL^2(\Omega)} + \norm{\varepsilon_p}_{\rL^2(\Omega)}
  \leq C h^{k+1} \left( \norm{\bk}_{\bbH^{k+1}(\Omega)} + \norm{\bsigma}_{\bbH^{k+1}(\Omega)} + \norm{p}_{\rH^{k+1}(\Omega)} \right).
\end{aligned}
\end{equation}
\end{theorem}

\begin{proof}
For each unknown, we decompose the total error $\varepsilon$ into an interpolation error and a discrete error:
\begin{align*}
  \varepsilon_{\bk} &= \bk - \bk_h = \bk - \pi_h^{\bdiv}(\bk) + (\pi_h^{\bdiv}(\bk) - \bk_h) = \eta_{\bk} + \xi_{\bk}, \\
  \varepsilon_{\bsigma} &= \bsigma - \bsigma_h = \bsigma - \pi_h^{\bdiv}(\bsigma) + (\pi_h^{\bdiv}(\bsigma) - \bsigma_h) =  \eta_{\bsigma} + \xi_{\bsigma}, \\
  \varepsilon_p &= p - p_h = p - \pi_h(p) + (\pi_h(p) - p_h) = \eta_p + \xi_p, \\
  \varepsilon_{\bu} &= \bu - \bu_h = \bu - \bar{\pi}_h(\bu) + (\bar{\pi}_h(\bu) - \bu_h) = \eta_{\bu} + \xi_{\bu}.
\end{align*}
The interpolation errors $\eta$ depend only on the regularity of the exact solution and satisfy the standard estimates:
\begin{align}\label{eq:intep_errs} 
  \norm{\eta_{\bk}}_{\bbL^2(\Omega)} &\leq C h^{k+1} \norm{\bk}_{\bbH^{k+1}(\Omega)},
  &\quad
  \norm{\eta_p}_{\rL^2(\Omega)} &\leq C h^{k+1} \norm{p}_{\rH^{k+1}(\Omega)}, \\
\nonumber   \norm{\eta_{\bsigma}}_{\bbL^2(\Omega)} &\leq C h^{k+1} \norm{\bsigma}_{\bbH^{k+1}(\Omega)},
  &\quad
  \norm{\eta_{\bu}}_{\bL^2(\Omega)} &\leq C h^{k} \norm{\bu}_{\bH^k(\Omega)},\quad 
  \norm{\eta_{\bsigma}}_{\bbH(\bdiv,\Omega)} &\leq C h^{k} \norm{\bsigma}_{\bbH^{k+1}(\Omega)}.
\end{align} 

Subtracting the discrete problem~\eqref{eq:well-posed-elast-h} from the continuous problem~\eqref{eq:well-posed-elast}, we obtain the error equations:
\begin{align*}
  \langle A(\xi_{\bk}), \bgamma_h \rangle + \langle B_1(\bgamma_h), (\xi_{\bsigma}, \xi_p)\rangle &= -\langle A(\eta_{\bk}), \bgamma_h \rangle - \langle B_1(\bgamma_h), (\eta_{\bsigma}, \eta_p)\rangle, &\forall \bgamma_h \in \mathbb{H}_{k,h},\\
  \langle B_1(\xi_{\bk}), (\btau_h,q_h)\rangle + \langle B_2(\btau_h,q_h), \xi_{\bu} \rangle &= -\langle B_1(\eta_{\bk}), (\btau_h,q_h)\rangle - \langle B_2(\btau_h,q_h), \eta_{\bu} \rangle, &\forall (\btau_h,q_h) \in \mathbb{H}_{k,h}^{\bzero} \times \mathrm{H}_{k,h},\\
  \langle B_2(\xi_{\bsigma}, \xi_p), \bv_h \rangle &= -\langle B_2(\eta_{\bsigma}, \eta_p), \bv_h \rangle, &\forall \bv_h \in \overline{\mathbf{H}}_{k-1,h}.
\end{align*}
In the second equation, the term $\langle B_2(\btau_h, q_h), \eta_{\bu} \rangle = \twoprod{\eta_{\bu}, \bnabla\cdot \btau_h}{\Omega}$ vanishes. Indeed, for any $\btau_h \in \mathbb{H}^{\bzero}_{k,h}$, we have $\bnabla \cdot \btau_h \in \overline{\mathbf{H}}_{k-1,h}$. Hence, by applying the $\bL^2(\Omega)$-orthogonality~\eqref{eq:l2_ortho}:
\[
  \twoprod{\bu - \bar{\pi}_h(\bu), \bnabla\cdot \btau_h}{\Omega} = 0.
\]
Likewise, in the third equation, the term involving $B_2$ and $(\eta_{\bsigma}, \eta_p)$ also vanishes by using the commutativity and $\bL^2(\Omega)$-orthogonality in~\eqref{eq:pi_properties}:
\[
  \langle B_2(\eta_{\bsigma}, \eta_p), \bv_h \rangle = \twoprod{\bnabla \cdot \eta_{\bsigma}, \bv_h}{\Omega} = \twoprod{\bnabla \cdot (\bsigma - \pi_h^{\bdiv}(\bsigma)), \bv_h}{\Omega} = \twoprod{\bnabla\cdot\bsigma - \bar{\pi}_h(\bnabla \cdot \bsigma), \bv_h}{\Omega} = 0.
\]

The operator $A$ is bounded:
\begin{align*}
  \langle A(\bk), \bgamma \rangle &= \mu \twoprod{\bk, \bgamma}{\Omega} \leq \mu \norm{\bk}_{\bbL^2(\Omega)} \norm{\bgamma}_{\bbL^2(\Omega)} \leq C \norm{\bk}_{\bbL^2(\Omega)} \norm{\bgamma}_{\bbL^2(\Omega)}.
\end{align*}
Similarly, the operator $B_1$ is bounded:
\begin{align*}
  \langle B_1(\bgamma), (\bsigma,p)\rangle &= -\twoprod{\bgamma, \bsigma}{\Omega} - \twoprod{\tr(\bgamma), p}{\Omega} \leq \norm{\bsigma}_{\bbL^2(\Omega)} \norm{\bgamma}_{\bbL^2(\Omega)} + \norm{p}_{\rL^2(\Omega)} \sqrt{d} \norm{\bgamma}_{\bbL^2(\Omega)}\\
  &\leq C \norm{\bgamma}_{\bbL^2(\Omega)} \left(\norm{\bsigma}_{\bbL^2(\Omega)} + \norm{p}_{\rL^2(\Omega)} \right).
\end{align*}

Invoking the stability results of Theorem~\ref{th:abstract-h}, together with the boundedness of $A$ and $B_1$, we obtain the estimate for the discrete errors:
\begin{equation}\label{eq:discrete_errs}
  \norm{\xi_{\bk}}_{\bbL^2(\Omega)} + \norm{\xi_{\bsigma}}_{\bbH(\bdiv,\Omega)} + \norm{\xi_p}_{\rL^2(\Omega)} + \norm{\xi_{\bu}}_{\bL^2(\Omega)}
  \leq C \left(\norm{\eta_{\bk}}_{\bbL^2(\Omega)} + \norm{\eta_{\bsigma}}_{\bbL^2(\Omega)} + \norm{\eta_p}_{\rL^2(\Omega)} \right).
\end{equation}

Finally, applying the triangle inequality $\norm{\varepsilon} = \norm{\eta + \xi} \leq \norm{\eta} + \norm{\xi}$ and combining the interpolation errors~\eqref{eq:intep_errs} with the discrete errors~\eqref{eq:discrete_errs} yields~\eqref{eq:err_estimates}.
Starting from~\eqref{eq:discrete_errs}, we also derive~\eqref{eq:kp_estimates}.
\end{proof}

\subsection{Error estimates for the post-processed displacement field}
As introduced in \sect{sec:fe:ddcorr}, in the linearised setting the discontinuous displacement approximation $\bu_h$ can also be enhanced by solving an additional elliptic problem. 
Assuming homogeneous displacement boundary condition for simplicity, given $\bk_h \in \bbH_{k,h}$, we define the post-processed displacement $\tilde{\bu}_h \in \bH^{\bzero}_{k+1,h}$ as the solution of
\begin{equation}\label{eq:disp_corr-h}
  \twoprod{\bnabla \tilde{\bu}_h, \bnabla \bv_h}{\Omega} = \twoprod{\bk_h, \bnabla \bv_h}{\Omega}, \qquad \forall \bv_h \in \bH^{\bzero}_{k+1,h}.
\end{equation} 
The bilinear form
\[
a(\tilde{\bu}_h,\bv_h) \doteq \twoprod{\bnabla \tilde{\bu}_h, \bnabla \bv_h}{\Omega}
\]
is continuous and coercive on $\bH^{\bzero}_{k+1,h}$; therefore, existence and uniqueness of $\tilde{\bu}_h$ follow from the Lax--Milgram theorem.

\begin{corollary}[Error estimates for the post-processed displacement field]\label{cor:na:postprocessed}
Assume that $\Omega$ is convex. Let $(\bk, (\bsigma, p), \bu)$ denote the exact solution of the continuous problem~\eqref{eq:well-posed-elast} and $(\bk_h, (\bsigma_h, p_h), \bu_h)$ the solution of the discrete problem~\eqref{eq:well-posed-elast-h}. Let $\tilde{\bu}_h\in \bH^{\bzero}_{k+1,h}$ be defined by~\eqref{eq:disp_corr-h}.
Then there exists a constant $C>0$, independent of $h$, such that
\begin{equation}\label{eq:disp_h1_estimate}
\begin{aligned}
  \norm{\bnabla (\bu-\tilde{\bu}_h)}_{\bbL^2(\Omega)}
  \leq C h^{k+1} (\norm{\bnabla \bu}_{\bbH^{k+1}(\Omega)}+\norm{\bsigma}_{\bbH^{k+1}(\Omega)} + \norm{p}_{\rH^{k+1}(\Omega)}).
\end{aligned}
\end{equation}
Moreover,
\begin{equation}\label{eq:disp_l2_estimate}
  \norm{\bu-\tilde{\bu}_h}_{\bL^2(\Omega)} \leq C h^{k+2} (\norm{\bnabla \bu}_{\bbH^{k+1}(\Omega)}+\norm{\bsigma}_{\bbH^{k+1}(\Omega)} + \norm{p}_{\rH^{k+1}(\Omega)}).
\end{equation}
\end{corollary}

\begin{proof}

Let $\bpi_h : \bL^2(\Omega) \rightarrow \bH^{\bzero}_{k+1,h}$ be a suitable interpolation operator.
We decompose the total error as
\[
  \varepsilon_{\tilde{\bu}} = \bu - \tilde{\bu}_h = \bu - \bpi_h(\bu) + \bpi_h(\bu) - \tilde{\bu}_h = \eta_{\tilde{\bu}} + \xi_{\tilde{\bu}}.
\]
The interpolation error $\eta_{\tilde{\bu}}$ admits
\[
  \norm{\eta_{\tilde{\bu}}}_{\bL^2(\Omega)} \leq C h^{k+2} \norm{\bu}_{\bH^{k+2}(\Omega)},\quad 
  \norm{\bnabla \eta_{\tilde{\bu}}}_{\bbL^2(\Omega)} \leq C h^{k+1} \norm{\bnabla \bu}_{\bbH^{k+1}(\Omega)}.
\]
Since the exact solution satisfies $\bk = \bnabla \bu$, we have
\[
  \twoprod{\bnabla \bu, \bnabla \bv_h}{\Omega} = \twoprod{\bk, \bnabla \bv_h}{\Omega}, \qquad \forall \bv_h \in \bH^{\bzero}_{k+1,h}.
\]
Subtracting~\eqref{eq:disp_corr-h} from the above equation yields
\[
  \twoprod{\bnabla \xi_{\tilde{\bu}}, \bnabla \bv_h}{\Omega} = \twoprod{\varepsilon_{\bk} - \bnabla \eta_{\tilde{\bu}}, \bnabla \bv_h}{\Omega}.
\]
Choosing $\bv_h = \xi_{\tilde{\bu}}$ and applying the Cauchy--Schwarz and triangle inequalities give
\begin{align*}
  \norm{\bnabla \xi_{\tilde{\bu}}}_{\bbL^2(\Omega)}^2 = \twoprod{\varepsilon_{\bk} - \bnabla \eta_{\tilde{\bu}}, \bnabla \xi_{\tilde{\bu}}}{\Omega}
  \leq \left(\norm{\varepsilon_{\bk}}_{\bbL^2(\Omega)} + \norm{\bnabla \eta_{\tilde{\bu}}}_{\bbL^2(\Omega)}\right) \norm{\bnabla \xi_{\tilde{\bu}}}_{\bbL^2(\Omega)}.
\end{align*}
Using the previously established estimate for $\varepsilon_{\bk}$ and $\varepsilon_p$ in~\eqref{eq:kp_estimates}, we obtain
\[
  \norm{\varepsilon_{\bk}}_{\bbL^2(\Omega)} \leq \norm{\varepsilon_{\bk}}_{\bbL^2(\Omega)} + \norm{\varepsilon_p}_{\rL^2(\Omega)} \leq C h^{k+1} \left( \norm{\bk}_{\bbH^{k+1}(\Omega)} + \norm{\bsigma}_{\bbH^{k+1}(\Omega)} + \norm{p}_{\rH^{k+1}(\Omega)} \right).
\]
Recalling that $\bk = \bnabla \bu$, we obtain~\eqref{eq:disp_h1_estimate}.
We now derive the optimal $\mathbf{L}^2(\Omega)$ error estimate for the post-processed displacement.
Consider the dual problem: find $\bz \in \bH^1_{\bzero}(\Omega)$ such that
\begin{equation}\label{eq:dual_problem}
  -\bDelta \bz = \varepsilon_{\tilde{\bu}} \quad \text{in } \Omega, \qquad
  \bz = \bzero \quad \text{on } \partial\Omega.
\end{equation}
Assuming that $\Omega$ is convex, elliptic regularity gives
\[
  \norm{\bz}_{\bH^2(\Omega)} \leq C \norm{\varepsilon_{\tilde{\bu}}}_{\bL^2(\Omega)}.
\]
Using~\eqref{eq:dual_problem} and integrating by parts, we achieve
\[
  \norm{\varepsilon_{\tilde{\bu}}}_{\bL^2(\Omega)}^2 
  = \twoprod{\varepsilon_{\tilde{\bu}},\varepsilon_{\tilde{\bu}}}{\Omega}
  = -\twoprod{\varepsilon_{\tilde{\bu}},\bDelta \bz}{\Omega}
  = \twoprod{\bnabla \varepsilon_{\tilde{\bu}}, \bnabla \bz}{\Omega}.
\]
Let $\bz_h \doteq \bpi_h(\bz) \in \bH^{\bzero}_{k+1,h}$ be an interpolant of $\bz$ and it comes with the property
\[
  \norm{\bnabla \bz - \bnabla \bz_h}_{\bbL^2(\Omega)} \leq C h \norm{\bz}_{\bH^2(\Omega)}.
\]
We then decompose $\bz = (\bz - \bz_h) + \bz_h$ and obtain
\begin{equation*}
  \norm{\varepsilon_{\tilde{\bu}}}_{\bL^2(\Omega)}^2 
  = \twoprod{\bnabla \varepsilon_{\tilde{\bu}}, \bnabla \bz - \bnabla \bz_h}{\Omega} + \twoprod{\bnabla \varepsilon_{\tilde{\bu}}, \bnabla \bz_h}{\Omega}.
\end{equation*}
The term $\twoprod{\bnabla \varepsilon_{\tilde{\bu}}, \bnabla \bz_h}{\Omega}$ disappears due to Galerkin orthogonality. By the Cauchy--Schwarz inequality, the interpolation property and the elliptic regularity, we reach
\begin{align*}
  \norm{\varepsilon_{\tilde{\bu}}}_{\bL^2(\Omega)}^2 &\leq \norm{\bnabla \varepsilon_{\tilde{\bu}}}_{\bbL^2(\Omega)} \norm{\bnabla \bz - \bnabla \bz_h}_{\bbL^2(\Omega)}
  \leq C h \norm{\bnabla \varepsilon_{\tilde{\bu}}}_{\bbL^2(\Omega)}\norm{\bz}_{\bH^2(\Omega)}\\
  &\leq C h \norm{\bnabla \varepsilon_{\tilde{\bu}}}_{\bbL^2(\Omega)} \norm{\varepsilon_{\tilde{\bu}}}_{\bL^2(\Omega)}.
\end{align*}
Cancelling $\norm{\varepsilon_{\tilde{\bu}}}_{\bL^2(\Omega)}$ from both sides and invoking the estimate in~\eqref{eq:disp_h1_estimate}, we obtain~\eqref{eq:disp_l2_estimate}.
\end{proof}

\section{Numerical experiments} \label{sec:exps}
We evaluate our method through extensive numerical experiments in both 2D and 3D. For brevity, the method is referred to as \ac{ddfem}. We consider two \ac{ddfem} stable pairs, \pairone{} and \pairtwo{}, defined in \sect{sec:fe:fe}. When the displacement approximation is postprocessed as described in \sect{sec:fe:ddcorr}, the tag ``(corr)'' is appended to the corresponding pairs.

\subsection{Compatible-strain mixed finite element methods}
To demonstrate the accuracy of \acp{ddfem}, we compare \acp{ddfem} with \acp{csfem}, which are also four-field formulations for incompressible nonlinear elasticity. At the continuous level, the weak formulation of the 3D \ac{csfem}~\cite{Shojaei2019} reads: find $(\bu, \bK, \bP, p) \in \bH^1_{\bar{\bu}}(\Omega) \times \bbH(\bcurl,\Omega) \times \bbH(\bdiv,\Omega) \times \mathrm{L}^2(\Omega)$ such that 
\begin{align*}
  \twoprod{\bP,\bnabla \bv}{\Omega} + \alpha s_1(\bu, \bK, \bv) &= \twoprod{\rho_0 \bb,\bv}{\Omega} + \twoprod{\bar{\bt}, \bv}{\Gamma_t}, &&\forall \bv \in \bH^1_{\bzero}(\Omega), \\
  \twoprod{\widetilde{\bP}(\bK), \bkappa}{\Omega} + \twoprod{p\bQ(\bK),\bkappa}{\Omega} - \twoprod{\bP,\bkappa}{\Omega} + \alpha s_2(\bu, \bK, \bkappa) &= 0, &&\forall \bkappa \in \bbH(\bcurl, \Omega),\\
  \twoprod{\bnabla \bu, \bpsi}{\Omega} - \twoprod{\bK, \bpsi}{\Omega} &= 0, &&\forall \bpsi \in \bbH(\bdiv, \Omega), \\
  \twoprod{C(J),q}{\Omega} &=0, &&\forall q \in \mathrm{L}^2(\Omega).
\end{align*} 
The spaces $\bH^1_{\bar{\bu}}(\Omega)$ and $\bH^1_{\bzero}(\Omega)$ denote subspaces of $\bH^1(\Omega)$ consisting of displacement fields that satisfy prescribed and homogeneous displacement boundary conditions on $\Gamma_d$, respectively.
The stabilisation terms
\begin{equation*}
  s_1(\bu,\bK,\bv) \doteq \twoprod{\bnabla \bu - \bK, \bnabla \bv}{\Omega},\quad
  s_2(\bu,\bK,\bkappa) \doteq \twoprod{\bK - \bnabla \bu, \bkappa}{\Omega},
\end{equation*}
are weighted by the parameter $\alpha$. For 2D problems, such stabilisations are unnecessary, so $s_1$ and $s_2$ can be safely omitted~\cite{Shojaei2018}. 

Regarding feasible \ac{csfem} pairs, \pairthree{} and \pairfour{} are the best-performing first- and second-order choices in 2D~\cite{Shojaei2018}. In 3D, however, neither pair is admissible, and \pairfive{} was proposed in~\cite{Shojaei2019}. 
The notation for these \ac{csfem} pairs follows that of the \ac{ddfem} pairs and is detailed in \tab{tab:fesymbols}.
In our experiments, we use \pairthree{} and \pairfour{} for 2D examples and \pairfive{} with stabilisation parameter $\alpha = 10^5$ for 3D simulations. 

In addition to the choice of stable pairs, the principal differences between the \ac{fe} formulations of \acp{ddfem} and \acp{csfem} are summarised as follows:
\begin{itemize}
  \item \Acp{ddfem} employ discontinuous displacements, whereas \acp{csfem} use continuous displacement fields.
  \item \Acp{ddfem} utilise $\bdiv$-conforming displacement gradients, while \acp{csfem} adopt $\bcurl$-conforming displacement gradients.
  \item \Acp{ddfem} require continuous pressure fields, whereas \acp{csfem} can have discontinuous pressure fields.
  \item In \acp{ddfem} without displacement correction, displacement boundary conditions are imposed weakly and traction boundary conditions are imposed strongly; the opposite holds for \acp{csfem}.
\end{itemize}

\subsection{Examples}
Here we present a detailed comparison between \acp{ddfem} and \acp{csfem}. 
For simplicity, all experiments employ neo-Hookean materials with Lam{\'e} parameter $\mu = 1$, assuming no body forces, i.e., $\bb=\bzero$. 
Possible choices of incompressibility constraint functions $C$ are listed in~\eqref{eq:C}, with the corresponding $\bQ$ functions defined in~\eqref{eq:Qs}.
Depending on the specific problem, either $C_1$ or $C_2$, or both, are used. 
Both formulations are implemented in the open-source \ac{fem} framework \texttt{Gridap.jl}~\cite{Badia2020,Verdugo2022}, and the resulting nonlinear \ac{fe} systems are solved using Newton--Raphson's method with a Hager--Zhang line search and a tolerance of $10^{-9}$. 

\subsubsection{Radial inflation} \label{sec:exps:inf}
As an initial benchmark, we consider the radial inflation of a cylindrical shell in 2D and a hollow spherical ball in 3D, following the setups used in \acp{csfem}~\cite{Shojaei2018,Shojaei2019}.
The geometries and the prescribed displacement boundary conditions are illustrated in \fig{fig:inf_geo}. The inner and outer radii are $R_{\mathrm{in}}=0.5$ and $R_{\mathrm{out}}=1$, respectively. 
Displacement boundary conditions are imposed on the outer boundary for the shell and on the inner boundary for the ball, while all remaining boundaries are traction-free. In both the 2D and 3D cases, the displacement boundary data can be represented by the form 
\[
  \bar{\bu}=(\lambda - 1)\bX,
\]
where $\bX$ is the referential position vector and $\lambda \geq 1$ is the displacement intensity parameter.

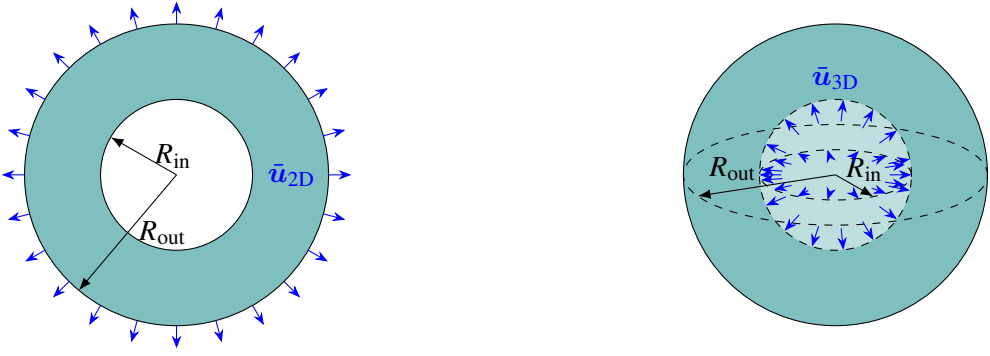
\begin{figure}[hbt!]
  \begin{subfigure}[h]{0.49\textwidth}
    \centering
    \begin{tikzpicture}[scale=2.]
      \def\Rin{0.5}
      \def\Rout{2*\Rin}

      \draw[fill=teal!50] (0,0) circle (\Rout);
      \draw[fill=white]   (0,0) circle (\Rin);

      \foreach \theta in {0,15,...,345}{
          \draw[blue,-Stealth]
          ({\Rout*cos(\theta)},{\Rout*sin(\theta)}) --
          ({1.15*\Rout*cos(\theta)},{1.15*\Rout*sin(\theta)});
        }
      \node[left] at (0.98*\Rout,0) {$\textcolor{blue}{\bar{\bu}_{\mathrm{2D}}}$};

      \draw[-{Latex}] (0,0) -- ({-\Rin*cos(30)}, {\Rin*sin(30)}) node[midway, right] {$R_{\mathrm{in}}$};
      \draw[-{Latex}] (0,0) -- (-{\Rout*cos(50)}, -{\Rout*sin(50)}) node[midway, right] {$R_{\mathrm{out}}$};
    \end{tikzpicture}
  \end{subfigure}
  \hfill
  \begin{subfigure}[h]{0.49\textwidth}
    \centering
    \begin{tikzpicture}[scale=2.]
      \def\Rin{0.5}
      \def\Rout{2*\Rin}

      \draw[fill=teal!50] (0,0) circle (\Rout);
      \draw[dashed,fill=teal!25] (0,0) circle (\Rin);
      \draw[dashed] (0,0) ellipse (\Rin cm and \Rin cm/3);
      \draw[dashed] (0,0) ellipse (\Rout cm and \Rout cm/3);

      \draw[-{Latex}] (0,0) -- node[near end,above] {$R_{\mathrm{in}}$} ({\Rin*cos(60)}, {-\Rin/3*sin(60)});
      \draw[-{Latex}] (0,0) -- node[near end,above]{$R_{\mathrm{out}}$} ({-\Rout*cos(25)}, {-\Rout/3*sin(25)});

      \foreach \angle in {12, 36, ..., 348} {
          \draw[blue,-Stealth]
          ({cos(\angle)*\Rin/1.4}, {sin(\angle)*\Rin/1.4}) --
          ({cos(\angle)*\Rin}, {sin(\angle)*\Rin});
        }
      \foreach \angle in {0, 24, ..., 336} {
          \draw[blue,-Stealth]
          ({cos(\angle)*\Rin/1.4}, {sin(\angle)*\Rin/4.2}) --
          ({cos(\angle)*\Rin}, {sin(\angle)*\Rin/3});
        }
      \node[above] at (0,\Rin) {$\textcolor{blue}{\bar{\bu}_{\mathrm{3D}}}$};
    \end{tikzpicture}
  \end{subfigure}

  \caption{Geometries and displacement boundary conditions for the inflation problems. Left: a 2D cylindrical shell; right: a 3D hollow spherical ball.}
  \label{fig:inf_geo}
\end{figure}

Let $d \in \{2,3\}$ denote the spatial dimension and define $R=\norm{\bX}$. The exact radial map is given by $r(R)=(R^d + (\lambda^d - 1)R^d_{\mathrm{out}})^{1/d}$, and in 3D we additionally define $g(R)=R(3r^3(R)+(\lambda^3-1)R^3_{\mathrm{out}})/r^4(R)$. The exact displacement field is then expressed as
\[
  \bu(\bX) = \left( \frac{r(R)}{R} - 1 \right) \bX,
\]
and the exact pressure reads
\[
  p(\bX) = 
    \begin{cases}
      -\mu\frac{R^2}{r^2(R)} + \frac{\mu(\lambda^2 - 1)R^2_{\mathrm{out}}}{2} \left(\frac{1}{r^2(R_{\mathrm{in}})} - \frac{1}{r^2(R)}\right) + \mu \ln\left(\frac{r(R_{\mathrm{in}}) R}{r(R) R_{\mathrm{in}}}\right) & \text{if $d=2$},\\
      -\mu\frac{R^4_{\mathrm{out}}}{r^4(R_{\mathrm{out}})} + \frac{\mu}{2}[g(R) - g(R_{\mathrm{out}})] & \text{if $d=3$}.
    \end{cases}
\]
We impose the incompressibility constraint using the function $C_1(J)=J-1$.
The exact displacement gradient is computed from $\bK = \bnabla \bu + \bbI$ and the exact first Piola--Kirchhoff stress is then given by $\bP = \tilde{\bP}(\bK) + p \bQ_1(\bK)$, where $\tilde{\bP}$ and $\bQ_1$ are defined in~\eqref{eq:tildeP} and~\eqref{eq:Q1}, respectively.
Due to symmetry, we model only one quarter of the 2D shell and one eighth of the 3D ball to reduce computational cost. The Cartesian mesh on $[0,0.5]^2$ is mapped to the upper-right quarter of the shell, while the one-eighth domain of the ball is discretised directly using Delaunay meshes. 

We set $\lambda=3$ and compare the \ac{ddfem} solutions directly against the exact ones; the corresponding error convergence curves are plotted in \fig{fig:inf_convs}, with reference slopes provided in the figure caption.
We first focus on the 2D convergences in \fig{fig:inf2d_conv}. For both \ac{ddfem} pairs \pairone{} and \pairtwo{}, the discontinuous displacement approximations converge at optimal rates. After postprocessing, the displacement errors reduce significantly while preserving the expected rates of the underlying continuous Lagrange elements, confirming both the accuracy of discontinuous displacement solutions and the effectiveness of the correction.
The displacement gradient also attains the expected convergence for both pairs.
The $\bbH(\bdiv,\Omega)$ errors, displayed in the third subfigure of \fig{fig:inf2d_conv}, exhibit super-convergence: for both pairs, the observed rates are one order higher than expected. This occurs because the body force is zero, i.e., $\bm{b}=\bzero$, implying $\bnabla \cdot \bP = \bzero$. Thus, the $\bbH(\bdiv,\Omega)$ error of $\bP_h$ coincides with its $\bbL^2(\Omega)$ error.
For the pressure, \pairone{} achieves the expected convergence rate while \pairtwo{} shows super-convergence, with a slope higher than the nominal order of three.

\begin{figure}[hbt!]
  \centering
  \begin{subfigure}{\textwidth}
    \centering
    \includegraphics[width=\textwidth]{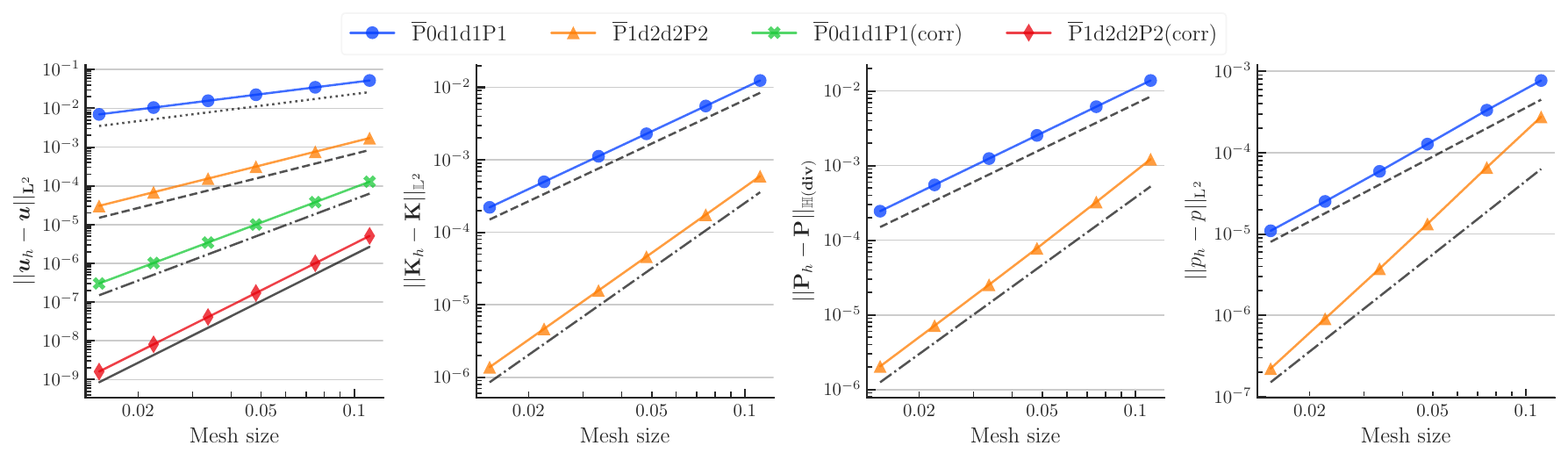}
    \caption{}
    \label{fig:inf2d_conv}
  \end{subfigure}
  \begin{subfigure}{\textwidth}
    \centering
    \includegraphics[width=\textwidth]{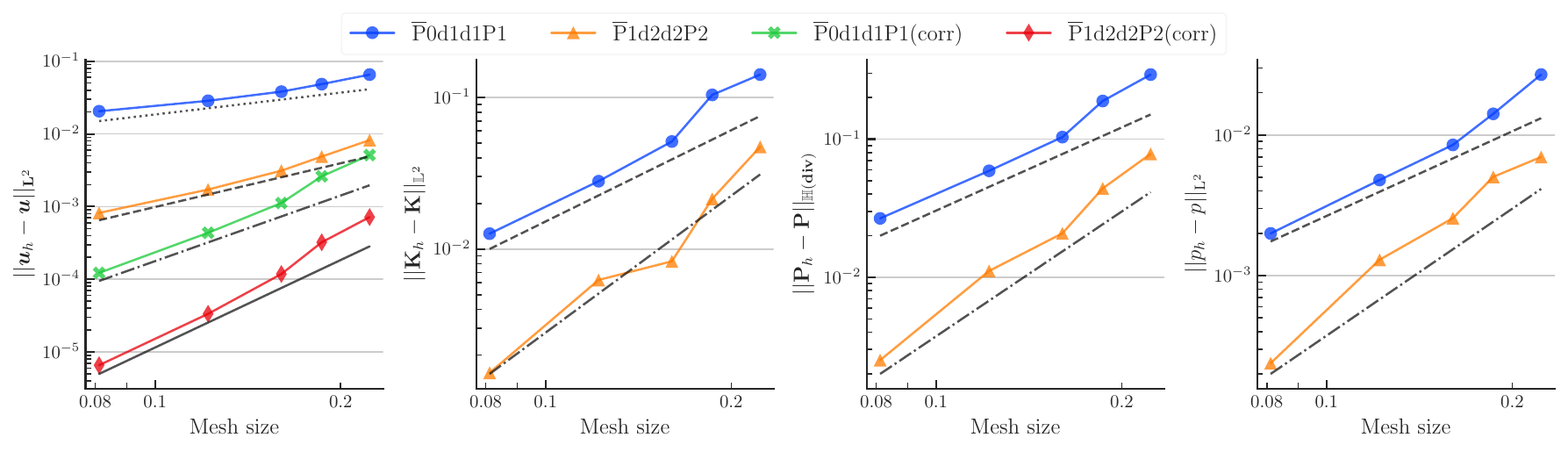}
    \caption{}
    \label{fig:inf3d_conv}
  \end{subfigure}
  \caption{Error convergence of the \ac{fe} solution errors versus mesh size for various \ac{fe} pairs in typical inflation problems. The displacement boundary data parameter is taken as $\lambda=3$. Reference slopes: 1 (dotted lines), 2 (dashed lines), 3 (dot-dashed lines), 4 (solid lines). Panels: (a) 2D results, (b) 3D results.}
  \label{fig:inf_convs}
\end{figure}

For the 3D convergence results in \fig{fig:inf3d_conv}, the observations closely mirror those in 2D. Convergence curves are slightly noisier, reflecting the added difficulty of discretising the domain with quasi-uniform 3D Delaunay meshes.
It is worth noting that the displacement correction remains effective in 3D, producing more accurate approximations while maintaining the expected super-convergence rates.
Overall, \fig{fig:inf_convs} shows that both \ac{ddfem} pairs deliver accurate solutions for the 2D radial inflation problem. More importantly, \acp{ddfem} extend robustly to 3D cases without any stabilisation or modification of standard \ac{fe} bases, whereas these are necessary in 3D \acp{csfem}~\cite{Shojaei2019}.
The convergences in \fig{fig:inf_convs} also show that \ac{ddfem} achieves the expected error estimates from~\eqref{eq:err_estimates},~\eqref{eq:kp_estimates} and~\eqref{eq:disp_l2_estimate} for all four fields plus the corrected displacement, and in some cases even surpasses them.

Figure~\ref{fig:inf2d_csfem_conv} displays the convergence results for the 2D inflation problem using the \ac{csfem} pairs \pairthree{} and \pairfour{}, computed on the same Cartesian meshes as in \fig{fig:inf2d_conv}.
Both pairs achieve the expected $\bH^1$ and $\mathrm{L}^2$ convergence rates for displacement and pressure, respectively. 
However, the displacement gradient converges one order suboptimally: order 1 instead of 2 for \pairthree{} and order 2 instead of 3 for \pairfour{}. The stress approximation is even less robust: \pairthree{} fails to reach first-order convergence, while \pairfour{} stagnates on coarse meshes before eventually recovering the expected rate.
Compared with \acp{ddfem}, \acp{csfem} are more sensitive to mesh quality. In particular, on Delaunay meshes, as reported in~\cite{Shojaei2018}, the displacement gradient convergence for \pairthree{} is slightly improved. In contrast, \ac{ddfem} pairs achieve expected convergence rates on both Cartesian and Delaunay meshes, as reflected in \fig{fig:inf_convs}.

\begin{figure}[hbt!]
  \centering
  \includegraphics[width=\textwidth]{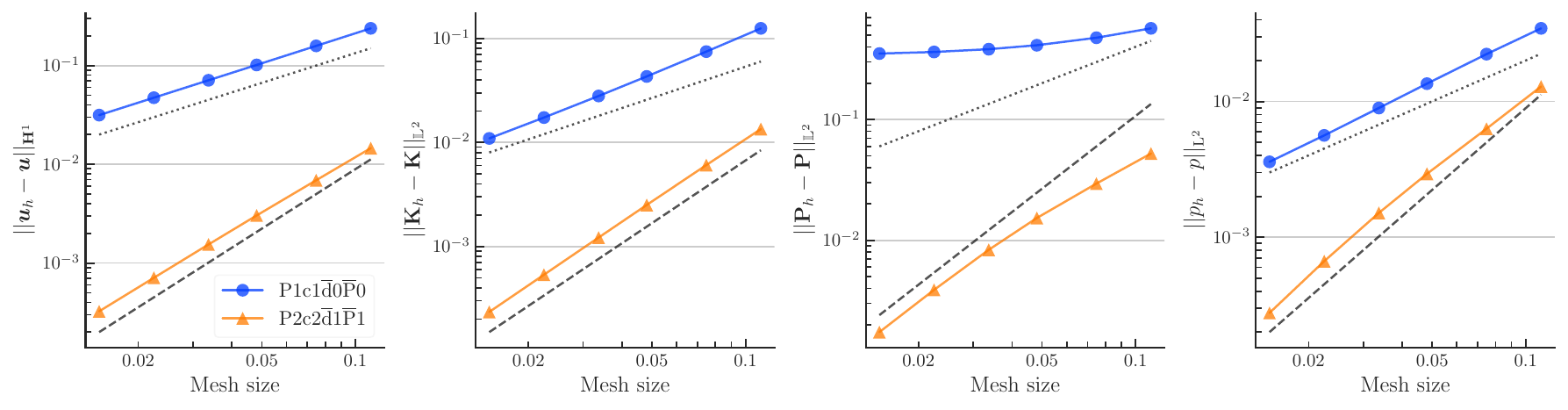}
  \caption{Error convergence of the \ac{csfem} solution errors versus mesh size in the 2D inflation problem. The displacement boundary intensity parameter is $\lambda=3$. Reference slopes: 1 (dotted lines), 2 (dashed lines).}
  \label{fig:inf2d_csfem_conv}
\end{figure}

\subsubsection{Cook's membrane} \label{sec:exps:mem}
The Cook's membrane problem is a classic benchmark for evaluating the performance of \acp{fem} in capturing combined bending and shear responses. We consider the 2D and 3D geometries shown in \fig{fig:mem_geo}, taken from~\cite{Shojaei2018} and~\cite{Shojaei2019}, respectively. The 3D membrane is formed by extruding the 2D membrane by 10 units along the $z$-direction. 
In 2D, the membrane is discretised structurally using $2\times n^2$ triangular elements. For the 3D case, we exploit symmetry of the domain and problem configuration and use Delaunay meshes to discretise only half of the membrane along the $z$-direction.
Regarding the boundary conditions, the left edge (2D) and face (3D) are clamped, while the right edge (2D) and face (3D) are subjected to, respectively, the following vertical tractions:
\[
  \bar{\bt}_{\mathrm{2D}} = (0, f)^{\tt t}, \quad \bar{\bt}_{\mathrm{3D}} = (0, f, 0)^{\tt t}.
\]
All other boundaries are traction-free. Due to the abrupt change of boundary condition types, singularities appear at the top-left vertex in 2D and along the corresponding edge in 3D.
As recommended in~\cite{Shojaei2018}, we employ $C_2$~\eqref{eq:C2} as the constraint function for both \acp{csfem} and \acp{ddfem}.

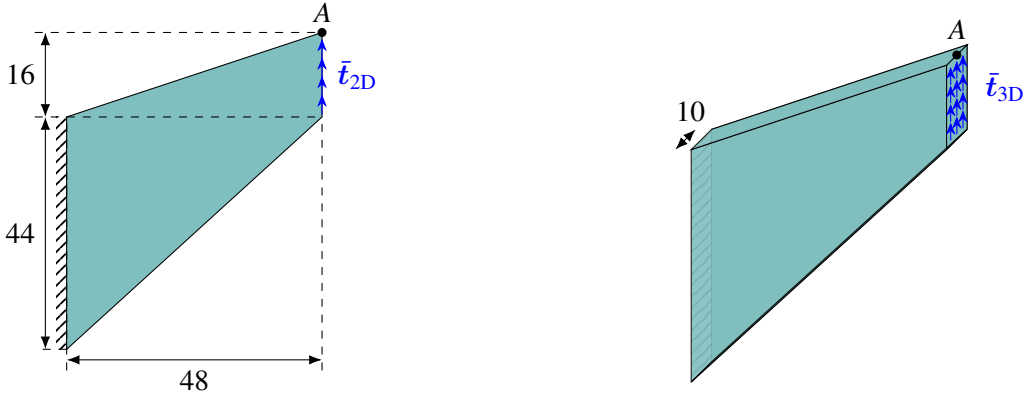
\begin{figure}[hbt!]
  \begin{subfigure}[h]{0.49\textwidth}
    \centering
    \begin{tikzpicture}[scale=0.07]
      \draw[fill=teal!50] (0,0) -- (48,44) -- (48,60) -- (0,44) -- cycle;
      \draw[pattern={Lines[distance=3,angle=45,line width=0.7]},draw=none] (-2,0) rectangle (0, 44);
      \draw[dashed] (-6,44) -- (48,44);
      \draw[dashed] (-6,60) -- (48,60);
      \draw[dashed] (48,-4) -- (48,44);
      \draw[dashed] (0,-4) -- (0,44);
      \draw[dashed] (0,0) -- (-6,0);
      \draw[fill] (48,60) circle [radius=0.75] node[above] {$A$};

      \draw[Latex-Latex] (-4,0) -- (-4,44) node[midway, left] {44};
      \draw[Latex-Latex] (-4,44) -- (-4,60) node[midway, left] {16};
      \draw[Latex-Latex] (0,-2) -- (48,-2) node[midway, below] {48};
      \foreach \y in {44.5, 48, 51.5, 55} {
          \draw[-Stealth,color=blue] (48,\y) -- (48,\y+3.8);
        }
      \node[right] at (49,52) {$\textcolor{blue}{\bar{\bt}_{\mathrm{2D}}}$};
    \end{tikzpicture}
  \end{subfigure}
  \hfill
  \begin{subfigure}[h]{0.49\textwidth}
    \centering
    \begin{tikzpicture}[scale=0.07]
      \draw[fill=teal!50, opacity=0.9] (0,0,0) -- (48,44,0) -- (48,60,0) -- (0,44,0) -- cycle;
      \draw[fill=teal!50, opacity=0.9] (0,0,10) -- (48,44,10) -- (48,60,10) -- (0,44,10) -- cycle;
      \draw[fill=teal!50, opacity=0.9] (48,44,10) -- (48,44,0) -- (48,60,0) -- (48,60,10) -- cycle;
      \draw[fill=teal!50, opacity=0.9] (48,60,10) -- (48,60,0) -- (0,44,0)  -- (0,44,10) -- cycle;
      \draw[fill=teal!50, opacity=0.9] (0,0,0) -- (0,0,10) -- (48,44,10)  -- (48,44,0) -- cycle;
      \draw[pattern={Lines[distance=3,angle=45,line width=0.7]},opacity=0.3,draw=none] (0,0,0) -- (0,0,10) -- (0,44,10) -- (0,44,0) -- cycle;
      \draw[fill] (48,60,5) circle [radius=0.75] node[above=2pt] {$A$};
      \draw[Latex-Latex] (-3,44,10) -- (-3,44,0) node[near end,above=1pt] {10};
      \foreach \y in {45, 48.5, 52, 55.5} {
        \foreach \z in {2, 5, 8} {
          \draw[-Stealth,color=blue] (48,\y,\z) -- (48,\y+3.4,\z);
          }
        }
      \node[right] at (49.5,52,0) {$\textcolor{blue}{\bar{\bt}_{\mathrm{3D}}}$};
    \end{tikzpicture}
  \end{subfigure}

  \caption{Geometries and traction boundary conditions for the Cook's membrane problems. Left: a 2D membrane; right: a 3D membrane.}
  \label{fig:mem_geo}
\end{figure}

Table~\ref{tab:mem2d_disp} reports the tip deflection at point~$A$ for the 2D Cook's membrane. The first two rows correspond to \ac{ddfem} pairs, followed by their displacement-corrected counterparts, while the last two rows correspond to the \ac{csfem} pairs.
For a very coarse mesh ($n=6$), the deflection from \pairone{} shows a slight deviation, but it converges steadily to high-order solutions as the mesh is refined. After applying the displacement correction in \sect{sec:fe:ddcorr}, the \pairone{}(corr) results closely match those of \pairthree{} and \pairfour{}.
The high-order \ac{ddfem} pair \pairtwo{} produces consistent deflections that agree well with the \ac{csfem} results, and its corrected values, denoted by \pairtwo{}(corr), show little further improvement.
Notably, on the finest mesh ($n=48$), the high-order \ac{csfem} pair \pairfour{} fails to converge for large traction ($f=0.4$), which is consistent with the observations reported in~\cite{Fu2025}.

\begin{table}[hbt!]
  \centering
{\small  \begin{tabular}{C{2.5cm} C{0.7cm} C{2.3cm} C{2.3cm} C{0.7cm} C{2.3cm} C{2.3cm}}
    \toprule
    \multirow{2}{*}{Pair}
      & \multicolumn{3}{c}{$f = 0.2$}
      & \multicolumn{3}{c}{$f = 0.4$} \\
    \cmidrule(lr){2-4} \cmidrule(lr){5-7}
      & $n$ & $\bu_x(\bX_A)$ & $\bu_y(\bX_A)$
      & $n$ & $\bu_x(\bX_A)$ & $\bu_y(\bX_A)$ \\
    \midrule
    \multirow{4}{*}{\pairone{}}
      &  6 & -12.0410 & 12.9525 &  6 & -20.0804 & 20.0880 \\
      & 12 & -12.8011 & 13.5491 & 12 & -21.2474 & 20.8351 \\
      & 24 & -13.2271 & 13.8653 & 24 & -21.9595 & 21.2735 \\
      & 48 & -13.4641 & 14.0320 & 48 & -22.3480 & 21.4985 \\
    \midrule

    \multirow{4}{*}{\pairtwo{}}
      &  6 & -13.6393 & 14.1568 &  6 & -22.5085 & 21.6388 \\
      & 12 & -13.6878 & 14.1896 & 12 & -22.5591 & 21.6425 \\
      & 24 & -13.7104 & 14.1988 & 24 & -22.8951 & 21.8188 \\
      & 48 & -13.7249 & 14.2044 & 48 & -22.6051 & 21.6443 \\
    \midrule

    \multirow{4}{*}{\pairone{}(corr)}
      &  6 & -14.2161 & 14.1681 &  6 & -22.8054 & 21.6237 \\
      & 12 & -13.9479 & 14.1875 & 12 & -22.6705 & 21.6312 \\
      & 24 & -13.8282 & 14.1967 & 24 & -22.6967 & 21.6829 \\
      & 48 & -13.7775 & 14.2027 & 48 & -22.7294 & 21.7081 \\
    \midrule

    \multirow{4}{*}{\pairtwo{}(corr)}
      &  6 & -14.0563 & 14.1492 &  6 & -22.7314 & 21.6333 \\
      & 12 & -13.8711 & 14.1790 & 12 & -22.6474 & 21.6390 \\
      & 24 & -13.7904 & 14.1915 & 24 & -22.9333 & 21.8209 \\
      & 48 & -13.7598 & 14.2002 & 48 & -22.6209 & 21.6438 \\
    \midrule

    \multirow{4}{*}{\pairthree{}}
      &  6 & -13.4500 & 14.0621 &  6 & -22.2027 & 21.2951 \\
      & 12 & -13.5719 & 14.1252 & 12 & -22.2393 & 21.3996 \\
      & 24 & -13.6271 & 14.1524 & 24 & -22.2966 & 21.4585 \\
      & 48 & -13.6620 & 14.1651 & 48 & -22.3751 & 21.5034 \\
    \midrule

    \multirow{4}{*}{\pairfour{}}
      &  6 & -13.5575 & 14.1325 &  6 & -22.2118 & 21.4396 \\
      & 12 & -13.6271 & 14.1506 & 12 & -22.2911 & 21.4642 \\
      & 24 & -13.6635 & 14.1625 & 24 & -22.3571 & 21.4964 \\
      & 48 & -13.6846 & 14.1716 & 48 & Failed & Failed \\
    \bottomrule
  \end{tabular}}
  \caption{Deflection of the tip $A$ for the 2D Cook's membrane problem under different traction boundary data $f$ for various element pairs.}
  \label{tab:mem2d_disp}
\end{table}

The \ac{fe} solutions on the deformed geometry are presented in \fig{fig:mem2d_results}, computed on a mesh of $2\times 24^2$ triangles. The top and bottom rows correspond to the high-order pairs \pairtwo{}(corr) and \pairfour{}, respectively. Overall, the solution patterns produced by \pairtwo{}(corr) closely resemble those of \pairfour{}, with only minor differences.
Several notable nuances are revealed after closer inspection.
Firstly, near the singular vertex, \pairtwo{} produces more distorted and rotated elements than \pairfour{}, a behaviour that is physically more consistent with the expected large deformation in this region.
Secondly, in the stress solutions (second column of \fig{fig:mem2d_results}), \pairfour{} exhibits visible checkerboarding around the singular vertex, despite not employing the lowest-order stress and pressure \ac{fe} pair. In contrast, the \pairtwo{} stress solution shows no such numerical artefacts.
Lastly, the overall stress solution obtained with \pairtwo{} appears smoother and more consistent with the expected stress distribution, suggesting a more natural and robust approximation than that provided by \pairfour{}.  

\begin{figure}[hbt!]
  \centering
    \begin{subfigure}{0.32\textwidth}
    \centering
    \includegraphics[width=\textwidth]{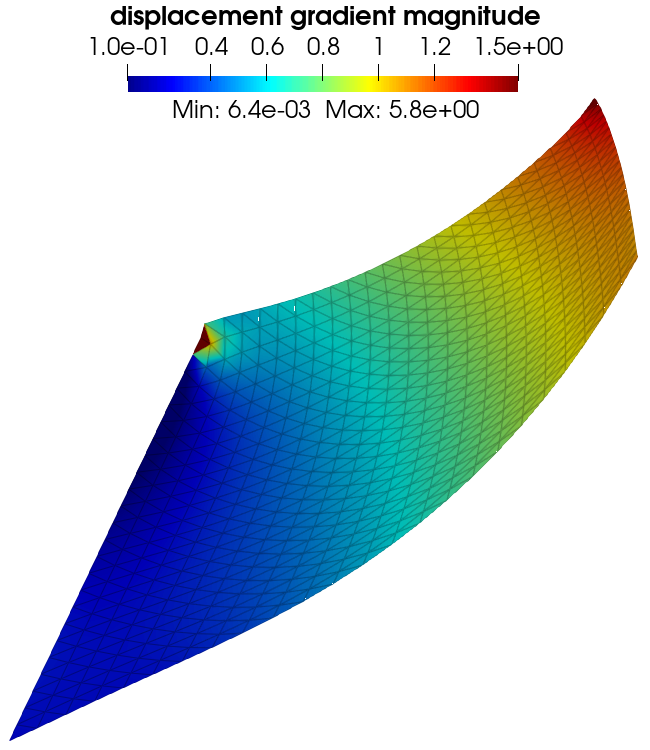}    
  \end{subfigure}
  \begin{subfigure}{0.32\textwidth}
    \centering
    \includegraphics[width=\textwidth]{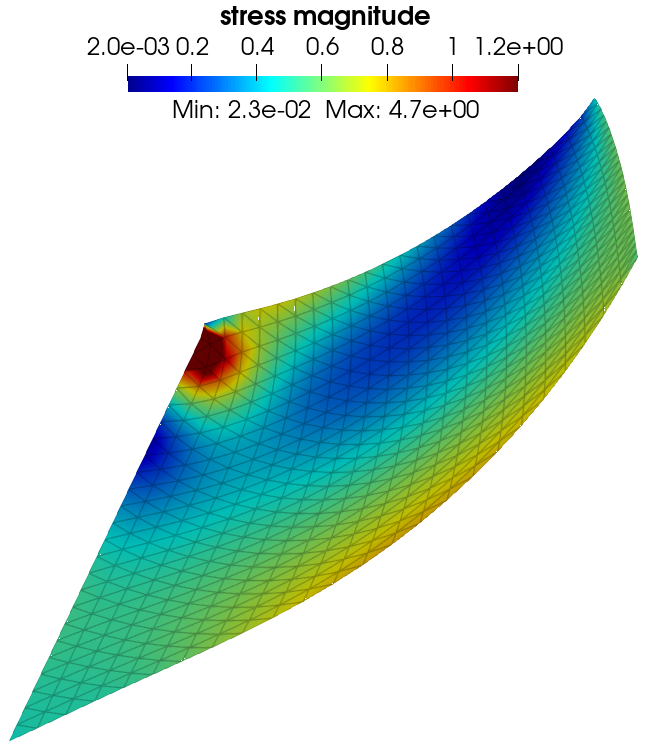}    
  \end{subfigure}
  \begin{subfigure}{0.32\textwidth}
    \centering
    \includegraphics[width=\textwidth]{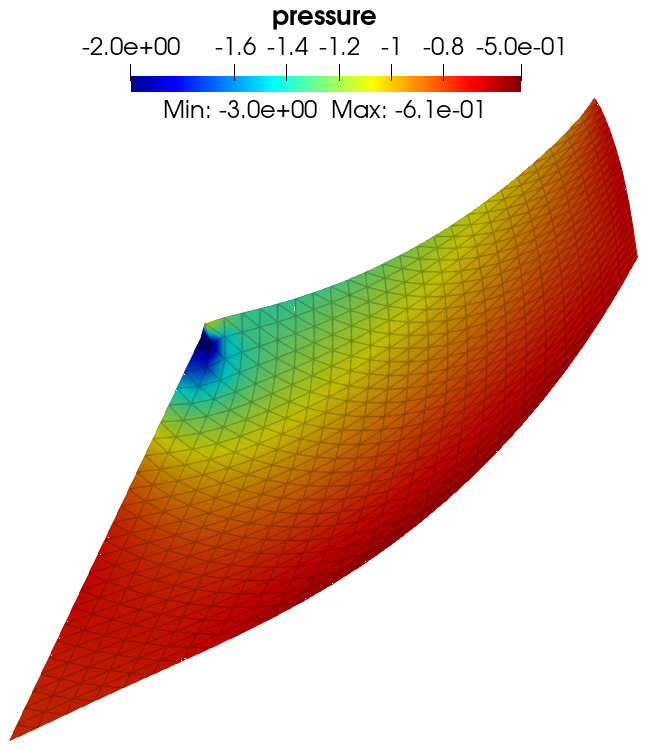}    
  \end{subfigure}

  \begin{subfigure}{0.32\textwidth}
    \centering
    \includegraphics[width=\textwidth]{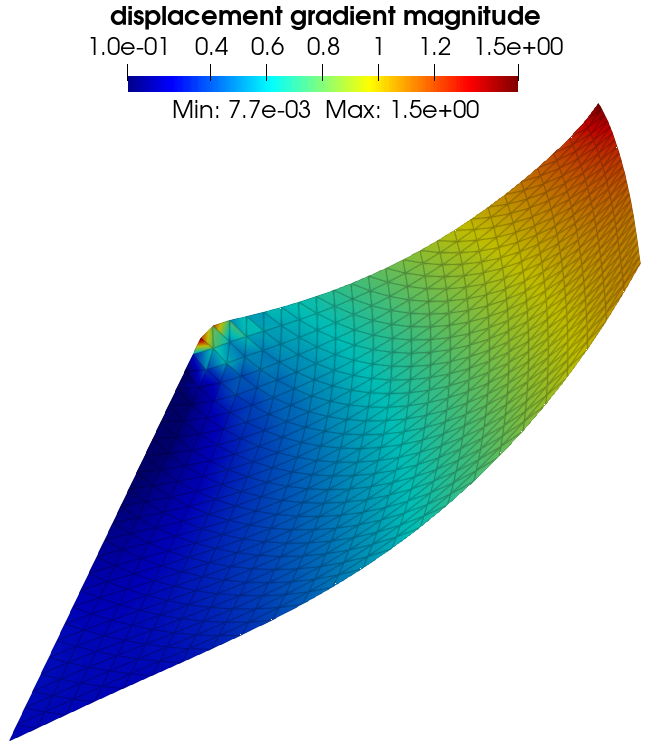}    
  \end{subfigure}
  \begin{subfigure}{0.32\textwidth}
    \centering
    \includegraphics[width=\textwidth]{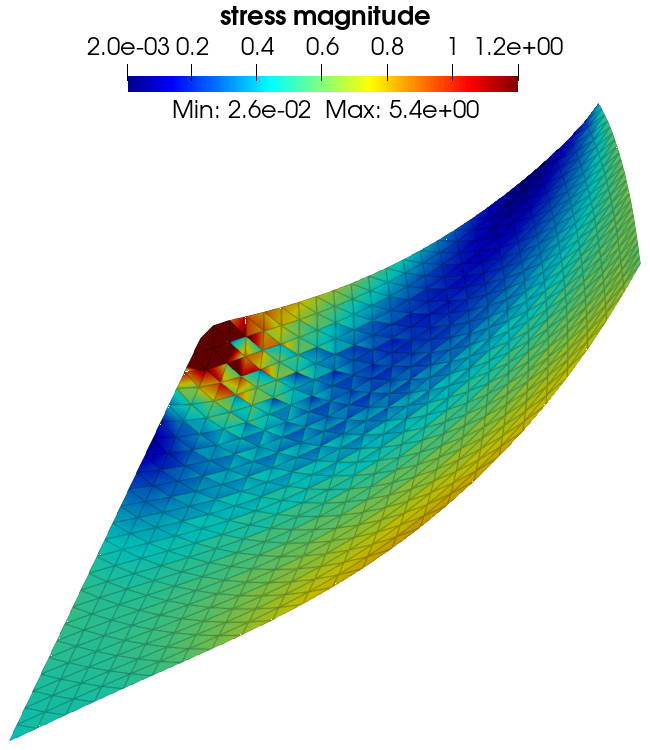}    
  \end{subfigure}
  \begin{subfigure}{0.32\textwidth}
    \centering
    \includegraphics[width=\textwidth]{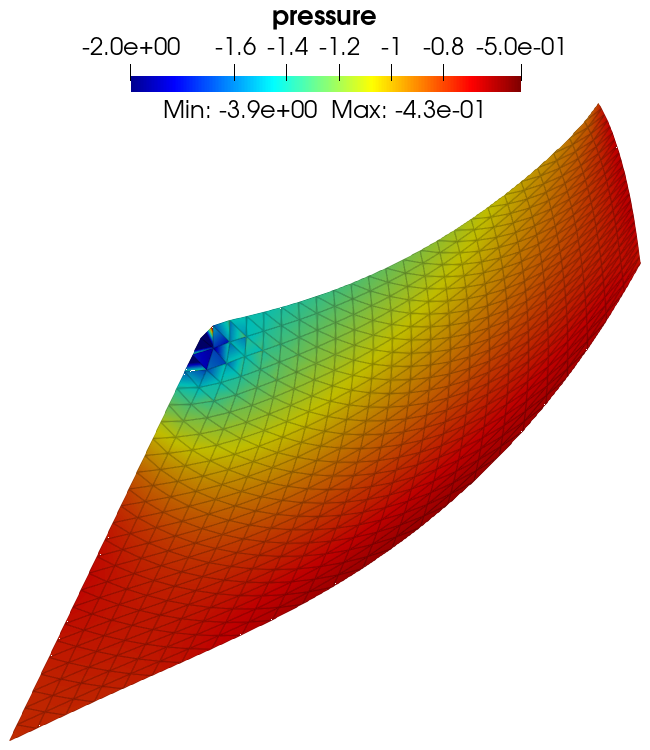}    
  \end{subfigure}
  
  \caption{\ac{fe} solutions on the deformed 2D Cook's membrane for traction boundary data $f=0.4$ and mesh parameter $n=24$. Columns show (1) displacement gradient magnitude, (2) stress magnitude and (3) pressure. Top row: \pairtwo{}(corr); bottom row: \pairfour{}.}
  \label{fig:mem2d_results}
\end{figure}

For the 3D Cook's membrane, the deflection at point~$A$ versus element count is shown in \fig{fig:mem3d_disp}. As in the 2D case, the results obtained with \pairtwo{} closely match those of the 3D \ac{csfem} pair \pairfive{} across all mesh resolutions. 
The low-order \ac{ddfem} pair \pairone{} converges toward the \pairfive{} response as the mesh is refined. Owing to the challenges associated with 3D quasi-uniform meshing, its convergence curves are noticeably noisier than in 2D. Nevertheless, after postprocessing, the \pairone{}(corr) curves become visually indistinguishable from those of \pairfive{}, indicating highly accurate displacement predictions. As already shown in \tab{tab:mem2d_disp}, the corrected \pairtwo{} displacements are nearly identical to the original solutions and are therefore omitted from the figure.
Figure~\ref{fig:mem3d_results} shows the \ac{fe} solution fields for \pairone{}(corr) and \pairfive{} on the finest mesh used in \fig{fig:mem3d_disp}.
A direct comparison between \pairtwo{} and \pairfive{} is not meaningful, since, as mentioned in \sect{sec:fe:fe}, \pairtwo{} utilises high-order \ac{bdm} elements and requires significantly more \acp{dof}. The deformed geometry for the \ac{ddfem} results is visualised using the corrected displacement to enhance clarity.
Once again, we observe similar patterns in the \ac{fe} solutions obtained with \pairone{} and \pairfive{}. However, the stress field produced by \pairone{} is smoother and more physically natural, owing to the richer \acp{fe} employed for the stress approximation.

\begin{figure}[hbt!]
  \centering
  \includegraphics[width=0.75\textwidth]{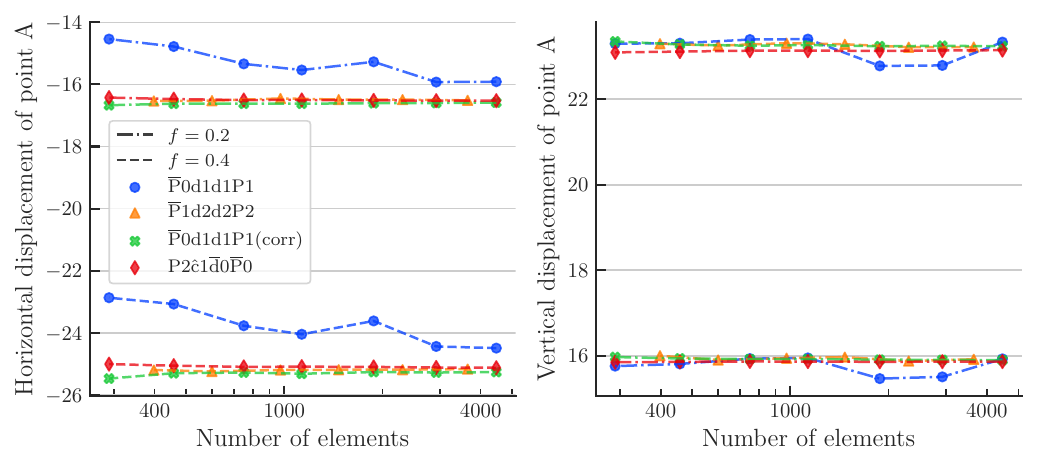}
  \caption{Deflection of the point $A$ for the 3D Cook's membrane problem under different traction boundary data $f$ for various pairs.}
  \label{fig:mem3d_disp}
\end{figure}

\begin{figure}[hbt!]
  \centering
    \begin{subfigure}{0.32\textwidth}
    \centering
    \includegraphics[width=\textwidth]{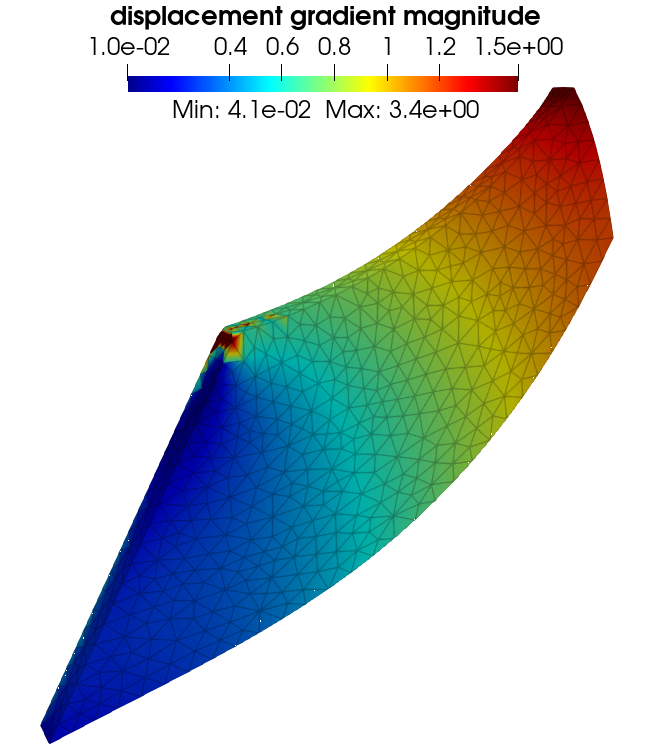}    
  \end{subfigure}
  \begin{subfigure}{0.32\textwidth}
    \centering
    \includegraphics[width=\textwidth]{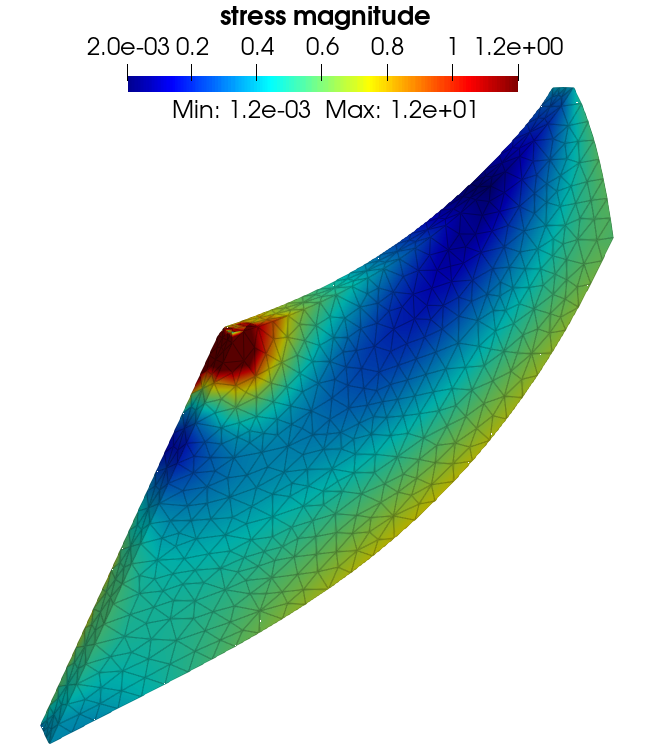}    
  \end{subfigure}
  \begin{subfigure}{0.32\textwidth}
    \centering
    \includegraphics[width=\textwidth]{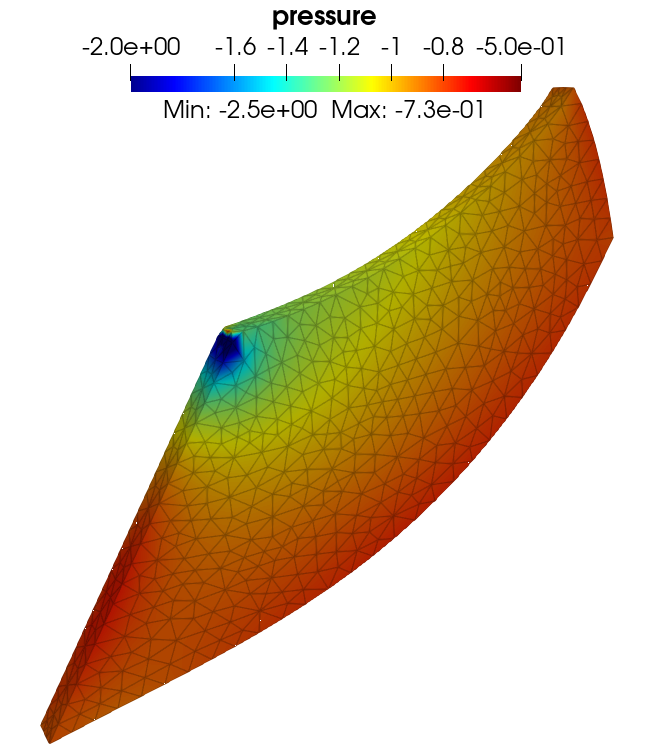}    
  \end{subfigure}

  \begin{subfigure}{0.32\textwidth}
    \centering
    \includegraphics[width=\textwidth]{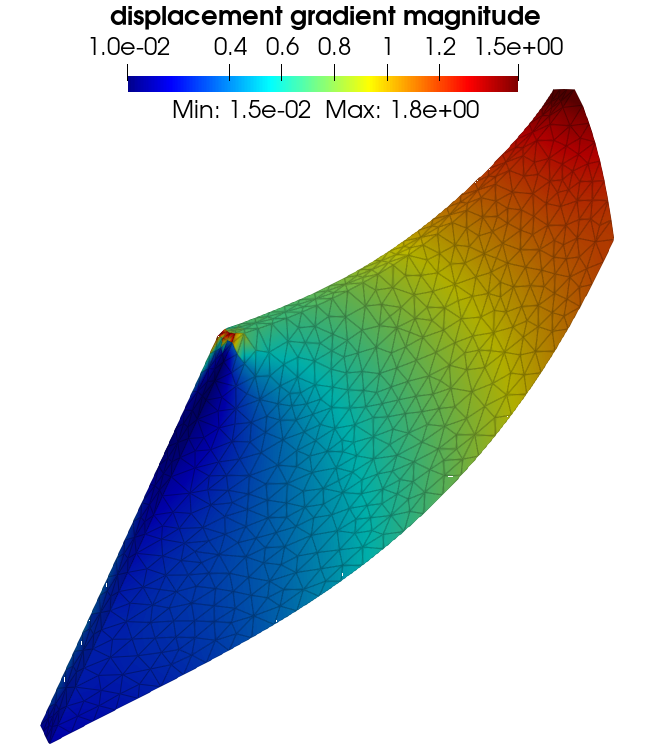}    
  \end{subfigure}
  \begin{subfigure}{0.32\textwidth}
    \centering
    \includegraphics[width=\textwidth]{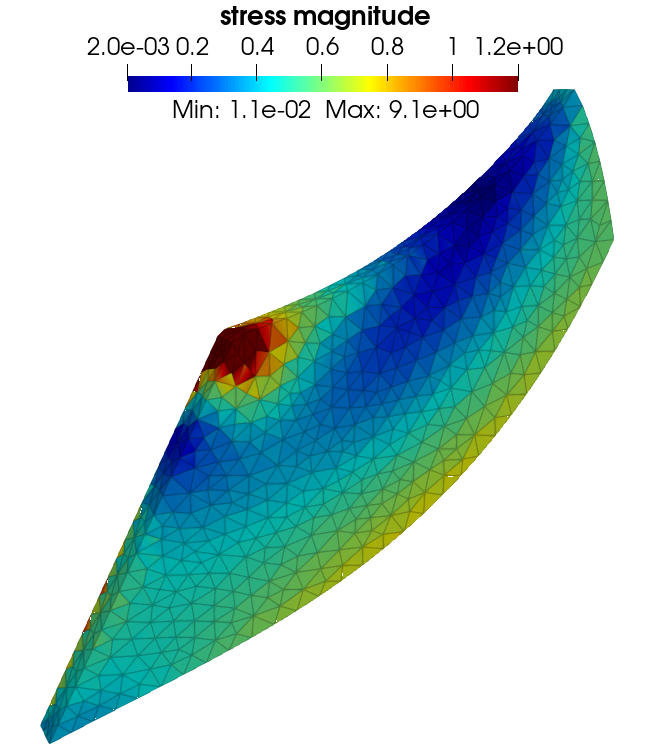}    
  \end{subfigure}
  \begin{subfigure}{0.32\textwidth}
    \centering
    \includegraphics[width=\textwidth]{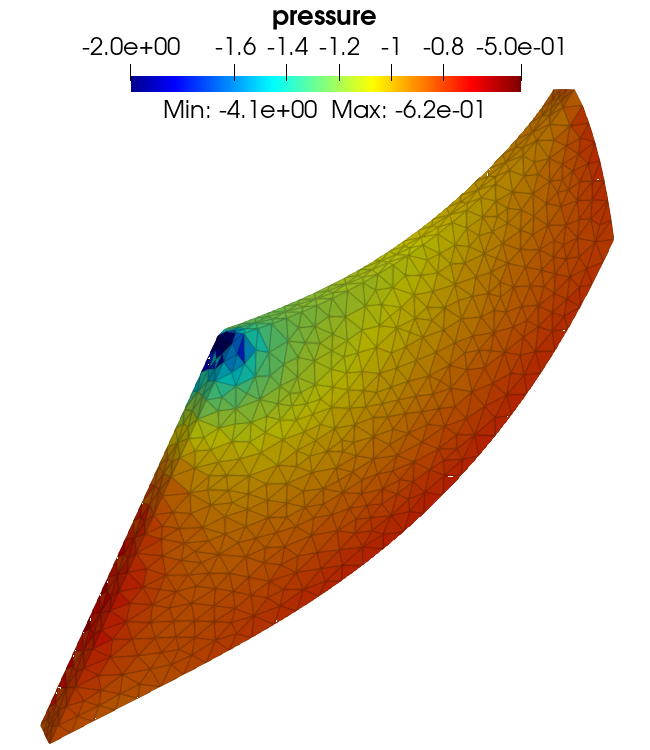}    
  \end{subfigure}
  
  \caption{\ac{fe} solutions on the deformed 3D Cook's membrane for traction boundary data $f=0.4$. Columns show (1) displacement gradient magnitude, (2) stress magnitude and (3) pressure. Top row: \pairone{}(corr); bottom row: \pairfive{}.}
  \label{fig:mem3d_results}
\end{figure}

It is also insightful to examine how the $L^2$ norms of the \ac{fe} solutions change as the mesh becomes finer, which offers a global comparison between the two methods. The $L^2$ norms for different pairs and traction levels $f$ are illustrated in \fig{fig:mem3d_norms}. 
The norms of the corrected displacements are nearly identical to those of the original \ac{ddfem} displacements; for clarity, the corrected curves are therefore omitted in the first subfigure.
For the \ac{ddfem} pairs \pairone{} and \pairtwo{}, the convergence curves are almost flat, indicating robust and consistent solutions even on coarse meshes. Interestingly, despite using a piecewise-constant space for the displacement in \pairone{}, the overall \ac{fe} solutions remain accurate.
In contrast, for \pairfive{}, the stress and pressure \ac{fe} solutions are inaccurate on coarse meshes, indicating that the 3D \ac{csfem} employs stress and pressure \ac{fe} spaces that are not sufficiently rich.

\begin{figure}[hbt!]
  \centering
  \includegraphics[width=\textwidth]{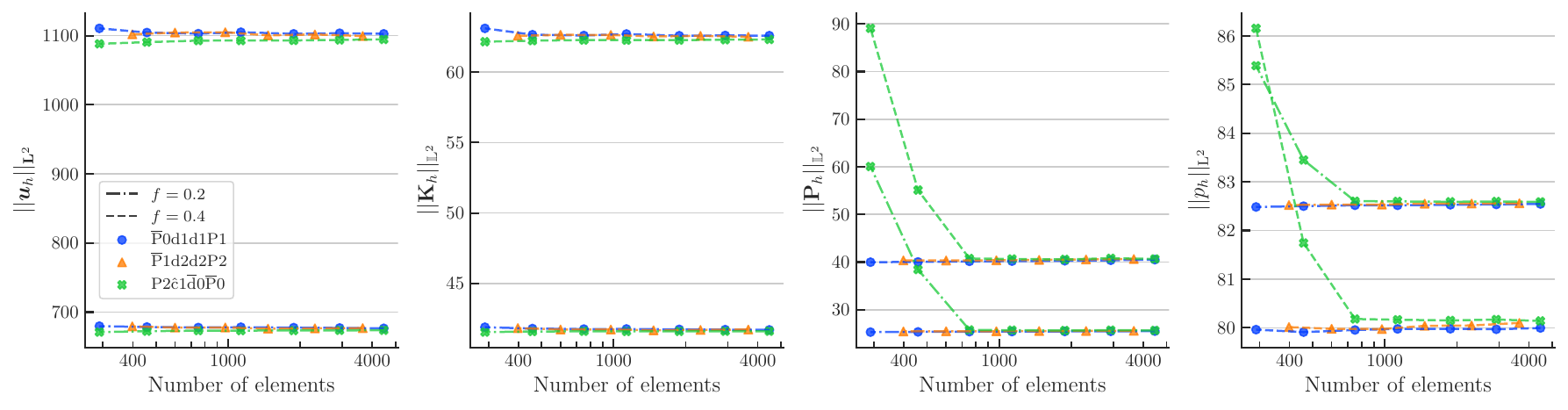}
  \caption{$L^2$ norms of displacement, displacement gradient, stress and pressure versus number of elements for the 3D Cook's membrane problem under different traction boundary data $f$ for various pairs.}
  \label{fig:mem3d_norms}
\end{figure}

\subsubsection{Stretching of perforated blocks} \label{sec:exps:stre}
As the final set of examples, we consider the stretching of blocks with holes. The 2D and 3D geometries are shown in \fig{fig:stre_geo}. Both the 2D square and 3D cube have side length 1. In 2D, the square contains four larger circular holes with diameter 0.2 and four smaller circular holes with diameter 0.1. The centres of the larger circles are located at: $(0.2,0.15)$, $(0.25,0.55)$, $(0.8,0.2)$, and $(0.8,0.75)$, while the smaller circles are centred at: $(0.15,0.85)$, $(0.45,0.8)$, $(0.5,0.25)$, and $(0.6,0.5)$. The left edge of the square is fixed, the right edge is subjected to a horizontal displacement boundary condition $\bar{\bu}_{\mathrm{2D}}$, and the remaining edges are traction-free.
In 3D, the cube has a spherical hole of diameter 0.6 at its centre. The left and right faces are subjected to equal and opposite horizontal displacement boundary conditions $\bar{\bu}_{\mathrm{3D}}$, and the other faces are traction-free. Again, owing to the symmetry of the problem, only one-eighth of the cube is modelled.
The incompressibility constraint is enforced using $C_1$~\eqref{eq:C1} for \acp{csfem}.
The displacement boundary data can be expressed as
\begin{equation*}
  \bar{\bu}_{\mathrm{2D}} = \begin{cases}
    (0, 0)^{\tt t} &\text{on the left edge},\\
    (u, 0)^{\tt t} &\text{on the right edge},
  \end{cases} \quad
  \bar{\bu}_{\mathrm{3D}} = \begin{cases}
    (-u, 0, 0)^{\tt t} &\text{on the left face},\\
    (u, 0, 0)^{\tt t} &\text{on the right face}.
  \end{cases}
\end{equation*}

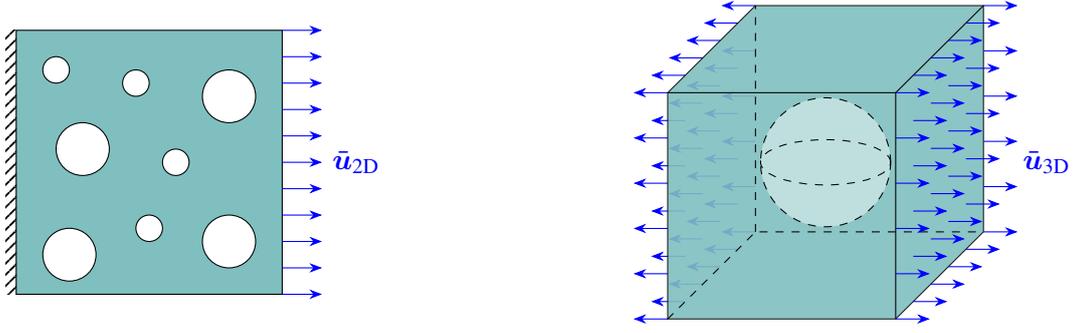
\begin{figure}[hbt!]
  \begin{subfigure}[h]{0.49\textwidth}
    \centering
    \begin{tikzpicture}[scale=3.5]
      \def\sr{0.05}
      \def\lr{0.1}

      \draw[fill=teal!50] (0,0) rectangle (1,1);
      \draw[fill=white] (0.2,0.15) circle(\lr);
      \draw[fill=white] (0.15, 0.85) circle(\sr);
      \draw[fill=white] (0.45, 0.8) circle(\sr);
      \draw[fill=white] (0.5, 0.25) circle(\sr);
      \draw[fill=white] (0.6, 0.5) circle(\sr);
      \draw[fill=white] (0.25, 0.55) circle(\lr);
      \draw[fill=white] (0.8, 0.2) circle(\lr);
      \draw[fill=white] (0.8, 0.75) circle(\lr);

      \draw[pattern={Lines[distance=3,angle=45,line width=0.7]},draw=none] (0,0) -- (-0.04,0) -- (-0.04,1) -- (0,1) -- cycle;
      \foreach \y in {0, 0.1, 0.2, ..., 1}{
        \draw[-Stealth,color=blue] (1,\y) -- (1.15,\y);
      }
      \node[right] at (1.15,0.5) {$\textcolor{blue}{\bar{\bu}_{\mathrm{2D}}}$};
    \end{tikzpicture}
  \end{subfigure}
  \hfill
  \begin{subfigure}[h]{0.49\textwidth}
    \centering
    \begin{tikzpicture}[scale=3.]
      \foreach \y in {0, 0.2, 0.4, ..., 1} {
        \foreach \z in {0, 0.2, 0.4, ..., 1} {
          \draw[-Stealth,color=blue] (0,\y,\z) -- (-0.15,\y,\z);
        }
      }
      \draw[fill=teal!50,opacity=0.9] (0,1,0) -- (1,1,0) -- (1,0,0) -- (1,0,1) -- (0,0,1) -- (0,1,1) -- cycle;
      \draw (1,1,1) -- (0,1,1) (1,1,1) -- (1,0,1) (1,1,1) -- (1,1,0);
      \draw[dashed] (0,0,1) -- (0,0,0) -- (1,0,0);
      \draw[dashed] (0,0,0) -- (0,1,0);
      \draw[dashed,fill=teal!25] (0.5,0.5,0.5) circle (0.285);
      \draw[dashed] (0.5,0.5,0.5) ellipse (0.285 cm and 0.1 cm);
      \foreach \y in {0, 0.2, 0.4, ..., 1} {
        \foreach \z in {0, 0.2, 0.4, ..., 1} {
          \draw[-Stealth,color=blue] (1,\y,\z) -- (1.15,\y,\z);
        }
      }
      \node[right] at (1.32,0.5,0.5) {$\textcolor{blue}{\bar{\bu}_{\mathrm{3D}}}$};
    \end{tikzpicture}
  \end{subfigure}

  \caption{Geometries and displacement boundary conditions for the stretching problems. Left: a 2D square with randomly distributed holes; right: a 3D cube with a central hollow sphere.}
  \label{fig:stre_geo}
\end{figure}

The \ac{fe} solutions on the deformed square are shown in \fig{fig:stre2d_results}, with \pairtwo{}(corr) in the top row and \pairfour{} in the bottom. The Delaunay mesh contains 3,344 elements with finer cells around the smaller holes.
Visually, the \pairfour{} displacement closely matches the corrected \pairtwo{} solution.
Overall, the displacement gradient solutions (first column of \fig{fig:stre2d_results}) obtained by the two pairs exhibit very similar patterns, except near the midpoint of the right edge, where notable differences occur. 
To further investigate this dissimilarity, we present the Jacobian determinant solutions computed by $J_h = \det(\bK_h + \bbI)$ in the second column of \fig{fig:stre2d_results}. We observe that the \pairfour{} Jacobian determinant becomes very small in that region, with the minimum value turning negative (nonphysical). 
In contrast, the \pairtwo{} Jacobian determinant remains close to 1 throughout most of the domain, with only minor regions showing relatively large or small values, but never negative.
Regarding the stress solutions, similar to the observations in \sect{sec:exps:mem}, \pairtwo{} produces smoother and more natural \ac{fe} solutions than \pairfour{}.
As expected, the continuous pressure solution by \pairtwo{} is also smoother compared to the discontinuous pressure solution of \pairfour{}.

\begin{figure}[hbt!]
  \centering
    \begin{subfigure}{0.245\textwidth}
    \centering
    \includegraphics[width=\textwidth]{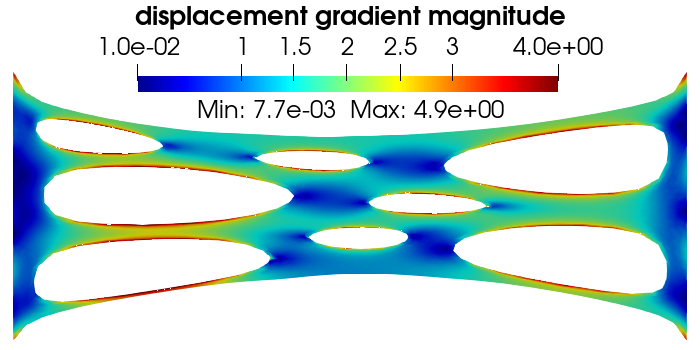}    
  \end{subfigure}
  \begin{subfigure}{0.245\textwidth}
    \centering
     \includegraphics[width=\textwidth]{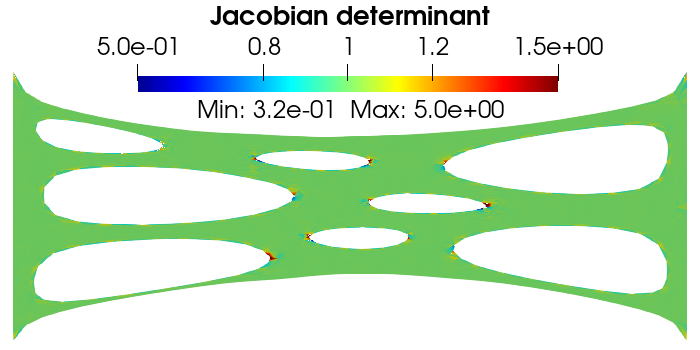}   
      \end{subfigure}
     \begin{subfigure}{0.245\textwidth}
    \centering
    \includegraphics[width=\textwidth]{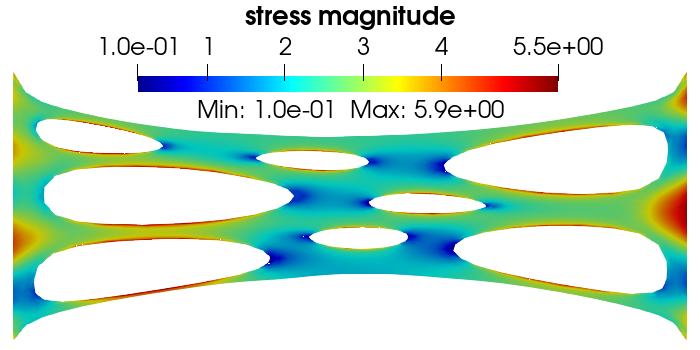}    
  \end{subfigure}
  \begin{subfigure}{0.245\textwidth}
    \centering
     \includegraphics[width=\textwidth]{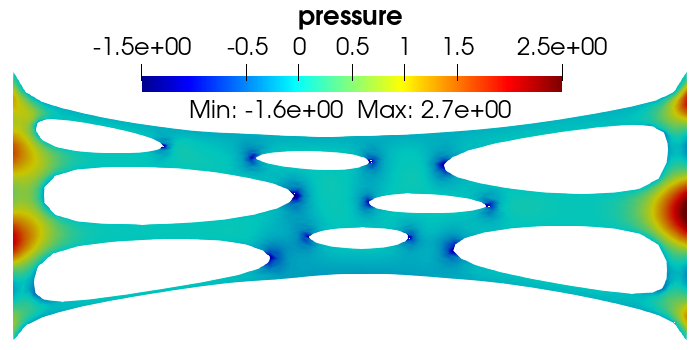}   
      \end{subfigure}
     
    \begin{subfigure}{0.245\textwidth}
    \centering
    \includegraphics[width=\textwidth]{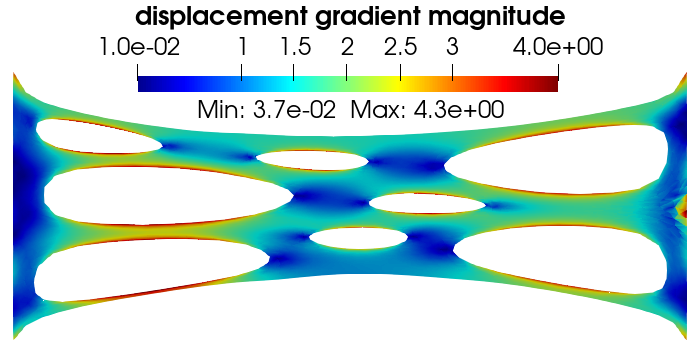}    
  \end{subfigure}
  \begin{subfigure}{0.245\textwidth}
    \centering
     \includegraphics[width=\textwidth]{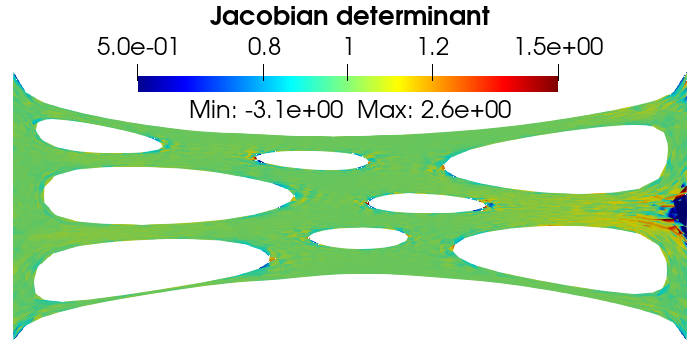}   
      \end{subfigure}
     \begin{subfigure}{0.245\textwidth}
    \centering
    \includegraphics[width=\textwidth]{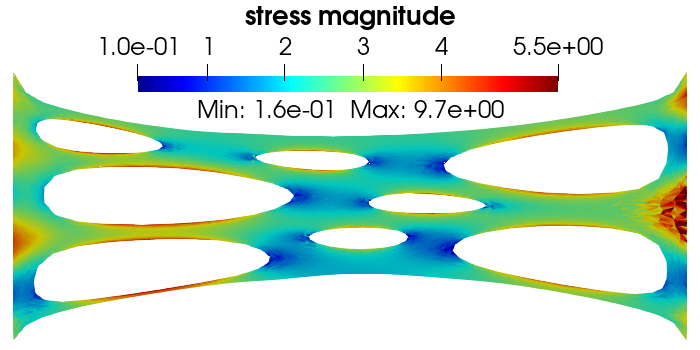}    
  \end{subfigure}
  \begin{subfigure}{0.245\textwidth}
    \centering
     \includegraphics[width=\textwidth]{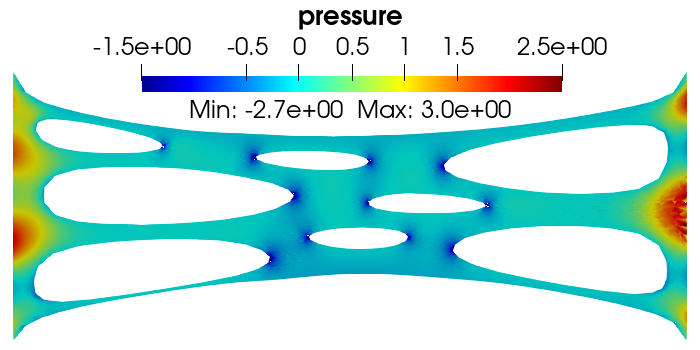}   
      \end{subfigure}

  \caption{\ac{fe} solutions obtained by different pairs on the deformed 2D perforated block. The displacement boundary data $u=1.5$. Columns show (1) displacement gradient magnitude, (2) Jacobian, (3) stress magnitude and (4) pressure. Top row: \pairtwo{}(corr); bottom row: \pairfour{}.}
  \label{fig:stre2d_results}
\end{figure}

We examine the projected Jacobian determinants as box plots in \fig{fig:stre2d_jacs}. 
To obtain the $\mathrm{L}^2$-projection of the Jacobian, we first solve the 2D stretching problem using all four element pairs on the same Delaunay mesh consisting of 3,344 triangular elements. We then project the computed \ac{fe} Jacobian $J_h = \det(\bK_h + \bbI)$ onto a discontinuous Lagrange \ac{fe} space ($\overline{\mathrm{H}}_{k,h}$). 
The order of the projection space matches the order of the \ac{fe} space used for the displacement gradient: we use $\overline{\mathrm{H}}_{1,h}$ for projecting \pairone{} and \pairthree{} Jacobian determinants, and $\overline{\mathrm{H}}_{2,h}$ for \pairtwo{} and \pairfour{}. Finally, we plot the nodal values of the projected Jacobian determinant for each pair as box plots in \fig{fig:stre2d_jacs}. 
The whiskers indicate the minimum and maximum values lying within 1.5 times the interquartile range. Due to the geometric complexity, it is natural that the box plots for all element pairs exhibit numerous fliers. 
The box plots clearly indicate that \pairtwo{} preserves volume best, with its Jacobian determinant values tightly concentrated around 1, while \pairthree{} exhibits the largest spread and thus the poorest volume preservation.
Figure~\ref{fig:stre2d_jacs} also provides insights into why \pairthree{} fails under large deformations: its projected Jacobian determinant contains a significant number of negative values. In contrast, \pairone{} shows no negative Jacobian determinants, and \pairtwo{} only has one negative value, making these two \ac{ddfem} pairs more reliable than \ac{csfem} pairs for this 2D stretching problem.

\begin{figure}[hbt!]
  \centering
  \includegraphics[width=0.9\textwidth]{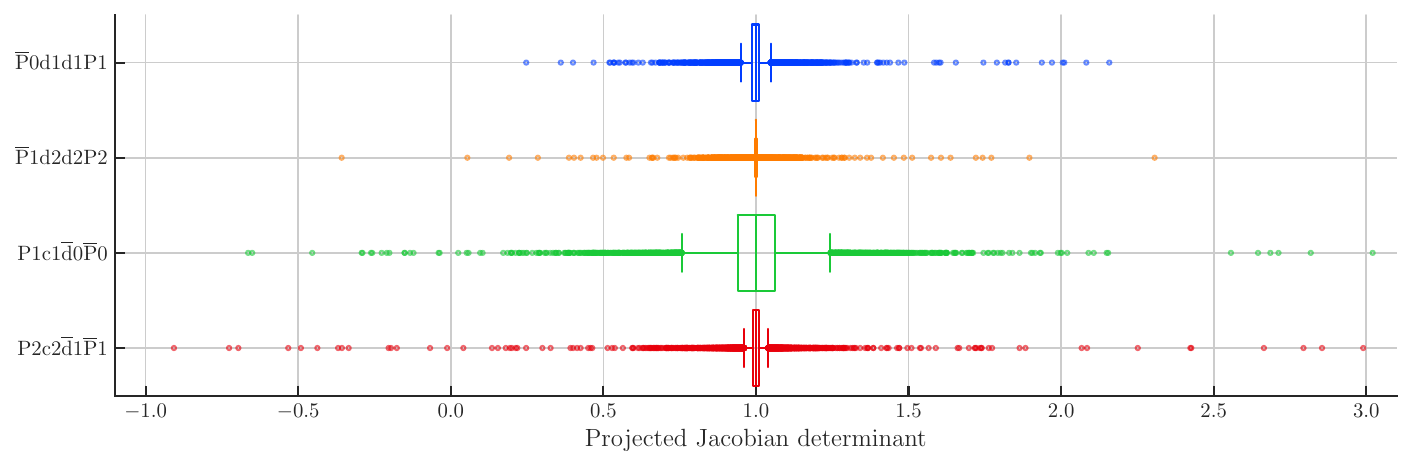}
  \caption{Box plot of the projected Jacobian determinant values for the 2D stretching problem. The displacement boundary data $u=1.5$. The Delaunay mesh consists of 3,344 elements.}
  \label{fig:stre2d_jacs}
\end{figure}

The \ac{fe} solutions for the 3D stretching problem solved by \pairone{}(corr) and \pairfive{} are displayed in the top and bottom rows of \fig{fig:stre3d_results}, respectively. The Delaunay mesh consists of 4,131 tetrahedra, and the deformed geometry for \pairone{} is corrected as described in \sect{sec:fe:ddcorr}.
The deformed geometries and overall solution patterns are broadly similar for both pairs. However, pronounced checkerboarding appears in the stress and pressure fields of \pairfive{}, indicating again that the corresponding 3D stress and pressure \ac{fe} pair is insufficiently rich. In contrast, the \pairone{} stress and pressure solutions are much smoother.
On this mesh, \pairfive{} has 122,242 \acp{dof}, whereas \pairone{} has 166,135 \acp{dof}. Although \pairone{} requires around 36\% more \acp{dof}, it avoids numerical artefacts such as checkerboarding and yields more natural and visually appealing stress and pressure solutions.

\begin{figure}[hbt!]
  \centering
    \begin{subfigure}{0.32\textwidth}
    \centering
    \includegraphics[width=\textwidth]{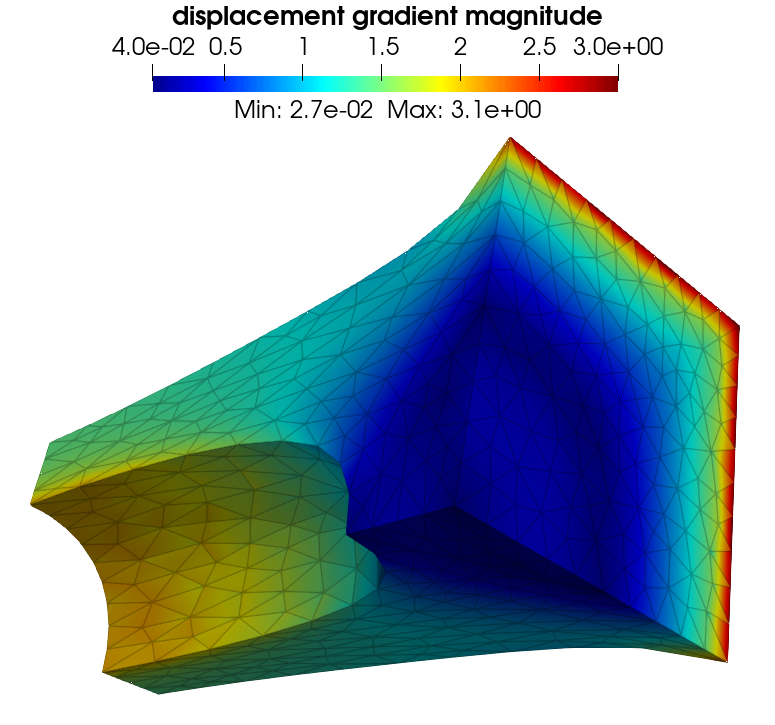}    
  \end{subfigure}
  \begin{subfigure}{0.32\textwidth}
    \centering
    \includegraphics[width=\textwidth]{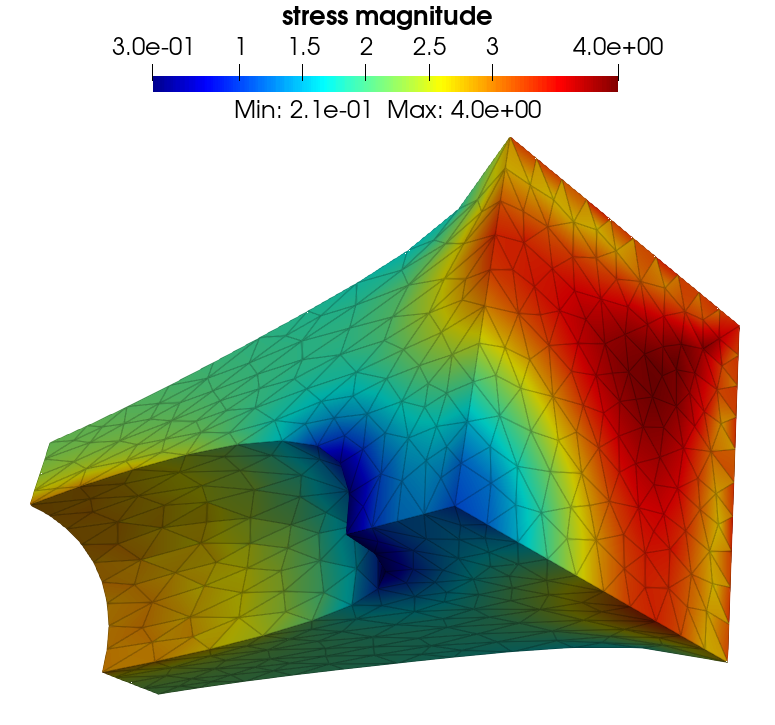}    
  \end{subfigure}
  \begin{subfigure}{0.32\textwidth}
    \centering
    \includegraphics[width=\textwidth]{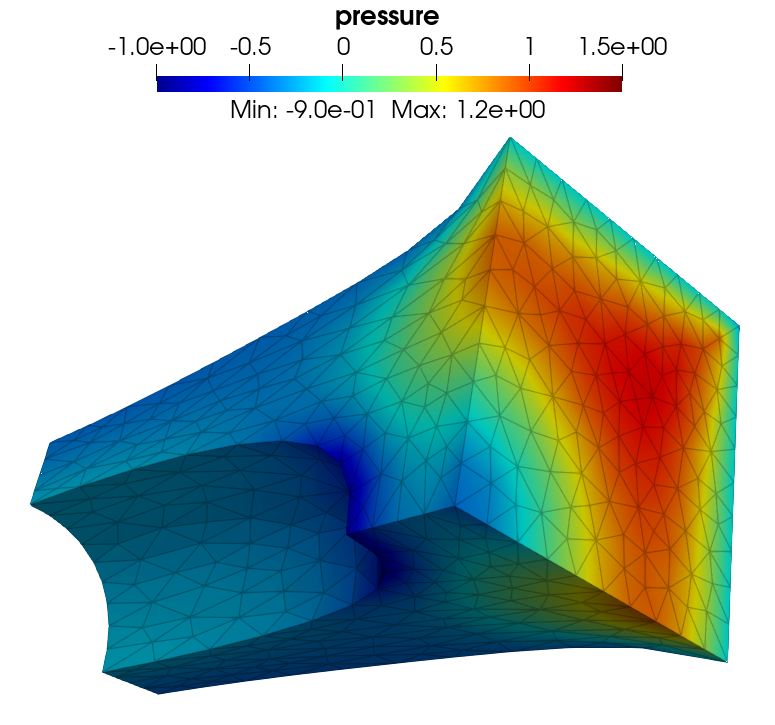}    
  \end{subfigure}
  
  \begin{subfigure}{0.32\textwidth}
    \centering
    \includegraphics[width=\textwidth]{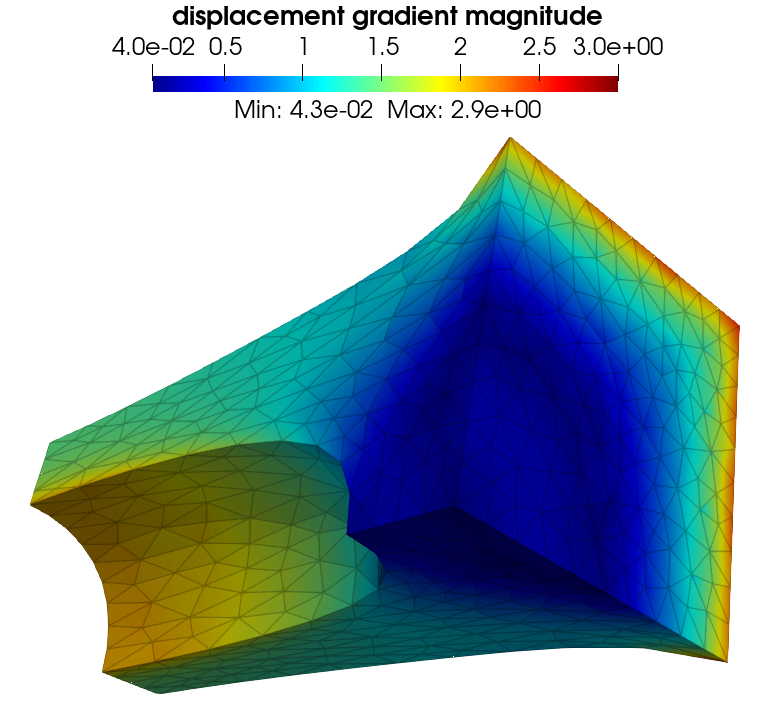}    
  \end{subfigure}
  \begin{subfigure}{0.32\textwidth}
    \centering
    \includegraphics[width=\textwidth]{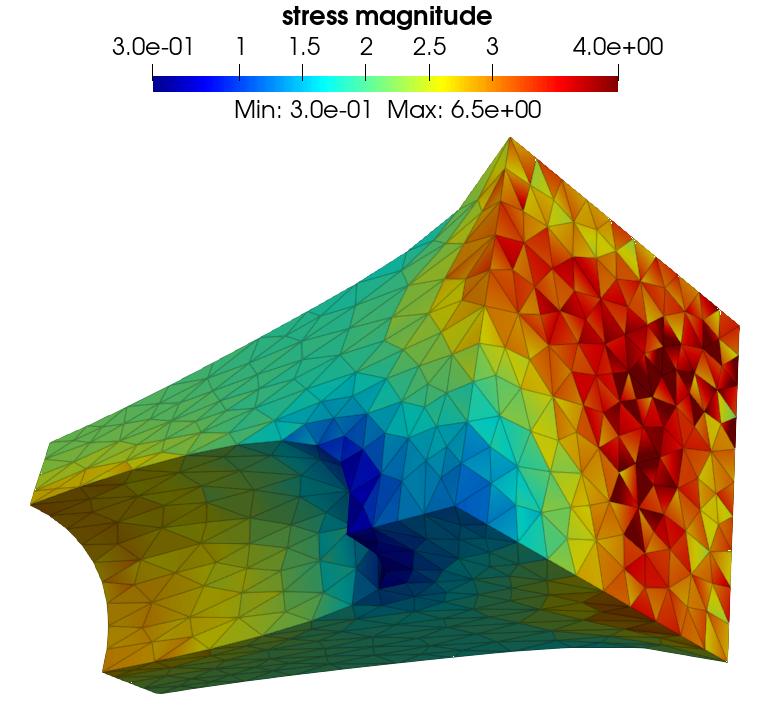}    
  \end{subfigure}
  \begin{subfigure}{0.32\textwidth}
    \centering
    \includegraphics[width=\textwidth]{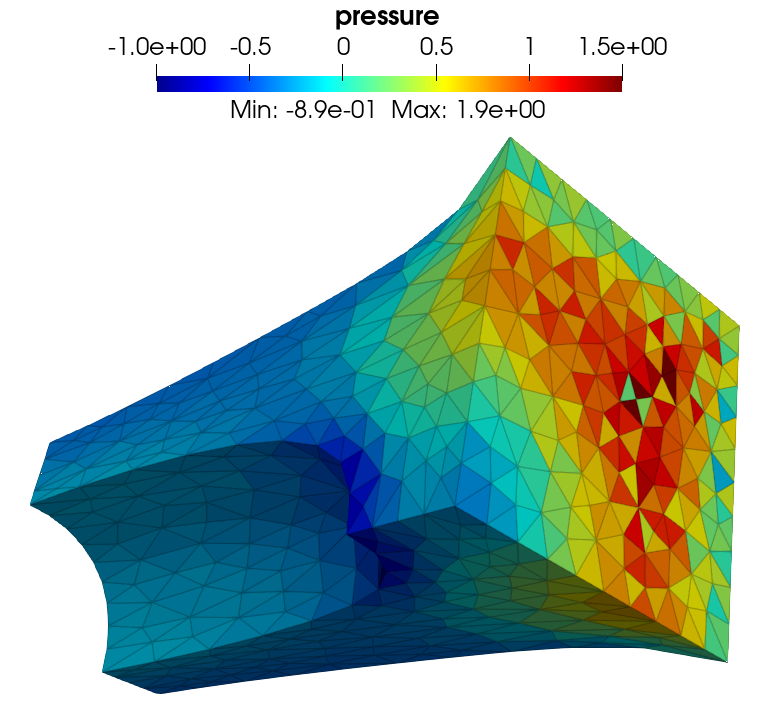}    
  \end{subfigure}
  
  \caption{\ac{fe} solutions obtained by different element pairs on the deformed 3D perforated block. The displacement boundary data $u=0.5$. Columns show (1) displacement gradient magnitude, (2) stress magnitude and (3) pressure. Top row: \pairone{}(corr); bottom row: \pairfive{}.}
  \label{fig:stre3d_results}
\end{figure}

To assess the consistency and accuracy of \acp{ddfem} and \acp{csfem} across different mesh resolutions, we plot the $L^2$ norms of the \ac{fe} solutions against the \acp{dof} in 2D and against element counts in 3D, as displayed in \fig{fig:stre_norms}.
For the 2D case, the two low-order pairs, \pairone{} and \pairthree{}, use the same set of meshes, while the other two high-order pairs, \pairtwo{} and \pairfour{}, share a different common set. 
All four pairs exhibit strong consistency in both the displacement and displacement gradient solutions, as reflected by their nearly flat curves. 
However, due to the use of order 0 Raviart--Thomas elements, the low-order \ac{csfem} pair \pairthree{} converges more slowly, while the other three pairs yield consistent stress results, as indicated by their flat stress curves.
For the pressure solution at small deformation ($u=0.5$), all four pairs converge, with \pairthree{} slower than \pairone{}, and \pairfour{} slower than \pairtwo{} in terms of \acp{dof}. Under large deformation ($u=1.5$), both \ac{csfem} pairs show poor convergences, and \pairfour{} fails on the finest mesh. This is not surprising, as negative Jacobian values already occur on the second finest mesh for the \pairfour{} solution, as shown in \fig{fig:stre2d_jacs}. 
For the two \ac{ddfem} pairs, the pressure solutions converge even under large deformations, with \pairone{} converging slightly faster than \pairtwo{}. 
It is also important to compare the \acp{dof} among low- and high-order pairs on the same mesh. For the displacement, \pairone{} has more \acp{dof} than \pairthree{} and \pairtwo{} possesses more than \pairfour{}. For the displacement gradient, the low-order pairs and high-order pairs have the same \acp{dof}, since \ac{bdm} elements can be constructed by rotating second-kind \nedelec{} elements of the same order. For the stress, \ac{bdm} elements are richer than Raviart-Thomas elements of one order lower, so \ac{ddfem} pairs have more \acp{dof}. For the pressure, \ac{ddfem} pairs have fewer \acp{dof}.
Overall, on the same 2D mesh, \pairone{} has around 30\% more \acp{dof} than \pairthree{}, and \pairtwo{} has roughly 18\% more than \pairfour{}.

\begin{figure}[hbt!]
  \centering
  \begin{subfigure}{\textwidth}
    \centering
    \includegraphics[width=\textwidth]{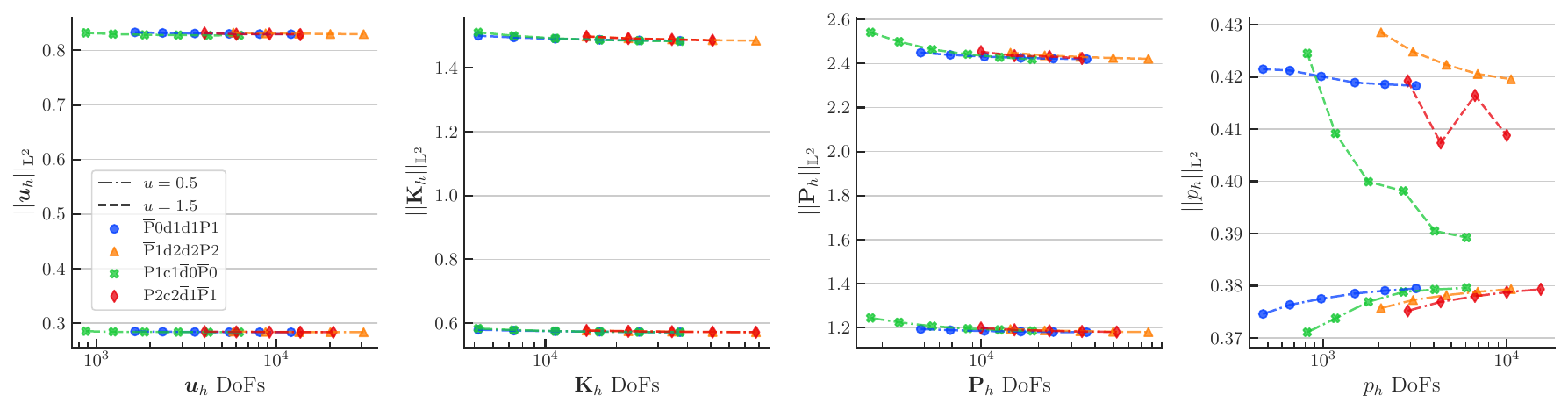}
    \caption{}
    \label{fig:stre2d_norms}
  \end{subfigure}
  \begin{subfigure}{\textwidth}
    \centering
    \includegraphics[width=\textwidth]{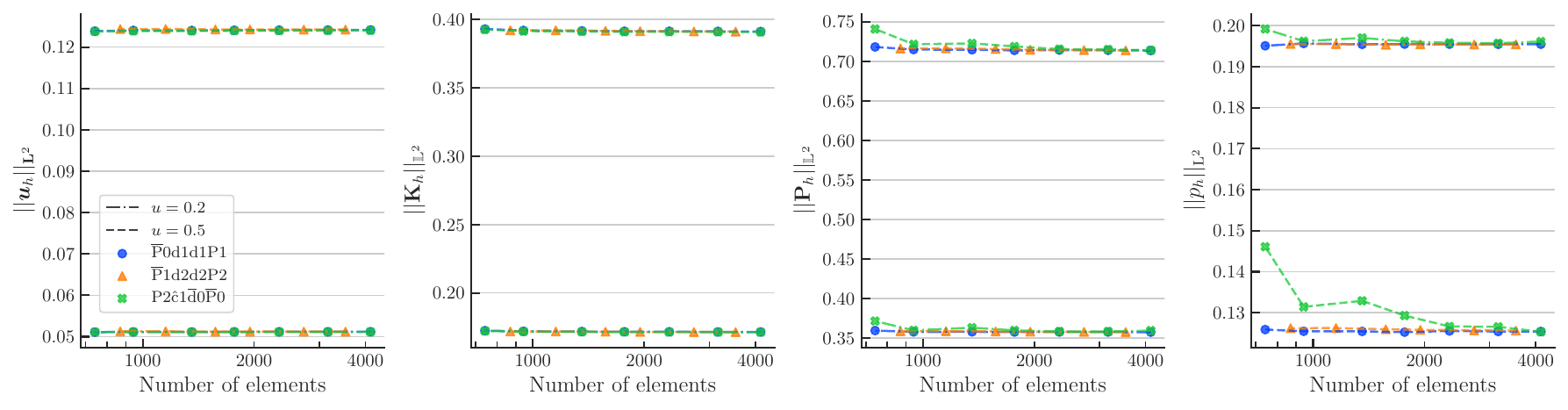}
    \caption{}
    \label{fig:stre3d_norms}
  \end{subfigure}
  \caption{$L^2$ norms of displacement, displacement gradient, stress and pressure versus number of elements for the 2D and 3D stretching problems under different displacement data $u$. Panels:  (a) 2D results, (b) 3D results.}
  \label{fig:stre_norms}
\end{figure}

In the $L^2$ norms for the 3D stretching cases shown in \fig{fig:stre3d_norms}, both \ac{ddfem} pairs exhibit nearly flat curves across all four fields, indicating accurate and stable \ac{fe} solutions over different mesh resolutions. 
The 3D \ac{csfem} pair \pairfive{} achieves fast convergence for displacement and displacement gradient solutions; nonetheless, its stress and pressure solutions converge slowly due to the use of the insufficiently rich lowest-order Raviart--Thomas \acp{fe}. For large deformations ($u=0.5$), the pressure convergence deteriorates further.
Overall, on sufficiently fine meshes, the \ac{fe} solutions from all three pairs are in good agreement.

\section{Concluding remarks} \label{sec:conclu}
The proposed four-field mixed formulation with discontinuous displacement offers an accurate, robust and scalable alternative for the simulation of incompressible nonlinear solids. 
Its stabilisation-free performance in 3D and compatibility with standard \ac{fe} spaces make it particularly attractive for large-deformation applications. 
Future work will focus on extending the formulation to aggregated \ac{fe} spaces~\cite{Badia2018} to better handle complex geometries, exploring the benefits of integrating neural networks~\cite{Badia2024,Badia2025} in such strongly nonlinear regimes, and combining the method with the active strain or active stress approaches~\cite{giantesio2025modeling} to address more challenging couplings of multiphysics and multiscale problems, such as cardiac electromechanics, tumour growth, brain tissue characterisation, or gastric function simulations.

\bigskip 
\noindent\textbf{Acknowledgments.} 
This research was partially funded by the Australian Government through the Australian Research Council (project number DP220103160), by the Institut Mittag--Leffler in Djursholm Sweden, under the research program \emph{Interfaces and unfitted discretization methods} (grant no. 2021-06594 from the Swedish Research Council), and by the Centre of Advanced Study (CAS) at the Norwegian Academy of Science and Letters under the program \emph{Mathematical Challenges in Brain Mechanics}. Computational resources were provided by the Australian Government through NCI under the NCMAS Merit Allocation Schemes.

\bigskip\noindent\textbf{Reproducibility.} 
The Gridap routines employed to generate the numerical results in the paper can be accessed from \url{https://github.com/wei3li/four-field-elasticity}. 

\printbibliography

@inbook{Hellinger1907,
  author    = {Hellinger, E.},
  editor    = {Klein, Felix and M{\"u}ller, Conr.},
  title     = {Die Allgemeinen Ans{\"a}tze der Mechanik der Kontinua},
  booktitle = {Mechanik},
  year      = {1907},
  publisher = {Vieweg+Teubner Verlag},
  address   = {Wiesbaden},
  pages     = {601--694},
  isbn      = {978-3-663-16028-1},
  doi       = {10.1007/978-3-663-16028-1_9},
  url       = {https://doi.org/10.1007/978-3-663-16028-1_9}
}

@article{Reissner1950,
  author  = {Reissner, Eric},
  title   = {On a Variational Theorem in Elasticity},
  journal = {Journal of Mathematics and Physics},
  volume  = {29},
  number  = {1-4},
  pages   = {90-95},
  doi     = {10.1002/sapm195029190},
  url     = {https://onlinelibrary.wiley.com/doi/abs/10.1002/sapm195029190},
  eprint  = {https://onlinelibrary.wiley.com/doi/pdf/10.1002/sapm195029190},
  year    = {1950}
}

@article{Viebahn2018,
  title    = {A simple and efficient Hellinger--Reissner type mixed finite element for nearly incompressible elasticity},
  journal  = {Computer Methods in Applied Mechanics and Engineering},
  volume   = {340},
  pages    = {278-295},
  year     = {2018},
  issn     = {0045-7825},
  doi      = {10.1016/j.cma.2018.06.001},
  url      = {https://www.sciencedirect.com/science/article/pii/S0045782518302901},
  author   = {Nils Viebahn and Karl Steeger and Jörg Schröder}
}

@article{Farrell2021,
  title    = {Mixed Kirchhoff stress--displacement--pressure formulations for incompressible hyperelasticity},
  journal  = {Computer Methods in Applied Mechanics and Engineering},
  volume   = {374},
  pages    = {113562},
  year     = {2021},
  issn     = {0045-7825},
  doi      = {10.1016/j.cma.2020.113562},
  url      = {https://www.sciencedirect.com/science/article/pii/S0045782520307477},
  author   = {Patrick E. Farrell and Luis F. Gatica and Bishnu P. Lamichhane and Ricardo Oyarzúa and Ricardo Ruiz-Baier}
}

@article{Stenberg1991,
  author  = {Stenberg, Rolf},
  title   = {Postprocessing schemes for some mixed finite elements},
  doi     = {10.1051/m2an/1991250101511},
  url     = {https://doi.org/10.1051/m2an/1991250101511},
  journal = {ESAIM: M2AN},
  year    = 1991,
  volume  = 25,
  number  = 1,
  pages   = {151-167}
}

@article{giantesio2025modeling,
  title     = {On the modeling of active deformation in biological transversely isotropic materials},
  author    = {Giantesio, Giulia and Musesti, Alessandro},
  journal   = {Journal of Elasticity},
  volume    = {157},
  number    = {1},
  pages     = {10},
  year      = {2025},
  publisher = {Springer}
}

@book{Ern04,
  author    = {Ern, A. and Guermond, J.-L.},
  title     = {Theory and Practice of Finite Elements},
  publisher = {Springer},
  series    = {Appl. Math. Sci.},
  volume    = {159},
  year      = {2004},
  address   = {New York}
}

@article{henke2026electromechanical,
  title     = {Electromechanical computational model of the human stomach},
  author    = {Henke, Maire S and Brandstaeter, Sebastian and Fuchs, Sebastian L and Aydin, Roland C and Gizzi, Alessio and Cyron, Christian J},
  journal   = {Computer Methods in Applied Mechanics and Engineering},
  volume    = {449},
  pages     = {118549},
  year      = {2026},
  publisher = {Elsevier}
}

@article{Shojaei2018,
  title    = {Compatible-strain mixed finite element methods for incompressible nonlinear elasticity},
  journal  = {Journal of Computational Physics},
  volume   = {361},
  pages    = {247-279},
  year     = {2018},
  issn     = {0021-9991},
  doi      = {10.1016/j.jcp.2018.01.053},
  url      = {https://www.sciencedirect.com/science/article/pii/S0021999118300755},
  author   = {Mostafa {Faghih Shojaei} and Arash Yavari}
}

@article{Shojaei2019,
  title    = {Compatible-strain mixed finite element methods for 3D compressible and incompressible nonlinear elasticity},
  journal  = {Computer Methods in Applied Mechanics and Engineering},
  volume   = {357},
  pages    = {112610},
  year     = {2019},
  issn     = {0045-7825},
  doi      = {10.1016/j.cma.2019.112610},
  url      = {https://www.sciencedirect.com/science/article/pii/S0045782519304864},
  author   = {Mostafa {Faghih Shojaei} and Arash Yavari}
}

@article{Fu2025,
  title    = {A four-field mixed formulation for incompressible finite elasticity},
  journal  = {Computer Methods in Applied Mechanics and Engineering},
  volume   = {444},
  pages    = {118082},
  year     = {2025},
  issn     = {0045-7825},
  doi      = {10.1016/j.cma.2025.118082},
  url      = {https://www.sciencedirect.com/science/article/pii/S0045782525003548},
  author   = {Guosheng Fu and Michael Neunteufel and Joachim Sch{\"o}berl and Adam Zdunek}
}

@article{Badia2020,
  doi       = {10.21105/joss.02520},
  url       = {https://doi.org/10.21105/joss.02520},
  year      = {2020},
  publisher = {The Open Journal},
  volume    = {5},
  number    = {52},
  pages     = {2520},
  author    = {Santiago Badia and Francesc Verdugo},
  title     = {Gridap: An extensible Finite Element toolbox in Julia},
  journal   = {Journal of Open Source Software}
}

@article{Verdugo2022,
  doi       = {10.1016/j.cpc.2022.108341},
  url       = {https://doi.org/10.1016/j.cpc.2022.108341},
  year      = {2022},
  month     = jul,
  publisher = {Elsevier {BV}},
  volume    = {276},
  pages     = {108341},
  author    = {Francesc Verdugo and Santiago Badia},
  title     = {The software design of Gridap: A Finite Element package based on the Julia {JIT} compiler},
  journal   = {Computer Physics Communications}
}

@article{Alvarez2015,
  author  = {Alvarez, Mario and Gatica, Gabriel N. and Ruiz-Baier, Ricardo},
  title   = {An augmented mixed-primal finite element method   for a coupled flow-transport problem},
  doi     = {10.1051/m2an/2015015},
  url     = {https://doi.org/10.1051/m2an/2015015},
  journal = {ESAIM: M2AN},
  year    = 2015,
  volume  = 49,
  number  = 5,
  pages   = {1399-1427},
  month   = aug
}

@article{Rossi2014,
  title    = {Thermodynamically consistent orthotropic activation model capturing ventricular systolic wall thickening in cardiac electromechanics},
  journal  = {European Journal of Mechanics - A/Solids},
  volume   = {48},
  pages    = {129-142},
  year     = {2014},
  issn     = {0997-7538},
  doi      = {10.1016/j.euromechsol.2013.10.009},
  url      = {https://www.sciencedirect.com/science/article/pii/S0997753813001228},
  author   = {Simone Rossi and Toni Lassila and Ricardo Ruiz-Baier and Adélia Sequeira and Alfio Quarteroni}
}

@book{Treloar2005,
  author    = {Treloar, G. L. R.},
  title     = {The Physics of Rubber Elasticity},
  publisher = {Oxford University Press},
  year      = {2005},
  month     = {10},
  isbn      = {9780198570271},
  doi       = {10.1093/oso/9780198570271.001.0001},
  url       = {https://doi.org/10.1093/oso/9780198570271.001.0001}
}

@book{Bonet2008,
  place     = {Cambridge},
  edition   = {2},
  title     = {Nonlinear Continuum Mechanics for Finite Element Analysis},
  publisher = {Cambridge University Press},
  author    = {Bonet, Javier and Wood, Richard D.},
  year      = {2008}
}

@article{Badia2016,
  author  = {Badia, Santiago and Mart\'{\i}n, Alberto F. and Principe, Javier},
  title   = {Multilevel Balancing Domain Decomposition at Extreme Scales},
  journal = {SIAM Journal on Scientific Computing},
  volume  = {38},
  number  = {1},
  pages   = {C22-C52},
  year    = {2016},
  doi     = {10.1137/15M1013511},
  url     = {https://doi.org/10.1137/15M1013511},
  eprint  = {https://doi.org/10.1137/15M1013511}
}

@book{Zienkiewicz2005,
  title     = {The Finite Element Method for Solid and Structural Mechanics},
  author    = {Zienkiewicz, O. C. and Taylor, R. L.},
  isbn      = {9780080455587},
  lccn      = {2006272263},
  url       = {https://books.google.com.au/books?id=VvpU3zssDOwC},
  year      = {2005},
  publisher = {Butterworth-Heinemann}
}

@book{Hughes2000,
  author    = {Hughes, Thomas J. R.},
  title     = {The Finite Element Method: Linear Static and Dynamic Finite Element Analysis},
  publisher = {Dover Publications},
  year      = {2000},
  address   = {Mineola, NY},
  isbn      = {9780486411811},
  series    = {Dover Civil and Mechanical Engineering}
}

@article{Arnold1990,
  title   = {Mixed finite element methods for elliptic problems},
  journal = {Computer Methods in Applied Mechanics and Engineering},
  volume  = {82},
  number  = {1},
  pages   = {281-300},
  year    = {1990},
  issn    = {0045-7825},
  doi     = {10.1016/0045-7825(90)90168-L},
  url     = {https://www.sciencedirect.com/science/article/pii/004578259090168L},
  author  = {Douglas N. Arnold}
}

@article{Taylor1973,
  title   = {A numerical solution of the Navier-Stokes equations using the finite element technique},
  journal = {Computers \& Fluids},
  volume  = {1},
  number  = {1},
  pages   = {73-100},
  year    = {1973},
  issn    = {0045-7930},
  doi     = {10.1016/0045-7930(73)90027-3},
  url     = {https://www.sciencedirect.com/science/article/pii/0045793073900273},
  author  = {C. Taylor and P. Hood}
}

@article{Crouzeix1973,
  author  = {Crouzeix, Michel and Raviart, Pierre-Arnaud},
  title   = {Conforming and nonconforming finite element methods for solving the stationary Stokes equations I},
  doi     = {10.1051/m2an/197307R300331},
  url     = {https://doi.org/10.1051/m2an/197307R300331},
  journal = {Revue fran{\c{c}}aise d'automatique informatique recherche op{\'e}rationnelle. Math{\'e}matique},
  year    = 1973,
  volume  = 7,
  number  = R3,
  pages   = {33-75}
}

@article{Arnold1984,
  title   = {A stable finite element for the stokes equations},
  author  = {Arnold, D. N. and Brezzi, F. and Fortin, M.},
  doi     = {10.1007/BF02576171},
  journal = {Calcolo},
  volume  = 21,
  number  = 4,
  pages   = {337--344},
  month   = dec,
  year    = 1984
}

@article{Auricchio2013,
  title   = {Approximation of incompressible large deformation elastic
             problems: some unresolved issues},
  author  = {Auricchio, Ferdinando and Beir{\~a}o da Veiga, Louren{\c c}o and Lovadina, Carlo and Reali, Alessandro and Taylor, Robert L and Wriggers, Peter},
  journal = {Computational Mechanics},
  doi     = {10.1007/s00466-013-0869-0},
  volume  = 52,
  number  = 5,
  pages   = {1153--1167},
  month   = nov,
  year    = 2013
}

@article{Hu1954,
  title   = {On some variational principles in the theory of elasticity and the theory of plasticity},
  author  = {{Hu, Hai-Chang}},
  journal = {Acta Physica Sinica},
  volume  = 10,
  number  = 3,
  pages   = {259--290},
  year    = 1954,
  doi     = {10.7498/aps.10.259}
}

@book{Washizu1955,
  title     = {On the variational principles of elasticity and plasticity},
  author    = {Washizu, K.},
  series    = {ASRL TR},
  url       = {https://books.google.com.au/books?id=NgsWOAAACAAJ},
  year      = {1955},
  publisher = {M.I.T. Aeroelastic and Structures Research Laboratory}
}

@article{Simo1985,
  title   = {Variational and projection methods for the volume constraint in finite deformation elasto-plasticity},
  author  = {Juan C. Simo and Robert L. Taylor and Kristofer S. Pister},
  journal = {Computer Methods in Applied Mechanics and Engineering},
  volume  = {51},
  number  = {1},
  pages   = {177-208},
  year    = {1985},
  issn    = {0045-7825},
  doi     = {10.1016/0045-7825(85)90033-7},
  url     = {https://www.sciencedirect.com/science/article/pii/0045782585900337}
}

@article{Simo1992,
  author  = {Simo, Juan C. and Armero, Francisco},
  title   = {Geometrically non-linear enhanced strain mixed methods and the method of incompatible modes},
  journal = {International Journal for Numerical Methods in Engineering},
  volume  = {33},
  number  = {7},
  pages   = {1413-1449},
  doi     = {10.1002/nme.1620330705},
  url     = {https://onlinelibrary.wiley.com/doi/abs/10.1002/nme.1620330705},
  eprint  = {https://onlinelibrary.wiley.com/doi/pdf/10.1002/nme.1620330705},
  year    = {1992}
}

@article{Angoshtari2016,
  title   = {Hilbert complexes of nonlinear elasticity},
  author  = {Angoshtari, Arzhang and Yavari, Arash},
  journal = {Zeitschrift f{\"u}r angewandte Mathematik und Physik},
  volume  = 67,
  doi     = {10.1007/s00033-016-0735-y},
  number  = 6,
  pages   = {143},
  month   = nov,
  year    = 2016
}

@article{Gopalakrishnan2019,
  author  = {Gopalakrishnan, Jay and Lederer, Philip L and Sch{\"o}berl, Joachim},
  title   = {A mass conserving mixed stress formulation for the Stokes equations},
  journal = {IMA Journal of Numerical Analysis},
  volume  = {40},
  number  = {3},
  pages   = {1838-1874},
  year    = {2019},
  month   = {05},
  issn    = {0272-4979},
  doi     = {10.1093/imanum/drz022},
  url     = {https://doi.org/10.1093/imanum/drz022},
  eprint  = {https://academic.oup.com/imajna/article-pdf/40/3/1838/33489419/drz022.pdf}
}

@article{Lamichhane2009,
  author   = {Lamichhane, Bishnu P.},
  title    = {A mixed finite element method for non-linear and nearly incompressible elasticity based on biorthogonal systems},
  journal  = {International Journal for Numerical Methods in Engineering},
  volume   = {79},
  number   = {7},
  pages    = {870-886},
  doi      = {10.1002/nme.2594},
  url      = {https://onlinelibrary.wiley.com/doi/abs/10.1002/nme.2594},
  eprint   = {https://onlinelibrary.wiley.com/doi/pdf/10.1002/nme.2594},
  year     = {2009}
}

@article{Liu2014,
  title    = {A nonlinear finite element model for the stress analysis of soft solids with a growing mass},
  journal  = {International Journal of Solids and Structures},
  volume   = {51},
  number   = {17},
  pages    = {2964-2978},
  year     = {2014},
  issn     = {0020-7683},
  doi      = {10.1016/j.ijsolstr.2014.04.010},
  url      = {https://www.sciencedirect.com/science/article/pii/S0020768314001590},
  author   = {Yin Liu and Hongwu Zhang and Yonggang Zheng and Sheng Zhang and Biaosong Chen}
}

@book{Holzapfel2001,
  title     = {Nonlinear Solid Mechanics. A Continuum Approach for Engineering},
  author    = {Gerhard A. Holzapfel},
  year      = {2001},
  language  = {English},
  publisher = {John Wiley \& Sons},
  edition   = {second print}
}

@book{Brezzi1991,
  author    = {Brezzi, Franco and Fortin, Michel},
  editor    = {Brezzi, Franco and Fortin, Michel},
  title     = {Mixed and Hybrid Finite Element Methods},
  year      = {1991},
  publisher = {Springer New York},
  address   = {New York, NY},
  isbn      = {978-1-4612-3172-1},
  doi       = {10.1007/978-1-4612-3172-1},
  url       = {https://doi.org/10.1007/978-1-4612-3172-1}
}

@article{Gatica2003,
  author  = {Gatica, Gabriel N. and Heuer, Norbert and Meddahi, Salim},
  title   = {On the numerical analysis of nonlinear twofold saddle point problems},
  journal = {IMA Journal of Numerical Analysis},
  volume  = {23},
  number  = {2},
  pages   = {301-330},
  year    = {2003},
  month   = apr,
  issn    = {0272-4979},
  doi     = {10.1093/imanum/23.2.301},
  url     = {https://doi.org/10.1093/imanum/23.2.301},
  eprint  = {https://academic.oup.com/imajna/article-pdf/23/2/301/1997412/230301.pdf}
}

@article{Howell2011,
  author  = {Howell, Jason S. and Walkington, Noel J.},
  title   = {Inf--sup conditions for twofold saddle point problems},
  journal = {Numerische Mathematik},
  year    = {2011},
  month   = aug,
  day     = {01},
  volume  = {118},
  number  = {4},
  pages   = {663-693},
  issn    = {0945-3245},
  doi     = {10.1007/s00211-011-0372-5},
  url     = {https://doi.org/10.1007/s00211-011-0372-5}
}

@article{Gatica2004,
  title   = {A low-order mixed finite element method for a class of quasi-Newtonian Stokes flows. Part I: a priori error analysis},
  journal = {Computer Methods in Applied Mechanics and Engineering},
  volume  = {193},
  number  = {9},
  pages   = {881-892},
  year    = 2004,
  issn    = {0045-7825},
  doi     = {10.1016/j.cma.2003.11.007},
  url     = {https://www.sciencedirect.com/science/article/pii/S0045782503005929},
  author  = {Gabriel N. Gatica and Mar{\'i}a Gonz{\'a}lez and Salim Meddahi}
}

@article{Badia2018,
  title    = {The aggregated unfitted finite element method for elliptic problems},
  journal  = {Computer Methods in Applied Mechanics and Engineering},
  volume   = {336},
  pages    = {533-553},
  year     = {2018},
  issn     = {0045-7825},
  doi      = {10.1016/j.cma.2018.03.022},
  url      = {https://www.sciencedirect.com/science/article/pii/S0045782518301476},
  author   = {Santiago Badia and Francesc Verdugo and Alberto F. Mart\'{\i}n}
}

@article{Badia2024,
  title    = {Finite element interpolated neural networks for solving forward and inverse problems},
  journal  = {Computer Methods in Applied Mechanics and Engineering},
  volume   = {418},
  pages    = {116505},
  year     = {2024},
  issn     = {0045-7825},
  doi      = {10.1016/j.cma.2023.116505},
  url      = {https://www.sciencedirect.com/science/article/pii/S0045782523006291},
  author   = {Santiago Badia and Wei Li and Alberto F. Mart\'{\i}n}
}

@article{Badia2025,
  title    = {Compatible finite element interpolated neural networks},
  journal  = {Computer Methods in Applied Mechanics and Engineering},
  volume   = {439},
  pages    = {117889},
  year     = {2025},
  issn     = {0045-7825},
  doi      = {10.1016/j.cma.2025.117889},
  url      = {https://www.sciencedirect.com/science/article/pii/S0045782525001616},
  author   = {Santiago Badia and Wei Li and Alberto F. Mart\'{\i}n}
}

\end{document}